\documentclass{article}
\usepackage{alphabeta,amsfonts,amsmath,amssymb,amsthm,authblk}
\usepackage[english]{babel}
\usepackage{booktabs}
\usepackage[makeroom]{cancel}
\usepackage{comment}
\usepackage{enumerate}
\usepackage[shortlabels]{enumitem}
\usepackage[margin = 1 in]{geometry}
\usepackage{graphicx}
\usepackage[usenames,dvipsnames]{xcolor}
\usepackage[linktocpage,colorlinks=true,linkcolor=Fuchsia!70!Magenta,urlcolor=blue!60!cyan,citecolor=Green!70!PineGreen,
	hypertexnames=false]{hyperref}
\usepackage[utf8]{inputenc}
\usepackage{leftidx}
\usepackage{mathrsfs,mathtools}
\usepackage{relsize}
\usepackage{scalerel,stackengine}
\usepackage{tikz,tikz-cd}
\usepackage[titles]{tocloft}
\usepackage{xspace}

\newcommand{\N}{\mathbb{N}}

\newcommand{\R}{\mathbb{R}}
\newcommand{\C}{\mathbb{C}}

\newcommand{\sA}{\mathscr{A}}

\newcommand{\cB}{\mathcal{B}}
\newcommand{\sB}{\mathscr{B}}

\newcommand{\cC}{\mathcal{C}}

\newcommand{\sF}{\mathscr{F}}
\newcommand{\sG}{\mathscr{G}}
\newcommand{\sH}{\mathscr{H}}

\newcommand{\cK}{\mathcal{K}}
\newcommand{\cM}{\mathcal{M}}
\newcommand{\cN}{\mathcal{N}}
\newcommand{\cS}{\mathcal{S}}
\newcommand{\cT}{\mathcal{T}}

\newcommand{\e}{\varepsilon}
\newcommand{\Om}{\Omega} 
\newcommand{\om}{\omega}

\DeclareMathOperator{\ev}{ev}
\DeclareMathOperator{\Hom}{Hom}

\DeclareMathOperator{\Tr}{Tr}

\DeclareMathOperator{\spn}{span}
\DeclareMathOperator{\id}{id}
\DeclareMathOperator{\im}{im}

\newcommand{\la}{\langle}
\newcommand{\ra}{\rangle}

\newcommand{\sh}{\text{\smaller $\#$}}

\newcommand{\vertiii}[1]{{\left\vert\kern-0.25ex\left\vert\kern-0.25ex\left\vert #1 
    \right\vert\kern-0.25ex\right\vert\kern-0.25ex\right\vert}}

\DeclareMathOperator*{\esssup}{ess\,sup}

\newcommand{\hotimes}{\otimes_2}

\newcommand{\potimes}{\hat{\otimes}_{\pi}}
\newcommand{\wotimes}{\bar{\otimes}}
\newcommand{\iotimes}{\hat{\otimes}_i}

\stackMath
\newcommand\wch[1]{%
\savestack{\tmpbox}{\stretchto{%
  \scaleto{%
    \scalerel*[\widthof{\ensuremath{#1}}]{\kern-.6pt\bigwedge\kern-.6pt}%
    {\rule[-\textheight/2]{1ex}{\textheight}}
  }{\textheight}%
}{0.5ex}}%
\stackon[1pt]{#1}{\scalebox{-1}{\tmpbox}}%
}

\numberwithin{equation}{section}

\newcommand\numberthis{\addtocounter{equation}{1}\tag{\theequation}}
\makeatletter
\newcommand{\oset}[3][0ex]{%
  \mathrel{\mathop{#3}\limits^{
    \vbox to#1{\kern-2\ex@
    \hbox{$\scriptstyle#2$}\vss}}}}
\makeatother

\theoremstyle{plain}
\newtheorem{prop}{Proposition}[subsection]
\theoremstyle{plain}

\theoremstyle{plain}
\newtheorem{lem}[prop]{Lemma}
\theoremstyle{plain}

\theoremstyle{plain}
\newtheorem{thm}[prop]{Theorem}
\theoremstyle{plain}

\theoremstyle{plain}
\newtheorem{cor}[prop]{Corollary}
\theoremstyle{plain}

\theoremstyle{plain}

\theoremstyle{plain}

\theoremstyle{definition}
\newtheorem{defi}[prop]{Definition}
\theoremstyle{definition}

\theoremstyle{definition}
\newtheorem{nota}[prop]{Notation}
\theoremstyle{definition}

\theoremstyle{definition}
\newtheorem{conv}[prop]{Convention}
\theoremstyle{definition}

\theoremstyle{definition}
\newtheorem{ex}[prop]{Example}
\theoremstyle{definition}

\theoremstyle{definition}

\theoremstyle{definition}

\theoremstyle{definition}
\newtheorem{rem}[prop]{Remark}
\theoremstyle{definition}

\theoremstyle{definition}

\theoremstyle{definition}

\theoremstyle{definition}
\newtheorem*{ack}{Acknowledgements}

\makeatletter
\renewenvironment{proof}[1][\proofname]{%
  \par\pushQED{\qed}\normalfont%
  \topsep6\p@\@plus6\p@\relax
  \trivlist\item[\hskip\labelsep\bfseries#1\@addpunct{.}]%
  \ignorespaces
}{%
  \popQED\endtrivlist\@endpefalse
}
\makeatother
\begin{document}

\title{Multiple operator integrals in \\
non-separable von Neumann algebras}
\author{Evangelos A. Nikitopoulos\thanks{Supported by NSF grant DGE 2038238 and partially supported by NSF grants DMS 1253402 and DMS 1800733}}
\affil{Department of Mathematics, University of California San Diego\protect\\
\noindent 9500 Gilman Drive, La Jolla, CA 92093-0112 (USA)\protect\\
Email: {\tt \href{mailto:enikitop@ucsd.edu}{enikitop@ucsd.edu}}}
\date{\vspace{-5ex}}

\maketitle

\begin{abstract}
A multiple operator integral (MOI) is an indispensable tool in several branches of noncommutative analysis.
However, there are substantial technical issues with the existing literature on the ``separation of variables" approach to defining MOIs, especially when the underlying Hilbert spaces are not separable.
In this paper, we provide a detailed development of this approach in a very general setting that resolves existing technical issues.
Along the way, we characterize several kinds of ``weak" operator valued integrals in terms of easily checkable conditions and prove a useful Minkowski-type integral inequality for maps with values in a semifinite von Neumann algebra.

$\,$


$\,$

\end{abstract}
\tableofcontents
\clearpage

\section{Introduction}\label{sec.intro}

Let $(\Om,\sF)$ be a measurable space and $H$ be a complex Hilbert space.
We write $B(H)$ for the space of bounded linear maps on $H$.
A function $P \colon \sF \to B(H)$ is called a \textbf{projection valued measure} if $P(\Om)=1 = \id_H$, $P(G)^2= P(G) = P(G)^*$ whenever $G \in \sF$, and $P$ is countably additive in the weak operator topology.
When $P \colon \sF \to B(H)$ is a projection valued measure, the quadruple $(\Om,\sF,H,P)$ is called a \textbf{projection valued measure space}. Fix such a space $(\Om,\sF,H,P)$.
If $h,k \in H$, then the assignment $\sF \ni G \mapsto P_{h,k}(G) \coloneqq \la P(G)h,k \ra \in \C$ is a complex measure with total variation norm at most $\|h\|\,\|k\|$.
Now, if $\varphi \colon \Om \to \C$ is bounded and measurable, then there exists unique $P(\varphi) = \int_{\Om} \varphi\,dP = \int_{\Om} \varphi(\om)\,P(d\om)$ in $B(H)$ such that $\la P(\varphi)h,k \ra = \int_{\Om} \varphi \,dP_{h,k}$, for all $h,k \in H$.
Unsurprisingly, $P(\varphi)$ is called the \textbf{integral} of $\varphi$ with respect to $P$.
It has the property that the assignment $\varphi \mapsto P(\varphi)$ is an algebra homomorphism that converts complex conjugation into the adjoint operation, i.e., it is a $\ast$-homomorphism from the space of bounded measurable functions $\Om \to \C$ to $B(H)$.
(Please see Chapter 5 of \cite{birmansolomyakBook} or Sections IX.1 and X.4 of \cite{conwayfunc} for projection valued measure and integration theory.)

We have just described the standard way to integrate \textit{scalar valued} functions with respect to $P$.
However, there are instances where it seems necessary to define some notion of $\int_{\Om} \Phi \,dP$ for \textit{operator valued} functions $\Phi \colon \Om \to B(H)$.
For example, when one studies Lipschitz/differentiability properties of scalar functions of operators \cite{aleksandrovOL,azamovetal,daletskiikrein,depagtersukochev,lemerdyskripka,nikitopoulosDiff,pellerMOIOpDer,pellerMOIPert} or spectral shift \cite{aleksandrovSmTrForm,dykemaskripka,potapovetal,skripkaTaylor}, one must consider integrals of the form
\[
\int_{\Om_{k+1}}\cdots\int_{\Om_1} \varphi(\om_1,\ldots,\om_{k+1}) \, P_1(d\om_1) \, b_1 \cdots P_k(d\om_k) \, b_k \,P_{k+1}(d\om_{k+1}), \numberthis\label{eq.formalMOI}
\]
where $(\Om_j,\sF_j,H,P_j)$ is a projection valued measure space, $\varphi \colon \Om_1 \times \cdots \times \Om_{k+1} \to \C$ is a scalar function, and $b_1, \ldots, b_k$ are bounded operators on $H$.
Note that the innermost integral $\int_{\Om_1}\varphi(\cdot,\om_2,\ldots,\om_{k+1})\,dP_1$ makes sense using the standard theory described in the previous paragraph, but it is already unclear how to integrate the map $\om_2 \mapsto \int_{\Om_1}\varphi(\cdot,\om_2,\ldots,\om_{k+1})\,dP_1 \,b_1$ with respect to $P_2$.
Yu. L. Daletskii and S. G. Krein made the first attempt at doing so in their seminal paper \cite{daletskiikrein}, wherein they used a Riemann--Stieltjes-type construction to define $\int_s^t \Phi(r)\,P(dr)$ for certain Borel projection valued measures on compact intervals $[s,t] \subseteq \R$ and maps $\Phi \colon [s,t] \to B(H)$.
This approach, which requires rather stringent regularity assumptions on $\Phi$, allows one to make sense of \eqref{eq.formalMOI} as an iterated integral for certain (highly regular) $\varphi$.

In general, an object that gives a rigorous meaning to \eqref{eq.formalMOI} is called a \textbf{multiple operator integral} (MOI).
Under the assumption that $H$ is separable, these have been studied and applied to various branches of noncommutative analysis extensively.
Please see A. Skripka and A. Tomskova's book \cite{skripka} for a comprehensive and well organized survey of the MOI literature and its applications.
In this paper, we shall concern ourselves with the ``separation of variables" approach to defining MOIs that is useful for differentiating operator functions (e.g., \cite{azamovetal,nikitopoulosDiff,pellerMOIOpDer}).
Loosely speaking, this means that one assumes $\varphi$ admits a decomposition
\[
\varphi(\om_1,\ldots,\om_{k+1}) = \int_{\Sigma} \varphi_1(\om_1,\sigma)\cdots \varphi_{k+1}(\om_{k+1},\sigma)\,\rho(d\sigma), \numberthis\label{eq.formalIPD}
\]
where $(\Sigma,\sH,\rho)$ is a measure space and $\varphi_j \colon \Om_j \times \Sigma \to \C$ is a (product) measurable function, and then one defines \eqref{eq.formalMOI} to be the ``weak" operator valued intergal
\[
\int_{\Sigma}P_1(\varphi_1(\cdot,\sigma))\,b_1\cdots  P_k(\varphi_k(\cdot,\sigma))\,b_k \,P_{k+1}( \varphi_{k+1}(\cdot,\sigma))\,\rho(d\sigma). \numberthis\label{eq.formalopint}
\]
When taking this approach, there are at least three questions to be addressed.
\begin{enumerate}[label=(Q.\arabic*), leftmargin=2\parindent]
    \item Exactly which decompositions \eqref{eq.formalIPD} does one allow?\label{q.IPD}
    \item Exactly what kind of operator valued integral is \eqref{eq.formalopint}?\label{q.opint}
    \item Assuming\hspace{-0.25mm} satisfactory \hspace{-0.25mm}answers\hspace{-0.25mm} to \hspace{-0.25mm}\ref{q.IPD}\hspace{-0.25mm} and \hspace{-0.25mm}\ref{q.opint},\hspace{-0.25mm} does \hspace{-0.25mm}\eqref{eq.formalopint}\hspace{-0.25mm} depend \hspace{-0.25mm}on\hspace{-0.25mm} the \hspace{-0.25mm}chosen\hspace{-0.25mm} decomposition \hspace{-0.25mm}\eqref{eq.formalIPD}?\label{q.welldef}
\end{enumerate}
There are various answers to these questions available in the literature, but existing answers are inadequate to cover the case when $H$ is not separable, and some of them have issues even when $H$ is separable.
(Please see, for instance, the comments in Section 4.6 of \cite{doddssub2} and Section 3.4 of \cite{nikitopoulosDiff}.)
In this paper, we provide detailed, very general answers to all three of the questions above without assuming that $H$ is separable.
\begin{enumerate}[label=(A.\arabic*), leftmargin=2\parindent]
    \item We consider \textit{integral projective decompositions} (Definition \ref{def.IPTPspecial}) of $\varphi$.
    In other words, we take $\varphi$ in the so called \textit{integral projective tensor product} $L^{\infty}(\Om_1,P_1) \iotimes \cdots \iotimes L^{\infty}(\Om_{k+1},P_{k+1})$, the idea for which is due to V. V. Peller \cite{pellerMOIOpDer}.
    There are substantial ``measurability issues," discussed in Remark \ref{rem.measissue}, with existing definitions of this object.
    We resolve these in Section \ref{sec.IPTP}.\label{a.IPD}
    \item Let $K$ be a complex Hilbert space, $B(H;K)$ be the space of bounded linear maps $H \to K$, and $V \subseteq B(H;K)$ be a linear subspace.
    In Theorem \ref{thm.wintBHK}, we characterize weak integrability (Definition \ref{def.GPint}) of maps $\Sigma \to V$ in the weak, strong, strong$^*$, $\sigma$-weak, $\sigma$-strong, and $\sigma$-strong$^*$ operator topologies (Section \ref{sec.optop}) on $V$.
    As an application of this (independently interesting) characterization, we prove in Section \ref{sec.welldef} that if $V=\cM \subseteq B(H)$ is a von Neumann algebra and $P_j(G) \in \cM$, for all $j \in \{1,\ldots,k+1\}$ and $G \in \sF_j$, then the integrand in \eqref{eq.formalopint} is weakly integrable in the $\sigma$-weak operator topology (also called the weak$^*$ topology) on $\cM$ whenever $b_1,\ldots,b_k \in \cM$.\label{a.opint}
    \item The independence of \eqref{eq.formalopint} of the chosen integral projective decomposition \eqref{eq.formalIPD} of $\varphi$ is highly nontrivial and has not yet been proven for non-separable $H$.
    In Section \ref{sec.welldef}, we present a robust new argument that proves this fact for general $H$.
    The two key ingredients to the argument, which we discuss in Section \ref{sec.keying}, are the multiplicative system theorem (Theorem \ref{thm.MST}, a basic fact from measure theory) and a Minkowski-type integral inequality (Theorem \ref{thm.Spinteg}) for Schatten $p$-norms of operator valued integrals that seems to be new in the non-separable case and is of independent interest.\label{a.welldef}
\end{enumerate}
We also prove in Section \ref{sec.otherdefs} that the above described approach to defining \eqref{eq.formalMOI} agrees with another commonly used approach, due to B. S. Pavlov \cite{pavlov}, when both apply. 
Finally, even with all of \ref{q.IPD}--\ref{q.welldef} answered, applications often demand answers to an additional question.
\begin{enumerate}[label=(Q.\arabic*), leftmargin=2\parindent]
\setcounter{enumi}{3}
    \item What kinds of quantitative norm estimates for \eqref{eq.formalMOI} are available?\label{q.estim}
\end{enumerate}
Our methods give us some answers to this question as well.
\begin{enumerate}[label=(A.\arabic*), leftmargin=2\parindent]
\setcounter{enumi}{3}
    \item The aforementioned Minkowski-type integral inequality has a generalization (Theorem \ref{thm.Lpinteg}) to noncommutative $L^p$-norms of semifinite von Neumann algebras that allows us to prove noncommutative $L^p$-norm estimates (Proposition \ref{prop.MOILpestim}) for \eqref{eq.formalMOI}.
\end{enumerate}
Actually, Theorem \ref{thm.Lpinteg} can be combined with the theory of symmetric operator spaces to give a \textit{much} more general answer to \ref{q.estim}. We carry this out in \cite{nikitopoulosDiff} and use it to prove new results about higher derivatives of operator functions in ideals of von Neumann algebras.

\subsection{Main results on multiple operator integrals}

In this section, we summarize our main results on MOIs.

\begin{nota}
Let $X$ be a topological space, $S$ and $T$ be sets, $(\Om,\sF)$ be a measurable space, $H$ be a complex Hilbert space, and $P \colon \sF \to B(H)$ be a projection valued measure on $(\Om,\sF)$. \begin{enumerate}[label=(\roman*),font=\normalfont,leftmargin=2\parindent]
    \item $\cB_X$ is the Borel $\sigma$-algebra on $X$.
    \item $T^S$ is the set of functions $S \to T$.
    If $\varphi \in \C^S$, then $\|\varphi\|_{\ell^{\infty}(S)} \coloneqq \sup_{s \in S}|\varphi(s)| \in [0,\infty]$.
    Also, we write $\ell^{\infty}(S)$ for the space of bounded functions $S \to \C$ and $\ell^{\infty}(\Om,\sF) \coloneqq \{\varphi \in \ell^{\infty}(\Om) : \varphi$ is $(\sF,\cB_{\C})$-measurable$\}$.
    \item $L^{\infty}(\Om,P) \coloneqq \ell^{\infty}(\Om,\sF)/{\sim_P}$, where $\sim_P$ denotes the $P$-a.e. equivalence relation.
    With the norm $\|\varphi\|_{L^{\infty}(P)} \coloneqq P\text{-}\esssup|\varphi| = \inf\{c \geq 0 : P(\{\om \in \Om : |\varphi(\om)| > c\}) = 0\}$, this space is a $C^*$-algebra under pointwise $P$-almost everywhere addition, multiplication, and complex conjugation.
\end{enumerate}
\end{nota}

If $\varphi \in \ell^{\infty}(\Om,\sF)$, then the operator norm of $P(\varphi) \in B(H)$ is equal to $\|\varphi\|_{L^{\infty}(P)}$.
Thus $P$-integration descends to a well-defined $\ast$-homomorphism $L^{\infty}(\Om,P) \to B(H)$, which we notate the same way.

For reasons explained in Remarks \ref{rem.normmeas} and \ref{rem.measissue}, we shall be forced to integrate non-measurable functions.
For this purpose, we use upper (and lower) integrals.
If $(\Sigma,\sH,\rho)$ is a measure space and $h \in [0,\infty]^{\Sigma}$, then
\[
\overline{\int_{\Sigma}} h(\sigma) \, \rho(d\sigma) = \overline{\int_{\Sigma}} h \, d\rho \coloneqq \inf\Bigg\{\int_{\Sigma} \tilde{h} \, d\rho : h \leq \tilde{h} \;\; \rho\text{-a.e.}, \; \tilde{h} \colon \Sigma \to [0,\infty] \text{ measurable}\Bigg\}
\]
is the \textbf{upper integral} of $h$.
Of course, if $h$ is measurable, then $\overline{\int_{\Sigma}} h \, d\rho = \int_{\Sigma} h \,d\rho$.
In Section \ref{sec.nonmeasint}, we prove the properties of this upper integral (and its lower counterpart) that are needed in this paper.

Next, we state the precise definition of $L^{\infty}(\Om_1,P_1) \iotimes \cdots \iotimes L^{\infty}(\Om_{k+1},P_{k+1})$.
To do so, we need the notion, due to M. S. Birman and M. Z. Solomyak \cite{birmansolomyakTensProd}, of the tensor product of projection valued measures.
(Please see Section \ref{sec.tensprod} for background about the Hilbert space tensor product $\hotimes$.)

\begin{thm}[Birman--Solomyak]\label{thm.tensprodPVM}
Let $(\Om_1,\sF_1,H_1,P_1),\ldots,(\Om_{k+1},\sF_{k+1},H_{k+1},P_{k+1})$ be projection valued measure spaces, and write $(\Om,\sF) \coloneqq (\Om_1 \times \cdots \times \Om_{k+1},\sF_1 \otimes \cdots \otimes \sF_{k+1})$. 
There exists unique projection valued measure $P \colon \sF \to B(H_1 \hotimes \cdots \hotimes H_{k+1})$ such that for all $G_1 \in \sF_1,\ldots,G_{k+1} \in \sF_{k+1}$,
\[
P(G_1 \times \cdots \times G_{k+1}) = P_1(G_1) \otimes \cdots \otimes P_{k+1}(G_{k+1}).
\]
We call $P$ the \textbf{tensor product} of $P_1,\ldots,P_{k+1}$ and write $P_1 \otimes \cdots \otimes P_{k+1} = P$.
\end{thm}

For completeness, we supply a proof in Section \ref{sec.tensprodPVM}.
Now, retain the setup of Theorem \ref{thm.tensprodPVM}, write $P \coloneqq P_1 \otimes \cdots \otimes P_{k+1}$, and let $\varphi \colon \Om \to \C$ be a function.
A \textbf{$\boldsymbol{L_P^{\infty}}$-integral projective decomposition} of $\varphi$ is a choice $(\Sigma,\rho,\varphi_1,\ldots,\varphi_{k+1})$ of a $\sigma$-finite measure space $(\Sigma,\sH,\rho)$ and measurable functions $\varphi_j \colon \Om_j \times \Sigma \to \C$ such that $\varphi_j(\cdot,\sigma) \in L^{\infty}(\Om_j,P_j)$ for $\sigma \in \Sigma$,
\[
\overline{\int_{\Sigma}} \,\|\varphi_1(\cdot,\sigma)\|_{L^{\infty}(P_1)} \cdots \|\varphi_{k+1}(\cdot,\sigma)\|_{L^{\infty}(P_{k+1})} \, \rho(d\sigma) < \infty, \numberthis\label{eq.IPDint}
\]
and \eqref{eq.formalIPD} holds for $P$-almost every $(\om_1,\ldots,\om_{k+1}) \in \Om$.
(The integral on the right hand side of \eqref{eq.formalIPD} makes sense $P$-almost everywhere by Lemma \ref{lem.PVMink}.) 
Now, we define $\|\varphi\|_{L^{\infty}(P_1) \iotimes \cdots \iotimes L^{\infty}(P_{k+1})}$ to be the infimum of the set of numbers \eqref{eq.IPDint} as $(\Sigma,\rho,\varphi_1,\ldots,\varphi_{k+1})$ ranges over all $L_P^{\infty}$-integral projective decompositions of $\varphi$.
In Section \ref{sec.IPTP}, we prove that if $L^{\infty}(\Om_1,P_1) \iotimes \cdots \iotimes L^{\infty}(\Om_{k+1},P_{k+1})$ is the space of $P$-a.e. equivalence classes of functions $\varphi$ admitting $L_P^{\infty}$-integral projective decompositions, then this space is a Banach $\ast$-algebra under $P$-almost everywhere operations and the norm $\|\cdot\|_{L^{\infty}(P_1) \iotimes \cdots \iotimes L^{\infty}(P_{k+1})}$.

Now, let $V$ be a Hausdorff locally convex topological vector space with topological dual $V^*$.
We call $F \colon \Sigma \to V$ \textbf{weakly measurable} if $\ell \circ F \colon \Sigma \to \C$ is $(\sH,\cB_{\C})$-measurable whenever $\ell \in V^*$.
Suppose in addition that $\int_{\Sigma}|\ell \circ F|\,d\rho < \infty$, for all $\ell \in V^*$.
If for all $S \in \sH$, there exists (necessarily unique) $\int_S F(\sigma)\,\rho(d\sigma) \in V$ such that $\ell\big(\int_S F(\sigma)\,\rho(d\sigma) \big) = \int_S (\ell \circ F) \, d\rho$ whenever $\ell \in V^*$, then $F$ is called \textbf{weakly} or \textbf{Gel'fand--Pettis integrable}.
In this case, $\int_S F(\sigma)\,\rho(d\sigma)$ is called the \textbf{weak} or \textbf{Gel'fand--Pettis integral} (over $S$) of $F$ with respect to $\rho$.

Finally, \hspace{-0.45mm}recall\hspace{-0.45mm} that \hspace{-0.45mm}a\hspace{-0.45mm} \textbf{von \hspace{-0.45mm}Neumann\hspace{-0.45mm} algebra} \hspace{-0.45mm}is\hspace{-0.45mm} a \hspace{-0.45mm}weak\hspace{-0.45mm} operator \hspace{-0.45mm}topology\hspace{-0.45mm} closed \hspace{-0.45mm}unital\hspace{-0.45mm} $\ast$-subalgebra \hspace{-0.45mm}of\hspace{-0.45mm} $B(H)$.

\begin{thm}[Well-definition of MOIs]\label{thm.MOI}
Let $H$ be a complex Hilbert space and $\cM \subseteq B(H)$ be a von Neumann algebra.
Suppose, for $j \in \{1,\ldots,k+1\}$, that $(\Om_j,\sF_j,H,P_j)$ is a projection valued measure space such that $P_j(G) \in \cM$ whenever $G \in \sF_j$.
If $(\Sigma,\rho,\varphi_1,\ldots,\varphi_{k+1})$ is a $L_P^{\infty}$-integral projective decomposition of a function $\varphi\in L^{\infty}(\Om_1,P_1) \iotimes \cdots \iotimes L^{\infty}(\Om_{k+1},P_{k+1})$ and $b = (b_1,\ldots,b_k) \in \cM^k$, then the map
\[
\Sigma \ni \sigma \mapsto P_1(\varphi_1(\cdot,\sigma))\,b_1 \cdots P_k(\varphi_k(\cdot,\sigma))\,b_k \,P_{k+1}(\varphi_{k+1}(\cdot,\sigma)) \in \cM
\]
is Gel'fand--Pettis integrable in the $\sigma$-weak operator topology on $\cM$, and the weak integral
\[
\big(I^{P_1,\ldots,P_{k+1}}\varphi\big)[b] \coloneqq \int_{\Sigma}P_1(\varphi_1(\cdot,\sigma))\,b_1 \cdots P_k(\varphi_k(\cdot,\sigma))\,b_k \,P_{k+1}(\varphi_{k+1}(\cdot,\sigma))\,\rho(d\sigma) \in \cM
\]
is independent of the chosen decomposition $(\Sigma,\rho,\varphi_1,\ldots,\varphi_{k+1})$ and the representation of $\cM$.
\end{thm}
\begin{proof}
Combine Corollary \ref{cor.MOIintgood}, Theorem \ref{thm.MOIwelldef}, and Theorem \ref{thm.MOIsinM}.
\end{proof}

We also prove in Proposition \ref{prop.linandmult} that the assignment $\varphi \mapsto I^{P_1,\ldots,P_{k+1}}\varphi$ is linear and multiplicative in a certain sense.
Finally, when $(\cM,\tau)$ is a semifinite von Neumann algebra (Definition \ref{def.trace}), we also prove (Proposition \ref{prop.MOILpestim}) that if $p,p_1,\ldots,p_k \in [1,\infty]$ are such that $\frac{1}{p_1}+\cdots +\frac{1}{p_k} = \frac{1}{p}$, then
\[
\big\|\big(I^{P_1,\ldots,P_{k+1}}\varphi\big)[b_1,\ldots,b_k]\big\|_{L^p(\tau)} \leq \|\varphi\|_{L^{\infty}(P_1) \iotimes \cdots \iotimes L^{\infty}(P_{k+1})}\|b_1\|_{L^{p_1}(\tau)}\cdots\|b_k\|_{L^{p_k}(\tau)},
\]
for all $b_1,\ldots,b_k \in \cM$, where $\|\cdot\|_{L^p(\tau)}$ is the noncommutative $L^p$-norm (Notation \ref{nota.ncLp}).
This allows for an ``extension" of the MOI $I^{P_1,\ldots,P_{k+1}}\varphi \colon \cM^k \to \cM$ to a bounded $k$-linear map $L^{p_1}(\tau) \times \cdots \times L^{p_k}(\tau) \to L^p(\tau)$.

\subsection{Discussion of the well-definition argument}\label{sec.keying}

Retain the setup of Theorem \ref{thm.MOI} with $\cM = B(H)$.
In this section, we discuss the key ingredients of the proof that the integral \eqref{eq.formalopint}, defined as described in the previous section, is independent of the chosen $L_P^{\infty}$-integral projective decomposition of $\varphi$ and why this argument is delicate when $H$ is not separable.
To maximize readability, we stick to the case of a \textit{double operator integral} (DOI), i.e., the case $k=1$.

Let $b \in B(H)$.
The goal is to show that if $(\Sigma,\rho,\varphi_1,\varphi_2)$ is a $L_{P_1 \otimes P_2}^{\infty}$-integral projective decomposition of $\varphi \in L^{\infty}(\Om_1,P_1) \iotimes L^{\infty}(\Om_2,P_2)$, then
\[
\int_{\Sigma} P_1(\varphi_1(\cdot,\sigma))\,b\,P_2(\varphi_2(\cdot,\sigma)) \, \rho(d\sigma)
\]
does not depend on $(\Sigma,\rho,\varphi_1,\varphi_2)$.
This is actually not very difficult to prove --- as in \cite{azamovetal,pellerMOIPert} --- when $b$ has finite rank, so the proof is complete if we can somehow reduce to this case.
In \cite{pellerMOIPert}, it is stated that this reduction is ``easy to see."
This is certainly not the case when $H$ is not separable.
Indeed, when $H$ is separable (as is assumed in \cite{azamovetal}), every $b \in B(H)$ is actually a strong operator limit of a \textit{sequence} of finite rank operators.
One can then use a vector valued dominated convergence theorem to finish the proof.
But this argument does not work when $H$ is not separable because, for instance, $\id_H$ is not a strong operator limit of a sequence of finite rank operators.

We opt instead to work with a different topology on $B(H)$ with respect to which finite rank operators are dense:
the ultraweak topology.
(This is another description of the $\sigma$-weak operator topology;
see Theorems \ref{thm.Schatten}\ref{item.finrk} and \ref{thm.optop}\ref{item.WOTsigmaWOT}.)
If we can show $b \mapsto I^{P_1,P_2}(\Sigma,\rho,\varphi_1,\varphi_2)[b] \coloneqq \int_{\Sigma} P_1(\varphi_1(\cdot,\sigma))\,b\,P_2(\varphi_2(\cdot,\sigma))\,\rho(d\sigma)$ is ultraweakly continuous, then the proof will be complete.
This ultraweak continuity is asserted in \cite{potapovsukochev} without proof or reference.
When $H$ is not separable, it is not at all obvious and, to the author's knowledge, has remained unproven until now.
To prove it, we must show that for all $a \in \cS_1(H)$ (trace class operators, Definition \ref{def.Schatten}), there exists $Ta \in \cS_1(H)$ such that
\[
\Tr\big(I^{P_1,P_2}(\Sigma,\rho,\varphi_1,\varphi_2)[b]\,a\big) = \Tr(b\,Ta), \; \text{ for all } b \in B(H). \numberthis\label{eq.uwcontgoal}
\]
To motivate what $Ta$ should be, fix $a,b \in \cS_1(H)$.
Then the maps $c \mapsto \Tr(ca)$ and $c \mapsto \Tr(bc)$ are $\sigma$-weakly/ultraweakly continuous.
Therefore, by definition of the Gel'fand--Pettis integral and basic properties (Theorem \ref{thm.Schatten}\ref{item.Trflip}) of $\Tr$, we have
\begin{align*}
    \Tr\big(I^{P_1,P_2}(\Sigma,\rho,\varphi_1,\varphi_2)[b]\,a\big) & = \int_{\Sigma} \underbrace{\Tr( P_1(\varphi_1(\cdot,\sigma))\,b\,P_2(\varphi_2(\cdot,\sigma))\,a)}_{=\Tr( b\,P_2(\varphi_2(\cdot,\sigma))\,a \, P_1(\varphi_1(\cdot,\sigma)))} \, \rho(d\sigma) \\
    & = \Tr\Bigg(b\int_{\Sigma} P_2(\varphi_2(\cdot,\sigma))\,a\,P_1(\varphi_1(\cdot,\sigma))\,\rho(d\sigma)\Bigg). \numberthis\label{eq.weakS1integ}
\end{align*}
We should therefore take $Ta = \int_{\Sigma} P_2(\varphi_2(\cdot,\sigma))\,a\,P_1(\varphi_1(\cdot,\sigma))\,\rho(d\sigma)$ in \eqref{eq.uwcontgoal}.
(Those familiar with the subject will recognize this as related to the Birman--Solomyak \cite{birmansolomyakDSOI1} definition of a DOI.
We elaborate on this in Section \ref{sec.otherdefs}.)
For this to have any chance at making sense, we need to know at the very least that
\[
a \in \cS_1(H) \implies \int_{\Sigma} P_2(\varphi_2(\cdot,\sigma))\,a\,P_1(\varphi_1(\cdot,\sigma)) \,\rho(d\sigma) \in \cS_1(H). \numberthis\label{eq.DOIinS1}
\]
Even \textit{this} is not obvious when $H$ is not separable!
It follows, however, from our first key ingredient:
Theorem \ref{thm.Spinteg}.
Assuming we know \eqref{eq.DOIinS1}, we must still verify \eqref{eq.weakS1integ} for \textit{all} $b \in B(H)$, not just for $b \in \cS_1(H)$.
If $b \in B(H)$ is arbitrary, then the map $\cS_1(H) \ni c \mapsto \Tr(bc) \in \C$ is bounded with respect to $\|\cdot\|_{\cS_1(H)}$.
Therefore, we could reverse the calculation that led to \eqref{eq.weakS1integ} if we knew that $\Sigma \ni \sigma \mapsto P_2(\varphi_2(\cdot,\sigma))\,a\,P_1(\varphi_1(\cdot,\sigma)) \in \cS_1(H)$ were Gel'fand--Pettis integrable as a map $\Sigma \to (\cS_1(H),\|\cdot\|_{\cS_1(H)})$, not just as a map $\Sigma \to (B(H),\sigma$-WOT$)$, whenever $a \in \cS_1(H)$. This is not automatic.
Moreover, if $H$ is not separable, then $\cS_1(H)$ is not separable, so Bochner integral techniques do not automatically apply either.
We tiptoe around these difficulties using the second key ingredient:
the multiplicative system theorem, which is a ``functional form" of the Dynkin system theorem.
To state it, we recall some additional notation and terminology.

\begin{nota}\label{nota.sigmaalggen}
Let $S$ be a set, $2^S$ be the power set of $S$, $(T,\mathscr{T})$ be a measurable space, and $\mathscr{S} \subseteq T^S$.
We write $\sigma_{\mathscr{T}}(\mathscr{S}) \subseteq 2^S$ for the smallest $\sigma$-algebra on $S$ with respect to which all members of $\mathscr{S}$ are measurable.
When we are in the case $(T,\mathscr{T}) = (\C,\cB_{\C})$, we shall write $\sigma_{\mathscr{T}}(\mathscr{S}) = \sigma(\mathscr{S})$.
\end{nota}
\pagebreak

Retain the above setting, and let $(\Xi,\sG)$ be another measurable space.
Note that a function $g \colon \Xi \to S$ is $(\sG,\sigma_{\mathscr{T}}(\mathscr{S}))$-measurable if and only if $s \circ g \colon \Xi \to T$ is $(\sG,\mathscr{T})$-measurable, for all $s \in \mathscr{S}$.
Finally, we recall the definition of \textbf{bounded convergence}.
A sequence $(\varphi_n)_{n \in \N}$ of functions $S \to \C$ is said to \textbf{converge boundedly} to $\varphi \in \C^S$ if $\varphi_n \to \varphi$ pointwise as $n \to \infty$ and $\sup_{n \in \N}\|\varphi_n\|_{\ell^{\infty}(S)} < \infty$.

\begin{thm}[Multiplicative system]\label{thm.MST}
Let $S$ be a set.
Suppose $\mathbb{H} \subseteq \C^S$ is a linear subspace containing the constant function $1$ that is closed under complex conjugation and sequential bounded convergence.
If $\mathbb{M} \subseteq \mathbb{H}$ is closed under multiplication and complex conjugation, then $\ell^{\infty}(S,\sigma(\mathbb{M})) \subseteq \mathbb{H}$.
\end{thm}

For a proof, please see Section 12.1 of \cite{driver} or Appendix A of \cite{janson}.
The corollary most relevant to the argument presently under discussion is as follows.

\begin{cor}\label{cor.MST}
Let $(\Om,\sF)$, $(\Sigma,\sH)$ be measurable spaces and $\mathbb{H}$ be a linear subspace of $\ell^{\infty}(\Om \times \Sigma)$ that is closed under complex conjugation and sequential bounded convergence. 
If $\{1_{G \times S} : G \in\sF, S \in \sH\} \subseteq \mathbb{H}$, then $\ell^{\infty}(\Om \times \Sigma,\sF \otimes \sH) \subseteq \mathbb{H}$.
\end{cor}
\begin{proof}
If $\mathbb{M} \coloneqq \{1_{G \times S} : G \in \sF, S \in \sH\}$, then $\mathbb{M}$ is closed under complex conjugation and pointwise multiplication (because $\{G \times S : G \in \sF, S \in \sH\}$ is a $\pi$-system).
Since $1 \in \mathbb{M} \subseteq \mathbb{H}$ and $\sigma(\mathbb{M}) = \sF \otimes \sH$, the conclusion follows from the multiplicative system theorem.
\end{proof}

Using this consequence of the multiplicative system theorem and our operator valued integral development, we are able to prove (in Section \ref{sec.welldef}) the following key result.

\begin{thm}[Trace of integral of a product]\label{thm.traceofinteg}
Let $(\Xi,\sG,K,Q)$ be a projection valued measure space, $(\Sigma,\sH,\rho)$ be a finite measure space, and $\varphi \in \ell^{\infty}(\Xi \times \Sigma,\sG \otimes \sH)$.
If $A \colon \Sigma \to B(K)$ is  pointwise weakly measurable (Definition \ref{def.Pettint}) and $\sup_{\sigma \in \Sigma}\|A(\sigma)\|_{\cS_1} < \infty$, then
\[
\int_{\Sigma} A(\sigma) \, Q(\varphi(\cdot,\sigma)) \,\rho(d\sigma),\; \int_{\Sigma} Q(\varphi(\cdot,\sigma)) \,A(\sigma) \,\rho(d\sigma) \in \cS_1(K)
\]
and
\[
\Tr\Bigg(\int_{\Sigma} A(\sigma) \, Q(\varphi(\cdot,\sigma)) \,\rho(d\sigma)\Bigg) = \Tr\Bigg( \int_{\Sigma} Q(\varphi(\cdot,\sigma)) \,A(\sigma) \,\rho(d\sigma)\Bigg). \numberthis\label{eq.traceofinteg}
\]
\end{thm}

By combining the formula \eqref{eq.traceofinteg} with truncation arguments on $\varphi_1$ and $\varphi_2$, we are able to prove the desired $\cS_1(H)$-valued weak integrability of $\Sigma \ni \sigma \mapsto P_2(\varphi_2(\cdot,\sigma))\,a\,P_1(\varphi_1(\cdot,\sigma)) \in \cS_1(H)$ when $a \in \cS_1(H)$.
The relevant results are Theorem \ref{thm.S1integrability} and Corollary \ref{cor.uwcont}.
Please see Remark \ref{rem.semifinitecase} as well.

\section{Background}\label{sec.bg}

For this section, fix complex Hilbert spaces $(H,\la \cdot, \cdot \ra_H)$ and $(K,\la \cdot,\cdot \ra_K)$.

\subsection{Operator topologies}\label{sec.optop}

In this section, we record facts that we shall need about some standard locally convex topologies on $B(H;K)$.
We assume the reader is quite familiar with these in the case $H=K$, which is covered in Chapter II of \cite{takesaki} and Chapter I.3 of \cite{dixmier}.
When $H \neq K$, all the basic properties still hold with essentially the same proofs.

The \textbf{weak operator topology} (WOT for short) on $B(H;K)$ is the one generated by the seminorms $B(H;K) \ni A \mapsto |\la Ah,k \ra_K| \in \R_+ \coloneqq [0,\infty)$, where $h \in H$ and $k \in K$.
The \textbf{strong operator topology} (SOT) is generated by the seminorms $B(H;K) \ni A \mapsto \|Ah\|_K \in \R_+$, where $h \in H$.
The \textbf{strong$\boldsymbol{^*}$ operator topology} (S$^*$OT) is generated by the seminorms $B(H;K) \ni A \mapsto \|Ah\|_K+\|A^*k\|_H \in \R_+$, where $h \in H$ and $k \in K$.
Next, write
\[
\ell^2(\N;H) \coloneqq \Bigg\{(h_n)_{n \in \N} \in H^{\N} : \sum_{n=1}^{\infty}\|h_n\|_H^2 < \infty\Bigg\}
\]
with inner product
\[
\la (h_n)_{n \in \N},(k_n)_{n \in \N} \ra_{\ell^2(\N;H)} \coloneqq \sum_{n=1}^{\infty} \la h_n,k_n \ra_H.\pagebreak
\]
The \textbf{$\boldsymbol{\sigma}$-weak operator topology} ($\sigma$-WOT) is generated by the seminorms
\[
B(H;K) \ni A \mapsto |\la (Ah_n)_{n \in \N},(k_n)_{n \in \N} \ra_{\ell^2(\N;K)}| \in \R_+,
\]
where $(h_n)_{n \in \N} \in \ell^2(\N;H)$ and $(k_n)_{n \in \N} \in \ell^2(\N;K)$.
The \textbf{$\boldsymbol{\sigma}$-strong operator topology} ($\sigma$-SOT) is generated by the seminorms
\[
B(H;K) \ni A \mapsto \|(Ah_n)_{n \in \N}\|_{\ell^2(\N;K)} \in \R_+,
\]
where \hspace{-0.4mm}$(h_n)_{n \in \N} \hspace{-0.4mm}\in\hspace{-0.4mm} \ell^2(\N;\hspace{-0.4mm}H)$.\hspace{-0.2mm}
Finally, \hspace{-0.4mm}the\hspace{-0.2mm} \textbf{$\boldsymbol{\sigma}$\hspace{-0.2mm}-strong$\boldsymbol{^*}$ \hspace{-0.4mm}operator\hspace{-0.4mm} topology} \hspace{-0.4mm}($\sigma$\hspace{-0.2mm}-S$^*$OT)\hspace{-0.4mm} is \hspace{-0.2mm}generated\hspace{-0.4mm} by \hspace{-0.4mm}the\hspace{-0.2mm} seminorms
\[
B(H;K) \ni A \mapsto \|(Ah_n)_{n \in \N}\|_{\ell^2(\N;K)}+\|(A^*k_n)_{n \in \N}\|_{\ell^2(\N;H)} \in \R_+,
\]
where $(h_n)_{n \in \N} \in \ell^2(\N;H)$ and $(k_n)_{n \in \N} \in \ell^2(\N;K)$.

When referring to these topologies, we shall often omit the term ``operator."
Also, if $V \subseteq B(H;K)$, then the subspace topologies inherited by $V$ from the aforementioned topologies on $B(H;K)$ are given the same names as above.
For example, the $\sigma$-weak topology ($\sigma$-WOT) on $V$ is the subspace topology $V$ inherits from the $\sigma$-WOT on $B(H;K)$.
Here are all the facts we shall need about these topologies.

\begin{thm}[Properties of operator topologies]\label{thm.optop}
Let $V \subseteq B(H;K)$ be a linear subspace, $\ell \colon V \to \C$ be a linear functional, and $\cT \in \{\mathrm{WOT}, \,\mathrm{SOT}, \,\mathrm{S^*OT}\}$.
\begin{enumerate}[label=(\roman*),font=\normalfont,leftmargin=2\parindent]
    \item $\cT$ agrees with the $\sigma$-$\cT$ on norm-bounded subsets of $B(H;K)$.
    In particular --- since the net of finite rank orthogonal projections on $K$ converges in the WOT to the identity on $K$ --- finite rank linear operators $H \to K$ are $\sigma$-weakly dense in $B(H;K)$. \label{item.WOTsigmaWOT}
    \item $\ell$ is $\cT$-continuous if and only if there exist $h_1,\ldots,h_n \in H$ and $k_1,\ldots,k_n \in K$ such that, for all $A \in V$,\label{item.contchar}
    \[
    \ell(A) = \sum_{j=1}^n\la Ah_j,k_j \ra_K.
    \]
    \item $\ell$ is $\sigma$-$\cT$-continuous if and only if there exist elements $(h_n)_{n \in \N} \in \ell^2(\N;H)$ and $(k_n)_{n \in \N} \in \ell^2(\N;K)$ such that, for all $A \in V$,\label{item.sigmacontchar}
    \[
    \ell(A) = \la (Ah_n)_{n \in \N},(k_n)_{n \in \N} \ra _{\ell^2(\N;K)} = \sum_{n=1}^{\infty}\la Ah_n,k_n \ra_K.
    \]
\end{enumerate}
Suppose now that $V \subseteq B(H;K)$ is also $\sigma$-weakly closed.
\begin{enumerate}[label=(\roman*),font=\normalfont,leftmargin=2\parindent]
\setcounter{enumi}{3}
    \item If $V_* \coloneqq \{\sigma$-weakly continuous linear functionals $V \to \C\} = (V,\sigma\text{-}\mathrm{WOT})^*$, then $V_* \subseteq V^*$ is a (norm) closed linear subspace, and the map $\ev_V \colon V \to V_*^{\;*}$ defined by $V \ni A \mapsto (\ell \mapsto \ell(A)) \in V_*^{\;*}$ is an isometric isomorphism.
    We therefore call $V_*$ the \textbf{predual} of $V$.\label{item.predual}
    \item The map $\ev_V$ from the previous part is a homeomorphism with respect to the $\sigma$-weak topology on $V$ and the weak$^*$ topology on $V_*^{\;*}$.
    The $\sigma$-weak topology on $V$ is therefore also called the \textbf{weak$\boldsymbol{^*}$ topology}.\label{item.weakstar}
\end{enumerate}
Finally, suppose $\cM \subseteq B(H)$ and $\cN \subseteq B(K)$ are von Neumann algebras.
\begin{enumerate}[label=(\roman*),font=\normalfont,leftmargin=2\parindent]
\setcounter{enumi}{5}
    \item If $\pi \colon \cM \to \cN$ is a $\ast$-isomorphism in the algebraic sense, then $\pi$ is a homeomorphism with respect to the $\sigma$-weak topologies on $\cM$ and $\cN$.\label{item.starisom}
\end{enumerate}
\end{thm}
\begin{rem}
If $H = K$, then Theorem \ref{thm.optop}\ref{item.WOTsigmaWOT}--\ref{item.weakstar} is contained in Section I.3.1 of \cite{dixmier} and Section II.2 of \cite{takesaki}.
The proofs of these statements when $H \neq K$ are slight notational modifications of the proofs in the aforementioned references.
Item \ref{item.starisom} is part of Corollary 1 in Section I.4.3 of \cite{dixmier}.
\end{rem}

\subsection{Schatten classes and noncommutative \texorpdfstring{$L^p$}{}-spaces}\label{sec.Schatten}

We now record some standard facts about Schatten $p$-class operators $H \to K$ that will be of use to us.
Please see Chapter 2 of \cite{ringrose} for the proofs of these basics (and more) in the case $H=K$.
For just the cases $p \in \{1,2,\infty\}$, please see also Sections 18--20 of \cite{conwayop}.
As with the material in Section \ref{sec.optop}, all the basic properties in the case $H \neq K$ have essentially the same proofs;
the main tools this time are the singular value and polar decompositions.
\pagebreak

\begin{defi}[Schatten classes]\label{def.Schatten}
If $p \in [1,\infty)$, $\mathcal{E} \subseteq H$ is an orthonormal basis, and $A \in B(H;K)$, then we define
\[
\|A\|_{\cS_p(H;K)} = \|A\|_{\cS_p} \coloneqq \Bigg(\sum_{e \in \mathcal{E}} \la |A|^pe,e\ra_H\Bigg)^{\frac{1}{p}} \in [0,\infty],
\]
where $|A|^p \coloneqq (A^*A)^{\frac{p}{2}} \in B(H)$, and $\cS_p(H;K) \coloneqq \{A \in B(H;K) : \|A\|_{\cS_p} < \infty\}$.
Also, we define $\cS_{\infty}(H;K) \coloneqq B(H;K)$ with the operator norm
\[
\|A\|_{\cS_{\infty}(H;K)} = \|A\|_{\cS_{\infty}} \coloneqq \|A\| = \|A\|_{H\to K}.
\]
For $p \in [1,\infty]$, $\cS_p(H;K)$ is called the set of \textbf{Schatten} (or \textbf{Schatten--von Neumann}) $\boldsymbol{p}$\textbf{-class operators} from $H$ to $K$.
Also, we write
\[
\cK(H;K) \coloneqq\{\text{compact linear operators } H \to K\},
\]
$\cK(H) \coloneqq \cK(H;H)$, and $\cS_p(H) \coloneqq \cS_p(H;H)$.
\end{defi}
\begin{rem}
The operator $(A^*A)^{\frac{p}{2}}$ above is defined via continuous functional calculus (Section VIII.2 of \cite{conwayfunc}).
Also, we caution the reader that sometimes $\cS_{\infty}(H;K)$ is taken to be the space of compact operators $H \to K$, and often the letter $\cC$ is used instead of $\cS$.
\end{rem}

\begin{thm}[Properties of Schatten classes]\label{thm.Schatten}
Let $p \in [1,\infty]$.
\begin{enumerate}[label=(\roman*),font=\normalfont,leftmargin=2\parindent]
    \item $(\cS_p(H;K),\|\cdot\|_{\cS_p})$ is a Banach space, $\|\cdot\|_{\cS_p}$ is independent of the chosen orthonormal basis, and, when $p < \infty$, the set of finite rank linear operators $H \to K$ is dense in $\cS_p(H;K)$.
    Also, $(\cK(H;K),\|\cdot\|)$ is a Banach space with the set of finite rank linear operators $H \to K$ as a dense linear subspace.
    Finally, if $1 \leq p \leq q < \infty$, then $\cS_p(H;K) \subseteq \cS_q(H;K) \subseteq \cK(H;K)$, and the inclusions $\cS_p \hookrightarrow \cS_q \hookrightarrow \cK$ each have operator norm at most one.
    \item If $A \in B(H)$ and $\mathcal{E} \subseteq H$ is an orthonormal basis, then
    \[
    \sum_{e \in \mathcal{E}} |\la Ae,e\ra_H| \leq \|A\|_{\cS_1}.
    \]
    If $A \in \cS_1(H)$, then $\Tr(A) \coloneqq \sum_{e \in \mathcal{E}} \la Ae,e \ra_H \in \C$ is called the \textbf{trace} of $A$ and is independent of the choice of $\mathcal{E}$.
    Moreover, we have $\|A^*\|_{\cS_1} = \|A\|_{\cS_1}$ and $\Tr(A^*) = \overline{\Tr(A)}$, for all $A \in \cS_1(H)$.
    \item {\rm (H\"{o}lder's inequality)} Let $H_1,\ldots,H_{k+1}$ be complex Hilbert spaces.
    Suppose $p_1,\ldots,p_k \in [1,\infty]$ are such that $\frac{1}{p_1}+\cdots+\frac{1}{p_k} = \frac{1}{p}$. Then
    \[
    \|A_1\cdots A_k\|_{\cS_p(H_{k+1};H_1)} \leq \|A_1\|_{\cS_{p_1}(H_2;H_1)} \cdots \|A_k\|_{\cS_{p_k}(H_{k+1};H_k)},
    \]
    for all $A_1 \in B(H_2;H_1),\ldots, A_k \in B(H_{k+1};H_k)$.
    (As usual, $0\cdot \infty \coloneqq 0$.)\label{item.SchattenHolder}
    \item If $q \in [1,\infty]$ is such that $\frac{1}{p}+\frac{1}{q} = 1$ and $A \in \cS_p(H;K)$, $B \in \cS_q(K;H)$, then $\Tr(AB) = \Tr(BA)$.\label{item.Trflip}
    \item If $p \in [1,\infty),q \in (1,\infty]$ are such that $\frac{1}{p}+\frac{1}{q}=1$, then $\cS_q(H;K) \cong \cS_p(K;H)^*$ isometrically via the map $A \mapsto (B \mapsto \Tr(AB))$.
    Also, $\cS_1(H;K) \cong \cK(K;H)^*$ isometrically via the same map.\label{item.Schattendual}
    \item The weak$^*$ topology on $B(H;K)$ induced by the identification
    \[
    B(H;K) = \cS_{\infty}(H;K) \cong \cS_1(K;H)^*
    \]
    is called the \textbf{ultraweak topology}, and it agrees with the $\sigma$-WOT.
    In particular, finite rank linear operators $H \to K$ are ultraweakly dense in $B(H;K)$.\label{item.finrk}
\end{enumerate}
\end{thm}

Next, we review some basics of semifinite von Neumann algebras and noncommutative $L^p$-spaces.
(The reader who is uninterested in semifinite von Neumann algebras may skip at this time to the next section.)
If $a,b \in B(H)$, then we write $a \leq b$ or $b \geq a$ to mean that $b-a$ is a positive operator, i.e., $\la (b-a)h,h \ra_H \geq 0$ for $h \in H$.
If $\cM \subseteq B(H)$ is a von Neumann algebra, then we write
\[
\cM_+ \coloneqq \{a \in \cM : a \geq 0\}.\pagebreak
\]
It is easy to see that $\cM_+$ is closed in the WOT.
Also, if $(a_j)_{j \in J}$ is a net in $\cM_+$ that is bounded (there exists $b \in B(H)$ such that $a_j \leq b$, for all $j \in J$) and increasing ($j_1 \leq j_2 \Rightarrow a_{j_1}\leq a_{j_2}$), then $\sup_{j \in J} a_j$ exists in $B(H)_+$ and belongs to $\cM_+$.
(Please see Proposition 43.1 in \cite{conwayop}.)
This is often known as Vigier's theorem.

\begin{defi}[Trace]\label{def.trace}
Let $\cM \subseteq B(H)$ be a von Neumann algebra and $\tau \colon \cM_+ \to [0,\infty]$.
We call $\tau$ a \textbf{trace} if for all $a,b \in \cM_+$, $c \in \cM$, and $\lambda \in \R_+$,
\begin{enumerate}[label=(\roman*),font=\normalfont,leftmargin=2\parindent]
    \item $\tau(a+b) = \tau(a)+\tau(b)$;\label{item.add}
    \item $\tau(\lambda a) = \lambda\,\tau(a)$;\label{item.poshom}
    \item $\tau(c^*c) = \tau(cc^*)$.\label{item.trace}
\end{enumerate}
A trace $\tau$ on $\cM$ is called
\begin{enumerate}[label=(\roman*),font=\normalfont,leftmargin=2\parindent]
    \item \textbf{normal} if $\tau(\sup_{j \in J}a_j)  = \sup_{j \in J}\tau(a_j)$ whenever $(a_j)_{j \in J}$ is a bounded and increasing net in $\cM_+$;
    \item \textbf{faithful} if $a \in \cM_+$ and $\tau(a) = 0$ imply $a=0$;
    \item \textbf{semifinite} if $\tau(a) = \sup\{\tau(b) : a \geq b \in \cM_+, \, \tau(b) < \infty\}$ for $a \in \cM_+$.
\end{enumerate}
If $\tau$ is a normal, faithful, semifinite trace on $\cM$, then $(\cM,\tau)$ is called a \textbf{semifinite von Neumann algebra}.
\end{defi}
\begin{rem}
In the presence of \ref{item.add} and \ref{item.poshom}, condition \ref{item.trace} is equivalent to $\tau(u^*au) = \tau(a)$ for all $a \in \cM_+$ and all unitaries $u$ belonging to $\cM$.
This is Corollary 1 in Section I.6.1 of \cite{dixmier}. 
\end{rem}

For basic properties of traces on von Neumann algebras, please see Chapter I.6 of \cite{dixmier} or Section V.2 of \cite{takesaki}.
Motivating examples of semifinite von Neumann algebras are $(B(H),\Tr)$ and $(L^{\infty}(\Om,\mu),\int_{\Om} \boldsymbol{\cdot}\,d\mu)$, where $(\Om,\sF,\mu)$ is a $\sigma$-finite measure space and $L^{\infty}(\Om,\mu)$ is represented as multiplication operators on $L^2(\Om,\mu)$.
We now record some basics of $L^p$-spaces associated to a normal, faithful, semifinite trace.
We shall mostly draw results from \cite{dasilva,dixmierLp}.
For more information and/or different perspectives, please see \cite{fackkosaki,nelson,terp,yeadon}.

\begin{nota}\label{nota.ncLp}
Let $(\cM,\tau)$ be a semifinite von Neumann algebra, and fix $p \in [1,\infty)$.
Write
\[
\|a\|_{L^p(\tau)} \coloneqq \tau(|a|^p)^{\frac{1}{p}} \in [0,\infty],
\]
for all $a \in \cM$, and
\[
\mathcal{L}^p(\tau) \coloneqq \{a \in \cM : \|a\|_{L^p(\tau)}^p = \tau(|a|^p) < \infty\}.
\]
For the $p=\infty$ case, we take $\mathcal{L}^{\infty}(\tau) \coloneqq \cM$ with $\|\cdot\|_{L^{\infty}(\tau)} \coloneqq \|\cdot\|$.
\end{nota}

It turns out that $\mathcal{L}^1(\tau) \subseteq \cM$ is a two sided ideal of $\cM$ that is spanned by $\mathcal{L}^1(\tau)_+ = \mathcal{L}^1(\tau) \cap \cM_+$.
Moreover, there exists a unique linear extension of $\tau|_{\mathcal{L}_1(\tau)_+} \colon \mathcal{L}_1(\tau)_+ \to \C$ to $\mathcal{L}^1(\tau)$, which we notate the same way, and this extension satisfies
\[
\tau(ab) = \tau(ba) \; \text{ for } a \in \cM, \, b \in \mathcal{L}^1(\tau).
\]
Finally, if $b \in \mathcal{L}^1(\tau)$, then the map $\cM \ni a \mapsto \tau(ab) \in \C$ is $\sigma$-weakly continuous.
(These facts are proven as Proposition 1 in Section I.6.1, together with the sentence before Proposition 9 in Section I.1.6, of \cite{dixmier}.)

\begin{thm}[Properties of $\mathcal{L}^p(\tau)$]\label{thm.Lp}
Let $(\cM,\tau)$ be a semifinite von Neumann algebra and $p \in [1,\infty]$.
\begin{enumerate}[label=(\roman*),font=\normalfont,leftmargin=2\parindent]
    \item $(\mathcal{L}^p(\tau),\|\cdot\|_{L^p(\tau)})$ is a normed vector space. Its completion is denoted $(L^p(\tau),\|\cdot\|_{L^p(\tau)})$ and is called the \textbf{noncommutative $\boldsymbol{L^p}$-space} associated to $(\cM,\tau)$.
    We therefore also write $\mathcal{L}^p(\tau) = L^p(\tau) \cap \cM$.
    \item If $a \in \mathcal{L}^1(\tau)$, then $|\tau(a)| \leq \tau(|a|) = \|a\|_{L^1(\tau)}$.
    Thus $\tau \colon \mathcal{L}^1(\tau) \to \C$ extends uniquely to a bounded linear map, notated the same way, $L^1(\tau) \to \C$.\label{item.tauext}
    \item {\rm (Nonommutative H\"{o}lder's inequality)} Suppose $p_1,\ldots,p_k \in [1,\infty]$ are such that $\frac{1}{p_1}+\cdots+\frac{1}{p_k} = \frac{1}{p}$.
    Then, for all $a_1,\ldots,a_k \in \cM$:\label{item.LpHolder}
    \[
    \|a_1\cdots a_k\|_{L^p(\tau)} \leq \|a_1\|_{L^{p_1}(\tau)} \cdots \|a_k\|_{L^{p_k}(\tau)}.
    \]
    \item If $q \in [1,\infty]$ is such that $\frac{1}{p}+\frac{1}{q}=1$, then for all $a \in \cM$,\label{item.Lpdual}
    \[
    \|a\|_{L^p(\tau)} = \sup\{\|ab\|_{L^1(\tau)} : b \in \mathcal{L}^q(\tau),\,\|b\|_{L^q(\tau)} \leq 1\}.
    \]
\end{enumerate}
\end{thm}

When $(\cM,\tau) = (B(H),\Tr)$, we have that $\mathcal{L}^p(\tau) = L^p(\tau) = \cS_p(H)$ and $\|\cdot\|_{L^p(\tau)} = \|\cdot\|_{\cS_p}$.
Therefore, Theorem \ref{thm.Lp} generalizes parts of Theorem \ref{thm.Schatten} in the case $H=K$.

\subsection{Tensor products}\label{sec.tensprod}

\begin{nota}
Write $\otimes$ for the algebraic tensor product, $\otimes_2$ for the Hilbert space tensor product, and $\potimes$ for the Banach space projective tensor product.
\end{nota}

Please see Section 3.2 of \cite{brownozawa} or Section 2.6 of \cite{kadisonringrose1} for information about $\hotimes$ and Chapter 2 of \cite{ryan} for information about $\potimes$.
We assume the reader has some familiarity with these tensor products;
nevertheless, we recall the definitions/constructions thereof (at least when there are only two tensorands).
First of all, there is a unique inner product $\la \cdot, \cdot \ra_{H \otimes K}$ on $H \otimes K$ such that
\[
\la h_1 \otimes k_1, h_2 \otimes k_2 \ra_{H \otimes K} = \la h_1,h_2 \ra_H\la k_1,k_2 \ra_K,
\]
for all $h_1,h_2 \in H$, $k_1,k_2 \in K$. The Hilbert space tensor product $H \hotimes K$ is defined to be the completion of $H \otimes K$ with respect to $\la \cdot, \cdot \ra_{H \otimes K}$.
Also, if $A \in B(H)$ and $B \in B(K)$, then there exists a unique bounded linear map $A \hotimes B \in B(H \hotimes K)$ such that $(A \hotimes B)(h \otimes k) = Ah \otimes Bk$, for all $h \in H$ and $k \in K$.
Moreover, $\|A \hotimes B\|_{B(H \hotimes K)} = \|A\|_{B(H)}\|B\|_{B(K)}$, and the linear map $B(H) \otimes B(K) \to B(H \hotimes K)$ determined by
\[
A \otimes B \mapsto A \hotimes B
\]
is an injective $\ast$-homomorphism when $B(H) \otimes B(K)$ is given the tensor product $\ast$-algebra structure.
This allows us to view $B(H) \otimes B(K)$ as a $\ast$-subalgebra of $B(H \hotimes K)$ and justifies writing, as we shall, $A \otimes B$ instead of $A \hotimes B$.

We shall only use the projective tensor product to motivate the \textbf{integral projective tensor products} described in Section \ref{sec.IPTP}.
Here is what is useful to know for this purpose.
For Banach spaces $V$ and $W$ and $u \in V \otimes W$, define
\[
\pi(u) \coloneqq \inf\Bigg\{\sum_{j=1}^n\|v_j\|_V\|w_j\|_W : u = \sum_{j=1}^n v_j\otimes w_j\Bigg\}.
\]
The projective tensor product $V \potimes W$ is defined to be the completion of $V \otimes W$ with respect to the norm $\pi$, and it satisfies the type of universal property that the algebraic tensor product satisfies:
it bounded linearizes bounded bilinear maps.
As for a more concrete description of $V \potimes W$, it can be shown that every element $u \in V \potimes W$ admits a decomposition
\[
u = \sum_{n=1}^{\infty} v_n \otimes w_n \;\text{ with }\; \sum_{n=1}^{\infty}\underbrace{\|v_n\|_V\|w_n\|_W}_{\pi(v_n \otimes w_n)} < \infty \numberthis\label{eq.projdecomp}
\]
and that
\[
\pi(u) = \inf\Bigg\{\sum_{n=1}^{\infty}\|v_n\|_V\|w_n\|_W : u = \sum_{n=1}^{\infty} v_n \otimes w_n \text{ as in \eqref{eq.projdecomp}}\Bigg\}.
\]
Please see Chapter 2 of \cite{ryan} for a proper development.

\section{Vector and operator valued integrals}\label{sec.vopinteg}

For this section, fix a measure space $(\Sigma,\sH,\rho)$.

\subsection{Upper and lower integrals}\label{sec.nonmeasint}

\begin{defi}[Upper and lower integrals]\label{def.nonmeasint}
For an arbitrary (not necessarily measurable) function $h \colon \Sigma \to [0,\infty]$, we define
\begin{align*}
    \overline{\int_{\Sigma}} h(\sigma) \, \rho(d\sigma) & = \overline{\int_{\Sigma}} h \, d\rho \coloneqq \inf\Bigg\{\int_{\Sigma} \tilde{h} \, d\rho : h \leq \tilde{h} \;\; \rho\text{-a.e.}, \; \tilde{h} \colon \Sigma \to [0,\infty] \text{ measurable}\Bigg\} \text{ and} \\
    \underline{\int_{\Sigma}} h(\sigma) \, \rho(d\sigma) & = \underline{\int_{\Sigma}} h \, d\rho \coloneqq \sup\Bigg\{\int_{\Sigma} \tilde{h} \, d\rho : \tilde{h} \leq h \;\; \rho\text{-a.e.}, \; \tilde{h} \colon \Sigma \to [0,\infty] \text{ measurable}\Bigg\}
\end{align*}
to be, respectively, the \textbf{upper} and \textbf{lower integral} of $h$ with respect to $\rho$.
\end{defi}

Of course, if $h$ is $\big(\overline{\sH}^{\rho},\cB_{[0,\infty]}\big)$-measurable, where $\overline{\sH}^{\rho}$ is the $\rho$-completion of $\sH$, then $\underline{\int_{\Sigma}} h \, d\rho = \overline{\int_{\Sigma}} h \,d\rho$.
Here are the properties of upper and lower integrals relevant to this paper.

\begin{prop}[Properties of upper and lower integrals]\label{prop.nonmeasinteg}
Fix arbitrary functions $h,h_1,h_2 \colon \Sigma \to [0,\infty]$ and a nonnegative number $c \geq 0$.
\begin{enumerate}[label=(\roman*),font=\normalfont,leftmargin=2\parindent]
    \item $\underline{\int_{\Sigma}} h \, d\rho \leq \overline{\int_{\Sigma}} h \, d\rho$. Also, if $S \in \sH$, then $\underline{\int_S}h|_S\,d\rho = \underline{\int_{\Sigma}}1_Sh\,d\rho$ and $\overline{\int_S}h|_S\,d\rho = \overline{\int_{\Sigma}}1_Sh\,d\rho$.
    \item If $h_1 \leq h_2$ $\rho$-almost everywhere, then $\underline{\int_{\Sigma}} h_1 \, d\rho \leq \underline{\int_{\Sigma}} h_2 \, d\rho$ and $\overline{\int_{\Sigma}} h_1 \, d\rho \leq \overline{\int_{\Sigma}} h_2 \, d\rho$.
    \item $\overline{\int_{\Sigma}}c\, h \, d\rho = c \overline{\int_{\Sigma}} h \, d\rho$ and $\overline{\int_{\Sigma}} (h_1+h_2) \, d\rho \leq \overline{\int_{\Sigma}} h_1 \, d\rho + \overline{\int_{\Sigma}} h_2 \, d\rho$.\label{item.sublin}
    \item {\rm (Dominated convergence theorem)} Suppose $(h_n)_{n \in \N}$ is a sequence of functions $\Sigma \to [0,\infty]$ such that $h_n \hspace{-0.5mm}\to\hspace{-0.4mm} 0$ \hspace{-0.3mm}pointwise\hspace{-0.4mm} $\rho$-almost \hspace{-0.4mm}everywhere\hspace{-0.4mm} as $n \hspace{-0.5mm}\to\hspace{-0.4mm} \infty$.
    If $\overline{\int_{\Sigma}}(\sup_{n \in \N} \hspace{-0.4mm}h_n) d\rho \hspace{-0.4mm}<\hspace{-0.4mm} \infty$, then\hspace{-0.2mm} $\underline{\int_{\Sigma}}h_n d\rho \to 0$ as $n \hspace{-0.5mm}\to\hspace{-0.4mm} \infty$. \label{item.poorDCT}\vspace{-4.25mm}
    \item If $(\Sigma_n,\sH_n,\rho_n)_{n \in \N}$ is a sequence of measure spaces and
    \[
    (\Sigma,\sH,\rho) = \Bigg(\coprod_{n \in \N} \Sigma_n,\coprod_{n \in \N} \sH_n, \sum_{n =1}^{\infty}\rho_n\Bigg)
    \]
    is their disjoint union, then\label{item.disjunupint}
    \[
    \underline{\int_{\Sigma}} h \, d\rho = \sum_{n = 1}^{\infty}\underline{\int_{\Sigma_n}} h|_{\Sigma_n} \, d\rho_n \; \text{ and } \; \overline{\int_{\Sigma}} h \, d\rho = \sum_{n = 1}^{\infty}\overline{\int_{\Sigma_n}} h|_{\Sigma_n} \, d\rho_n.
    \]
\end{enumerate}
\end{prop}
\begin{proof}
The first three items are easy consequences of the definitions.
We take the remaining items in turn.

\ref{item.poorDCT} By definition of the upper integral, there is some measurable $h \colon \Sigma \to [0,\infty]$ such that $\int_{\Sigma} h \, d\rho < \infty$ and $\sup_{n \in \N}h_n \leq h$ $\rho$-almost everywhere.
By definition of the lower integral, if $n \in \N$, then there exists a measurable $\tilde{h}_n \colon \Sigma \to [0,\infty]$ such that $0 \leq \tilde{h}_n \leq h_n$ $\rho$-almost everywhere and
\[
\underline{\int_{\Sigma}}h_n \, d\rho - \frac{1}{n} < \int_{\Sigma} \tilde{h}_n \, d\rho.
\]
Since $h_n \to 0$ $\rho$-almost everywhere as $n \to \infty$, and $0 \leq \tilde{h}_n \leq h_n$ $\rho$-almost everywhere, we have that $\tilde{h}_n \to 0$ $\rho$-almost everywhere as $n \to \infty$.
Also, $\tilde{h}_n \leq h_n \leq h$ $\rho$-almost everywhere.
Therefore, by the dominated convergence theorem,
\[
\limsup_{n \to \infty}\underline{\int_{\Sigma}}h_n \, d\rho = \limsup_{n \to \infty}\Bigg(\underline{\int_{\Sigma}}h_n \, d\rho - \frac{1}{n}\Bigg) \leq \limsup_{n \to \infty}\int_{\Sigma} \tilde{h}_n \, d\rho = 0,
\]
as desired.

\ref{item.disjunupint} We prove the claimed identity for lower integrals and leave the proof of the identity for upper integrals to the reader.
First, the definition of the disjoint union measure space and a standard application of the monotone convergence theorem give the desired identity when $h$ is measurable.
Next, suppose $\tilde{h} \colon \Sigma \to [0,\infty]$ is measurable and $\tilde{h} \leq h$ $\rho$-almost everywhere.
Then $\tilde{h}|_{\Sigma_n} \colon \Sigma_n \to [0,\infty]$ is measurable and $\tilde{h}|_{\Sigma_n} \leq h|_{\Sigma_n}$ $\rho_n$-almost everywhere, for all $n \in \N$. Therefore, by definition of the lower integral and our initial observation,
\[
\int_{\Sigma} \tilde{h} \, d\rho = \sum_{n = 1}^{\infty}\int_{\Sigma_n} \tilde{h}|_{\Sigma_n} \, d\rho_n \leq \sum_{n = 1}^{\infty}\underline{\int_{\Sigma_n}} h|_{\Sigma_n} \, d\rho_n.
\]
Taking the supremum over $\tilde{h}$ then yields
\[
\underline{\int_{\Sigma}}h \,d\rho \leq \sum_{n = 1}^{\infty}\underline{\int_{\Sigma_n}} h|_{\Sigma_n} \, d\rho_n.
\]
Finally, fix $\e > 0$.
By definition of the lower integral, if $n \in \N$, then there exists a measurable $\tilde{h}_n \colon \Sigma_n \to [0,\infty]$ such that $\tilde{h}_n \leq h|_{\Sigma_n}$ $\rho_n$-almost everywhere and
\[
\underline{\int_{\Sigma_n}} h|_{\Sigma_n} \, d\rho_n \leq \int_{\Sigma_n} \tilde{h}_n \, d\rho_n + \frac{\e}{2^n}.\pagebreak
\]
Letting $\tilde{h} \colon \Sigma \to [0,\infty]$ be the unique measurable function such that $\tilde{h}|_{\Sigma_n} = \tilde{h}_n$, for all $n \in \N$, we have that $\tilde{h} \leq h$ $\rho$-almost everywhere and
\[
\sum_{n=1}^{\infty}\underline{\int_{\Sigma_n}} h|_{\Sigma_n} \, d\rho_n \leq \sum_{n=1}^{\infty}\int_{\Sigma_n} \tilde{h}_n \, d\rho_n + \e = \int_{\Sigma} \tilde{h} \, d\rho + \e \leq \underline{\int_{\Sigma}} h \, d\rho + \e.
\]
Since $\e > 0$ was arbitrary, we get that $\sum_{n = 1}^{\infty}\underline{\int_{\Sigma_n}} h|_{\Sigma_n} \, d\rho_n \leq \underline{\int_{\Sigma}}h \,d\rho$ as well.
\end{proof}

\subsection{Gel'fand--Pettis integrals}\label{sec.winteg}

In this section, we discuss a ``weak" notion of vector valued integration that --- in the next section --- we shall apply to maps $\Sigma \to V$, where $V \subseteq B(H;K)$ is a linear subspace.

\begin{defi}[Weak measurability and integrability]\label{def.GPint}
Let $V$ be a Hausdorff locally convex topological vector space (HLCTVS) with topological dual $V^*$, and let $f \colon \Sigma \to V$ be a map. We say that $f$ is \textbf{weakly measurable} if it is $(\sH,\sigma(V^*))$-measurable. We say that $f$ is \textbf{weakly} or \textbf{Gel'fand--Pettis integrable} if it is weakly measurable, $\int_{\Sigma} |\ell \circ f|\,d\rho < \infty$ whenever $\ell \in V^*$, and for every $S \in \sH$ there exists (necessarily unique) $\int_S f \, d\rho \in V$ such that
\[
\ell \Bigg(\int_S f \, d\rho \Bigg) = \int_S (\ell \circ f) \,d\rho, \numberthis\label{eq.wint}
\]
for all $\ell \in V^*$. In this case, we call $\int_S f \,d\rho = \int_S f(\sigma) \, \rho(d\sigma) \in V$ the \textbf{weak} or \textbf{Gel'fand--Pettis integral} (over $S$) of $f$ with respect to $\rho$.
\end{defi}

\begin{rem}
The uniqueness of $\int_S f \, d\rho$ is a consequence of the fact that $V^*$ separates points when $V$ is Hausdorff and locally convex, i.e., $\ell(v) = 0$ for all $\ell \in V^*$ implies $v = 0$.
\end{rem}

By the comments after Notation \ref{nota.sigmaalggen}, $f \colon \Sigma \to V$ is weakly measurable if and only if $\ell \circ f \colon \Sigma \to \C$ is $(\sH,\cB_{\C})$-measurable for all $\ell \in V^*$.

\begin{prop}[Basic properties of weak integrals]\label{prop.GelfPetint}
Let $V$ and $W$ be HLCTVSs.
\begin{enumerate}[label=(\roman*),font=\normalfont,leftmargin=2\parindent]
    \item If $f \colon \Sigma \to V$ is weakly integrable and $S \in \sH$, then $1_Sf$ is weakly integrable and $\int_{\Sigma} 1_Sf\,d\rho = \int_S f\, d\rho$.\label{item.wintrest}
    \item If $f,g \colon \Sigma \to V$ are weakly measurable (respectively, integrable) and $c \in \C$, then $f+c\,g$ is weakly measurable (respectively, integrable with $\int_S(f+c\,g)\,d\rho = \int_Sf\,d\rho + c\int_Sg \,d\rho$ whenever $S \in \sH$).\label{item.wintlin}
    \item If $T \colon V \to W$ is linear and continuous and $f \colon \Sigma \to V$ is weakly measurable (respectively, integrable), then $Tf = T(f(\cdot)) \colon \Sigma \to W$ is weakly measurable (respectively, integrable with $T\int_S f \, d\rho = \int_S Tf\,d\rho$ whenever $S \in \sH$).\label{item.wintcontlin}
    \item {\rm (Triangle inequality)} If $f \colon \Sigma \to V$ is weakly integrable and $\alpha$ is a continuous seminorm on $V$, then
    \[
    \alpha\Bigg(\int_{\Sigma} f \, d\rho \Bigg) \leq \underline{\int_{\Sigma}} \alpha(f) \,d\rho.
    \]
    In particular, if $V$ is normed, then $\big\|\int_{\Sigma} f \,d\rho \big\|_V \leq \underline{\int_{\Sigma}} \|f\|_V \, d\rho$.\label{item.winttriangle}
    \item {\rm (Dominated convergence theorem)} Suppose $V$ is sequentially complete and $\mathscr{S} \subseteq \R_+^V$ is a collection of seminorms generating the topology of $V$.
    If $(f_n)_{n \in \N}$ is a sequence of weakly measurable maps $\Sigma \to V$ converging pointwise to $f \colon \Sigma \to V$, then $f$ is weakly measurable.
    If $f_n \colon \Sigma \to V$ is weakly integrable, for all $n \in \N$, and
    \[
    \overline{\int_{\Sigma}} \sup_{n \in \N}\alpha (f_n) \,d\rho < \infty,
    \]
    for all $\alpha \in \mathscr{S}$, then $f$ is weakly integrable and $\int_{\Sigma} f_n \, d\rho \to \int_{\Sigma} f \,d\rho$ in $V$ as $n \to \infty$.\label{item.wDCT}
\end{enumerate}
\end{prop}
\begin{proof}
The first three items are easy consequences of the definitions and the fact that \eqref{eq.wint} (for all $\ell \in V^*$) uniquely characterizes $\int_S f \,d\rho$.
We take the remaining items one at a time.
\pagebreak

\ref{item.winttriangle} Let $v \coloneqq \int_{\Sigma} f \, d\rho$.
By the Hahn--Banach theorem, there is some linear $\ell \colon V \to \C$ such that $\ell(v) = \alpha(v)$ and $|\ell(w)| \leq \alpha(w)$, for all $w \in V$.
Since $\alpha$ is continuous, the latter inequality implies $\ell \in V^*$.
Therefore, we have the following by definition of the Gel'fand--Pettis and lower integrals:
\[
\alpha\Bigg(\int_{\Sigma} f\,d\rho\Bigg) = \ell\Bigg(\int_{\Sigma}f \,d\rho\Bigg) = \int_{\Sigma} (\ell \circ f) \, d\rho \leq \int_{\Sigma} |\ell \circ f| \, d\rho \leq \underline{\int_{\Sigma}} \alpha(f) \,d\rho.
\]

\ref{item.wDCT} The first statement is clear.
For the second, let $S \in \sH$ and $\alpha \in \mathscr{S}$.
Then
\[
\alpha\Bigg(\int_S f_n\,d\rho - \int_S f_m \, d\rho\Bigg) = \alpha\Bigg(\int_S(f_n-f_m) \, d\rho\Bigg) \leq \underline{\int_S}\alpha(f_n-f_m)\,d\rho \leq \underline{\int_{\Sigma}}\alpha(f_n-f_m) \,d\rho \to 0
\]
as $n,m \to \infty$ by the triangle inequality and Proposition \ref{prop.nonmeasinteg}\ref{item.poorDCT}, which applies because $\alpha(f_n-f_m) \to 0$ pointwise as $n,m \to \infty$ and
\[
\overline{\int_{\Sigma}} \sup_{n,m \in \N}\alpha(f_n-f_m) \,d\rho \leq 2\overline{\int_{\Sigma}}\sup_{n \in \N}\alpha(f_n)\,d\rho < \infty.
\]
In particular, the sequence $\big(\int_Sf_n\,d\rho\big)_{n \in \N}$ is Cauchy in $V$.
Letting $v_S\in V$ be its limit, which exists because $V$ is assumed to be sequentially complete, we claim $v_S = \int_S f \, d\rho$.
Indeed, if $\ell \in V^*$, then there exist $\alpha_1,\ldots,\alpha_m \in \mathscr{S}$ and $C \geq 0$ such that $|\ell(v)| \leq C\sum_{j=1}^m\alpha_j(v)$, for all $v \in V$.
Therefore,
\[
\ell(v_S) = \lim_{n \to \infty}\ell\Bigg(\int_S f_n\,d\rho\Bigg) = \lim_{n \to \infty}\int_S (\ell \circ f_n)\,d\rho = \int_S (\ell \circ f)\,d\rho
\]
by the standard dominated convergence theorem, which applies because
\[
\int_{\Sigma} \sup_{n \in \N}|\ell \circ f_n|\,d\rho \leq  C \overline{\int_{\Sigma}}\sum_{j=1}^m\sup_{n \in \N}\alpha_j(f_n) \,d\rho \leq C \sum_{j=1}^m\overline{\int_{\Sigma}}\sup_{n \in \N}\alpha_j(f_n) \,d\rho < \infty.
\]
Thus $f$ is weakly integrable, and the above shows at least that $\int_S f_n\,d\rho \to \int_S f\,d\rho$ weakly as $n \to \infty$.
Finally, whenever $\alpha \in \mathscr{S}$, we have
\[
\alpha\Bigg(\int_S f_n\,d\rho - \int_S f \, d\rho\Bigg) = \alpha\Bigg(\int_S(f_n-f) \, d\rho\Bigg) \leq \underline{\int_S}\alpha(f_n-f) \,d\rho \leq \underline{\int_{\Sigma}}\alpha(f_n-f) \,d\rho \to 0
\]
as $n \to \infty$ again by the triangle inequality and Proposition \ref{prop.nonmeasinteg}\ref{item.poorDCT}, which applies because $\alpha(f_n-f) \to 0$ pointwise as $n \to \infty$ and $\sup_{n \in \N}\alpha(f_n-f) \leq 2\sup_{n \in \N}\alpha(f_n)$, which has finite upper integral.
\end{proof}
\begin{rem}\label{rem.normmeas}
If $\alpha \colon V \to \R_+$ is a continuous seminorm, then weak measurability of $f \colon \Sigma \to V$ does not necessarily imply measurability of $\alpha(f) \colon \Sigma \to \R_+$ (even when $\alpha$ is a norm).
This is one reason we need upper and lower integrals above.
In a certain sense, this is what complicates the study of MOIs --- or, really, vector valued integrals --- in a general (not necessarily separable) setting.
\end{rem}

Here now are two situations in which the existence of weak integrals is always guaranteed.
We prove the following well known result using standard closed graph theorem arguments, as in Section II.3 of \cite{diesteluhl}.

\begin{prop}[Existence of weak integrals]\label{prop.wintext}
Let $V$ be a Banach space.
\begin{enumerate}[label=(\roman*),font=\normalfont,leftmargin=2\parindent]
    \item If $g \colon \Sigma \to V^*$ is such that, for all $v \in V$, $g(\cdot)(v) \colon \Sigma \to \C$ is measurable and $\int_{\Sigma} |g(\sigma)(v)|\,\rho(d\sigma) < \infty$, then $g$ is Gel'fand--Pettis integrable in the weak$^*$ topology on $V^*$ and $\big\|\int_{\Sigma} g \, d\rho \big\|_{V^*} \leq \underline{\int_{\Sigma}}\|g\|_{V^*}\,d\rho$.
    In this case, we say $g$ is \textbf{weak$\boldsymbol{^*}$ integrable} and $\int_{\Sigma} g \,d\rho$ is the \textbf{weak$\boldsymbol{^*}$ integral} of $g$. \label{item.weakstarintexist}
    \item Suppose $V$ is reflexive.
    If $f \colon \Sigma \to V$ is weakly measurable and
    $\int_{\Sigma} |\ell \circ f| \,d \rho < \infty$, for all $\ell \in V^*$, then
    $f$ is weakly integrable.\label{item.triangleconverse}
\end{enumerate}
\end{prop}
\begin{proof}
We take both items one at a time.

\ref{item.weakstarintexist} First, we recall the evaluation map $V \ni v \mapsto (\ell \mapsto \ell(v)) \in (V^*,\,$weak$^*)^*$ is a linear isomorphism.
Therefore, we know $g \colon \Sigma \to V^*$ is weakly measurable in the weak$^*$ topology.
Second, define $T \colon V \to L^1(\Sigma,\rho)$ by $v \mapsto g(\cdot)(v)$.
Certainly, $T$ is linear.
Also, if $(v_n)_{n \in \N} \in V^{\N}$ converges to $v \in V$, then, for all $\sigma \in \Sigma$,
\[
(Tv_n)(\sigma) = g(\sigma)(v_n) \to g(\sigma)(v) = (Tv)(\sigma)\pagebreak
\]
as $n \to \infty$.
Therefore, if $(Tv_n)_{n \in \N}$ converges in $L^1(\Sigma,\rho)$, its limit must be $Tv$.
In other words, $T$ is closed.
By the closed graph theorem, $T$ is bounded.
Finally, fix $S \in \sH$, and define $I_S \colon V \to \C$ by $v \mapsto \int_S g(\sigma)(v)\,\rho(d\sigma) = \int_S (Tv)(\sigma)\,\rho(d\sigma)$.
Then $\|I_S\|_{V^*} \leq \|T\|_{V \to L^1} < \infty$, i.e., $I_S \in V^*$.
Unraveling the definitions and using the first sentence of this paragraph again, we conclude that $g$ is Gel'fand--Pettis integrable in the weak$^*$ topology and that $I_S = \int_S g\,d\rho$.

For the triangle inequality, we have
\begin{align*}
    \Bigg\|\int_{\Sigma}g\,d\rho \Bigg\|_{V^*} & = \sup\Bigg\{\Bigg|\Bigg(\int_{\Sigma} g\,d\rho \Bigg)(v)\Bigg| : v \in V,\,\|v\|_V \leq 1\Bigg\} = \sup\Bigg\{\Bigg|\int_{\Sigma} g(\sigma)(v)\,\rho(d\sigma)\Bigg|  : v \in V, \, \|v\|_V \leq 1\Bigg\} \\
    & \leq \sup\Bigg\{\int_{\Sigma} |g(\sigma)(v)|\,\rho(d\sigma)  : v \in V, \, \|v\|_V \leq 1\Bigg\} \leq \underline{\int_{\Sigma}} \|g\|_{V^*}\,d\rho,
\end{align*}
as desired.

\ref{item.triangleconverse} Suppose first that $V$ is a general Banach space, $\ev \colon V \hookrightarrow V^{**}$ is the natural inclusion, and we retain the same assumptions on $f$.
Since $((\ev \circ f)(\sigma))(\ell) = (\ell \circ f)(\sigma)$, for all $\ell \in V^*$ and $\sigma \in \Sigma$, the assumptions on $f$ and the previous part imply that $\ev \circ f \colon \Sigma \to V^{**} = (V^*)^*$ is weak$^*$ integrable.
(The weak$^*$ integrals of $\ev \circ f$ are called the \textit{Dunford integrals} of $f$.)
Now, if $V$ is reflexive and $S \in \sH$, then there exists unique $v_S \in V$ such that $\ev(v_S) = \int_S (\ev \circ f)\,d\rho$, where the latter is the weak$^*$ integral of $f$ over $S$.
Unraveling the definitions yields that $v_S = \int_S f \, d\rho$ and therefore that $f$ is weakly integrable.
\end{proof}

We end this section with a characterization of weak measurability and integrability in a Hilbert space, as this will be important for developing the ``operator valued" case in the following section.

\begin{ex}[Hilbert spaces]\label{ex.wmeasHilb}
Let $(H,\la\cdot,\cdot \ra_H)$ be a Hilbert space.
The Riesz representation theorem says that $\ell \in H^*$ if and only if there is some $k \in H$ such that $\ell(h) = \la h, k \ra_H$, for all $h \in H$.
Therefore, $f \colon \Sigma \to H$ is weakly measurable if and only if $\la f(\cdot), k \ra_H \colon \Sigma \to \C$ is $(\sH,\cB_{\C})$-measurable whenever $k \in H$.
Also, since $H$ is reflexive, $f \colon \Sigma \to H$ is weakly integrable if and only if $\la f(\cdot),k \ra_H \in L^1(\Sigma,\rho)$ whenever $k \in H$ (for example, if $f$ is weakly measurable and $\underline{\int_{\Sigma}}\|f\|_H\,d\rho < \infty$) by Proposition \ref{prop.wintext}\ref{item.triangleconverse}.
\end{ex}

\subsection{Pointwise Pettis integrals}\label{sec.pwpetwstarint}

At this time, we shall begin to use material from Section \ref{sec.optop} about various topologies on spaces of bounded operators.

\begin{conv}
For the duration of this and the following section, fix complex Hilbert spaces $(H,\la \cdot,\cdot \ra_H)$ and $(K,\la \cdot,\cdot\ra_K)$.
Also, write $\|\cdot\| = \|\cdot\|_{H \to K}$ for the operator norm on $B(H;K)$.
\end{conv}

In this paper, we shall need to integrate maps $\Sigma \to V$, where $V \subseteq B(H;K)$ is a linear subspace.
Given the number of topologies on $B(H;K)$ one might consider, there are possibly many notions of Gel'fand--Pettis integrability of a map $\Sigma \to V$.
It turns out that in most reasonable circumstances, the choice of topology (from Section \ref{sec.optop}) does not matter.
Indeed, we now introduce a notion of integrability --- \textbf{pointwise Pettis integrability} --- in this setting that is, in practice, quite easy to check.
Then we describe the relationship between pointwise Pettis integrability and Gel'fand--Pettis integrability in various operator topologies.

\begin{lem}\label{lem.pwPetexist}
Suppose $F \colon \Sigma \to B(H;K)$ is such that $\la F(\cdot)h,k\ra_K \colon \Sigma \to \C$ is $(\sH,\cB_{\C})$-measurable and $\int_{\Sigma}|\la F(\sigma)h,k \ra_K| \,\rho(d\sigma) < \infty$, for all $h \in H$ and $k \in K$.
Then $F(\cdot)h \colon \Sigma \to K$ is weakly integrable, for all $h \in H$, and the map $\rho_S(F) \colon H \to K$ defined by $h \mapsto \int_S F(\sigma)h\, \rho(d\sigma)$ is linear and bounded, for all $S \in \sH$.
\end{lem}
\begin{proof}
Let $h \in H$ and $k \in K$.
Then the characterization given in Example \ref{ex.wmeasHilb} implies that $F(\cdot)h \colon \Sigma \to K$ and $F(\cdot)^*k \colon \Sigma \to H$ are both weakly integrable.
In particular, if $S \in \sH$ and we define a sesquilinear map $B_S \colon H \times K \to \C$ by $(h,k) \mapsto \int_S \la F(\sigma)h,k \ra_K \, \rho(d\sigma)$, then the identities
\[
\Bigg\la \int_S F(\sigma)h \, \rho(d\sigma), k \Bigg\ra_K = \int_S \la F(\sigma)h,k\ra_K \, \rho(d\sigma) = \int_S \la h,F(\sigma)^*k\ra_H \, \rho(d\sigma) = \Bigg\la h,\int_S F(\sigma)^*k \, \rho(d\sigma)\Bigg\ra_H
\]
imply that $B_S$ is a sesquilinear map that is bounded in each argument separately.
By the principle of uniform boundedness, $B_S$ is bounded.
Since $\la \rho_S(F)h,k \ra_K = B_S(h,k)$, we get $\rho_S(F) \in B(H;K)$.
\end{proof}

\begin{defi}[Pointwise Pettis integrability]\label{def.Pettint}
A map $F \colon \Sigma \to B(H;K)$ is called \textbf{pointwise weakly measurable} if, for all $h \in H$ and $k \in K$, the function $\la F(\cdot)h,k \ra_K \colon \Sigma \to \C$ is $(\sH,\cB_{\C})$-measurable.
If also $\int_{\Sigma} |\la F(\sigma)h,k\ra_K|\,\rho(d\sigma) < \infty$, for all $h \in H$ and $k \in K$, then we call $F$ \textbf{pointwise Pettis integrable}.
In this case, the operator $\rho_S(F) \in B(H;K)$ from Lemma \ref{lem.pwPetexist} is called the \textbf{pointwise Pettis integral} (over $S$) of $F$  with respect to $\rho$.
Finally, if in addition $V \subseteq B(H;K)$ is a linear subspace, $F(\Sigma) \subseteq V$, and $\rho_S(F) \in V$ whenever $S \in \sH$, then we call $F$ \textbf{pointwise Pettis integrable in $\boldsymbol{V}$}.
\end{defi}
\begin{rem}
The use of the term ``pointwise" is not at all standard;
we have chosen it in order to avoid overusing or abusing the terms ``weak" and ``strong."
The pointwise Pettis integral above is often called a ``weak integral" in contrast to the ``stronger" Bochner integral.
However, we shall see in Theorem \ref{thm.wintBHK} that the notion of pointwise Pettis integrability of $F \colon \Sigma \to B(H;K)$ is precisely the notion of Gel'fand--Pettis integrability of $F$ as a map with values in $(B(H;K),\,$WOT$)$ \textit{or} $(B(H;K),\,$SOT$)$.
It is therefore arguably just as appropriate to apply the term ``strong" to the pointwise Pettis integral.
\end{rem}

\begin{rem}\label{rem.pwPetinM}
Note that if $H=K$ and $V = \cM \subseteq B(H)$ is a von Neumann algebra, then any pointwise Pettis integrable map $F \colon \Sigma \to \cM \subseteq B(H)$ is actually pointwise Pettis integrable in $\cM$.
Indeed, suppose $a \in \cM'$.
If $S \in \sH$, then it is easy to see from the definition that $a\,\rho_S(F) = \rho_S(a\,F) = \rho_S(F\,a) = \rho_S(F)\,a$, i.e., $\rho_S(F) \in \cM'' = \cM$ by the bicommutant theorem.
\end{rem}

We now compare the notion of pointwise Pettis integrability to various notions of Gel'fand--Pettis integrability.
To this end, we recall (Theorem \ref{thm.optop}\ref{item.predual}--\ref{item.weakstar}) that if $V \subseteq B(H;K)$ is a $\sigma$-weakly closed linear subspace --- e.g., a von Neumann algebra --- then $V_* \coloneqq (V,\sigma$-WOT$)^*$ is the predual of $V$.
More precisely, $V_*$ is a Banach space with the operator norm and the map $V \ni A \mapsto (\ell \mapsto \ell(A)) \in V_*^{\;*}$ is an isometric isomorphism that is also a homeomorphism with respect to the $\sigma$-WOT on $V$ and weak$^*$ topology on $V_*^{\;*}$.

\begin{thm}[Integrals in $V \subseteq B(H;K)$]\label{thm.wintBHK}
Let $V \subseteq B(H;K)$ be a linear subspace, $F \colon \Sigma \to V$, and $\cT \in \{\mathrm{WOT}, \,\mathrm{SOT}, \,\mathrm{S^*OT}\}$.
(To be clear, we consider $\cT$ to be a topology on $V$.)
\begin{enumerate}[label=(\roman*),font=\normalfont,leftmargin=2\parindent]
    \item If
    \begin{align*}
        \sF & \coloneqq \sigma(\{V \ni A \mapsto \Tr(AB) \in \C : B \in \cS_1(K;H)\}) \subseteq 2^V \; \text{ and} \\
        \sG & \coloneqq \sigma(\{V \ni A \mapsto \la Ah,k \ra_K \in \C : h \in H, \, k \in K\}) \subseteq 2^V,
    \end{align*}
    then $\sF = \sG = \sigma\big((V,\,\sigma\text{-}\cT)^*\big) = \sigma\big((V,\cT)^*\big)$.
    In particular, $F \colon \Sigma \to V$ is pointwise weakly measurable if and only if it is weakly measurable in $\cT$, if and only if it is weakly measurable in $\sigma$-$\cT$.\label{item.wmeaschar}
    \item $F \colon \Sigma \to V$ is pointwise Pettis integrable in $V$ if and only if it is Gel'fand--Pettis integrable in $\cT$, and the pointwise Pettis and Gel'fand--Pettis integrals of $F$ agree in this case.
    In particular, we may write $\rho_S(F) = \int_S F \, d\rho$, for all $S \in \sH$, with no chance of confusion.
    Also, the following triangle inequality holds in this case:\label{item.wintchar}
    \[
    \Bigg\|\int_{\Sigma} F \, d\rho \Bigg\| \leq \underline{\int_{\Sigma}} \|F\| \, d\rho.
    \]
    \item Suppose $V \subseteq B(H;K)$ is $\sigma$-weakly closed.
    Then $F \colon \Sigma \to V$ is Gel'fand--Pettis integrable in $\sigma$-$\cT$ if and only if it is weak$^*$ integrable under the usual identification $V \cong V_*^{\;*}$, if and only if $F$ is pointwise weakly measurable and
    \[
    \int_{\Sigma}|\la (F(\sigma)h_n)_{n \in \N},(k_n)_{n \in \N} \ra_{\ell^2(\N;K)}|\,\rho(d\sigma) < \infty \numberthis\label{eq.sigmaWOTint}
    \]
    whenever $(h_n)_{n \in \N} \in \ell^2(\N;H)$ and $(k_n)_{n \in \N} \in \ell^2(\N;K)$.
    In this case, the Gel'fand--Pettis, weak$^*$, and pointwise Pettis integrals of $F$ all agree.\label{item.wintexist}
    \item If $H=K$ and $V = \cM \subseteq B(H)$ is a von Neumann algebra, then the notions of pointwise weak measurability and Gel'fand--Pettis integrability in $\sigma$-$\cT$ are independent of the representation of $\cM$.
    More precisely, if $\cN$ is another von Neumann algebra and $\pi \colon \cM \to \cN$ is a $\ast$-isomorphism in the algebraic sense, then $F$ is pointwise weakly measurable (respectively, Gel'fand--Pettis integrable in $\sigma$-$\cT$) if and only if $\pi \circ F$ is pointwise weakly measurable (respectively, Gel'fand--Pettis integrable in $\sigma$-$\cT$ with $\int_S (\pi \circ F) \, d\rho = \pi\big(\int_S F \, d\rho\big)$
    whenever $S \in \sH$).\label{item.indepofrep}
\end{enumerate}
\end{thm}
\begin{proof}
We take each item in turn.

\ref{item.wmeaschar} By Theorem \ref{thm.optop}\ref{item.contchar},
\[
(V,\cT)^* = \spn\{V \ni A \mapsto \la Ah,k\ra_K \in \C : h \in H,\, k \in K\}. \numberthis\label{eq.taudual}
\]
Thus $\sG = \sigma\big((V,\cT)^*\big)$. By Theorem \ref{thm.optop}\ref{item.sigmacontchar}, $(V,\sigma\text{-}\cT)^* = (V,\sigma\text{-}\mathrm{WOT})^*$.
By the Hahn--Banach theorem and Theorem \ref{thm.Schatten}\ref{item.finrk}, we have
\[
(V,\sigma\text{-}\mathrm{WOT})^* = \{V \ni A \mapsto \Tr(AB) \in \C : B \in \cS_1(K;H)\}. \numberthis\label{eq.sigmataudual}
\]
Thus $\sF = \sigma\big((V,\sigma$-$\cT)^*\big)$.
Since $(V,\cT)^* \subseteq (V,\sigma$-$\cT)^*$, we have
\[
\sF = \sigma\big((V,\sigma\text{-}\cT)^*\big) \subseteq \sigma\big((V,\cT)^*\big) = \sG.
\]
It therefore suffices to prove that any element of $(V,\sigma$-$\cT)^*$ is $\sG$-measurable.
To this end, let $\ell \in (V,\sigma$-$\cT)^*$.
By Theorem \ref{thm.optop}\ref{item.sigmacontchar}, there are $(h_n)_{n \in \N} \in \ell^2(\N;H)$ and $(k_n)_{n \in \N} \in \ell^2(\N;K)$ such that
\[
\ell(A) = \la (Ah_n)_{n \in \N},(k_n)_{n \in \N} \ra _{\ell^2(\N;K)}= \sum_{n=1}^{\infty}\la Ah_n,k_n\ra_K,
\]
for all $A \in V$.
This exhibits $\ell$ a pointwise limit of elements of
\[
\{V \ni A \mapsto \la Ah,k \ra_K \in \C : h \in H, k \in K\}.
\]
Thus $\ell$ is $\sG$-measurable.

\ref{item.wintchar} The equivalence of Gel'fand--Pettis integrability in $\cT$ and pointwise Pettis integrability in $V$ (with agreement of Gel'fand--Pettis and pointwise Pettis integrals) follows directly from the definitions and \eqref{eq.taudual}.
For the triangle inequality, note that if $F \colon \Sigma \to B(H;K)$ is pointwise Pettis integrable and $h \in H$, then the $K$-valued triangle inequality gives
\[
\Bigg\|\int_{\Sigma} F(\sigma)h\,\rho(d\sigma)\Bigg\|_K \leq \underline{\int_{\Sigma}}\|F(\sigma)h\|_K\,\rho(d\sigma) \leq \|h\|_H\underline{\int_{\Sigma}}\|F\| \, d\rho.
\]
Taking the supremum over $h \in H$ with $\|h\|_H \leq 1$ gives the desired result.

\ref{item.wintexist} Since $(V,\sigma$-$\cT)^* = (V,\sigma$-WOT$)^* = V_*$, we may assume $\cT = \mathrm{WOT}$.
Under the usual identification $V_*^{\;*} \cong V$, the weak$^*$ topology on $V_*^{\;*}$ corresponds to the $\sigma$-WOT on $V$.
This implies the first equivalence. By item \ref{item.wmeaschar}, pointwise weak measurability of $F$ is equivalent to weak measurability of $F$ in the $\sigma$-WOT on $V$ and therefore to weak measurability of $F$ in the weak$^*$ topology on $V_*^{\;*}$.
By Theorem \ref{thm.optop}\ref{item.sigmacontchar}, the requirement \eqref{eq.sigmaWOTint} is precisely the requirement that
\[
\int_{\Sigma} |\ell \circ F|\,d\rho < \infty,
\]
for all $\ell \in V_* = (V,\sigma$-WOT$)^*$.
Proposition \ref{prop.wintext}\ref{item.weakstarintexist} and the form of the identification $V \cong V_*^{\;*}$ then give the second equivalence.

\ref{item.indepofrep} Again, we may assume $\cT = \mathrm{WOT}$.
This item follows from the fact that $\ast$-isomorphisms are automatically $\sigma$-WOT-homeomorphisms (Theorem \ref{thm.optop}\ref{item.starisom}), pointwise weak measurability is equivalent to weak measurability in the $\sigma$-WOT (item \ref{item.wmeaschar}), and Proposition \ref{prop.GelfPetint}\ref{item.wintcontlin} applied to $\pi$ and $\pi^{-1}$.
\end{proof}

Let $V \subseteq B(H;K)$ be a $\sigma$-weakly closed linear subspace and $F \colon \Sigma \to V$.
In view of Theorem \ref{thm.wintBHK}\ref{item.wintexist} and its proof, we have the following.
First, weak measurability of $F$ in the $\sigma$-WOT is equivalent to weak measurability of $F$ in the weak$^*$ topology when we identify $V \cong V_*^{\;*}$ in the usual way, which in turn is equivalent to pointwise weak measurability of $F$.
We are therefore justified in using the term \textbf{weak$\boldsymbol{^*}$ measurable} in place of pointwise weakly measurable.
Second, Gel'fand--Pettis integrability of $F$ in the $\sigma$-WOT is equivalent to weak$^*$ integrability of $F$ when we identify $V \cong V_*^{\;*}$ in the usual way, which in turn is equivalent to weak$^*$ measurability of $F$ and the requirement \eqref{eq.sigmaWOTint}.
We are therefore justified in using the term \textbf{weak$\boldsymbol{^*}$ integrable} in place of (any of) the terms in the previous sentence.
We end this section by isolating an important takeaway from this possibly confusing development.

\begin{cor}[Criterion for weak$^*$ integrability]\label{cor.wstarint}
Let $V \subseteq B(H;K)$ be a $\sigma$-weakly closed linear subspace and $F \colon \Sigma \to V$ be a map.
If $F$ is pointwise weakly measurable and $\underline{\int_{\Sigma}}\|F\|\,d\rho < \infty$, then $F$ is weak$^*$ integrable, and, for $S \in \sH$, the weak$^*$ integral $\int_S F \,d\rho \in V$ is uniquely determined by
\[
\Bigg\la \Bigg(\int_S F \,d\rho \Bigg)h,k \Bigg\ra_K = \int_{\Sigma}\la F(\sigma)h,k \ra_K\,\rho(d\sigma), \text{ for all } h \in H \text{ and } k \in K,
\]
i.e., $\int_S F\,d\rho = \rho_S(F)$.
\end{cor}
\begin{proof}
If $(h_n)_{n \in \N} \in \ell^2(\N;H)$ and $(k_n)_{n \in \N} \in \ell^2(\N;K)$, then
\begin{align*}
    \int_{\Sigma}|\la (F(\sigma)h_n)_{n \in \N},(k_n)_{n \in \N} \ra _{\ell^2(\N;K)}|\,\rho(d\sigma) & \leq \underline{\int_{\Sigma}}\|(F(\sigma)h_n)_{n \in \N}\|_{\ell^2(\N;K)}\|(k_n)_{n \in \N}\|_{\ell^2(\N;K)}\,\rho(d\sigma) \\
    & \leq \|(h_n)_{n \in \N}\|_{\ell^2(\N;H)}\|(k_n)_{n \in \N}\|_{\ell^2(\N;K)}\underline{\int_{\Sigma}}\|F\|\,d\rho
\end{align*}
by the Cauchy--Schwarz inequality.
Therefore, if the right hand side is finite, then we conclude from Theorem \ref{thm.wintBHK}\ref{item.wintexist} that $F$ is weak$^*$ integrable with agreement of weak$^*$ and pointwise Pettis integrals of $F$.
\end{proof}

\subsection{Minkowski's inequality for operator valued integrals}\label{sec.SchattenLpestim}

Our last order of business concerning operator valued integrals is to prove a Schatten $p$-norm Minkowski inequality for pointwise Pettis integrals.
This will be absolutely crucial, at least in the case $p=1$, for our development of MOIs.
After doing so, we use a similar technique to prove a noncommutative $L^p$-norm Minkowski inequality for weak$^*$ integrals in a semifinite von Neumann algebra.
We begin by proving a well known and useful recharacterization of $\|\cdot\|_{\cS_1} = \|\cdot\|_{\cS_1(H;K)}$.

\begin{defi}[Orthonormal frames]\label{def.onframe}
For $n \in \N_0$, write
\[
O_n(H) \coloneqq \{\boldsymbol{e} = (e_1,\ldots,e_n) \in H^n : e_1,\ldots,e_n \text{ is orthonormal}\}
\]
for the set of \textbf{orthonormal frames of length} $\boldsymbol{n}$. Note $O_0(H) = \emptyset$.
\end{defi}

\begin{lem}\label{lem.S1rechar}
If $A \in B(H;K)$, then
\[
\|A\|_{\cS_1} = \sup \Bigg\{ \sum_{j=1}^n |\la A e_j,f_j \ra_K| : n \in \N_0,  \,\boldsymbol{e} \in O_n(H), \, \boldsymbol{f} \in O_n(K)\Bigg\},
\]
where empty sums are zero.
In particular, $A \in \cS_1(H;K)$ if and only if the right hand side is finite.
\end{lem}

When $H=K$, this recharacterization is the $p=1$ case of Lemma 2.3.4 in \cite{ringrose}.

\begin{proof}
Let $A = U|A|$ be the polar decomposition of $A$ and $\mathcal{E}_1$ be an orthonormal basis of $\ker |A| = \ker A$.
We recall that the polar decomposition of $A$ is the (unique) product decomposition $A = U|A|$, where $U \in B(H;K)$ is a partial isometry with initial space $(\ker A)^{\perp} = (\ker |A|)^{\perp}$ and final space $\overline{\im A}$.
Note in this case that $|A| = U^*A$.

First, by definition, if $e \in \mathcal{E}_1$, then $|A|e = 0$ and therefore $\la |A|e,e \ra_H = 0$.
Next, complete $\mathcal{E}_1$ to an orthonormal basis $\mathcal{E} \supseteq \mathcal{E}_1$ of $H$.
Then
\[
\|A\|_{\cS_1} = \sum_{e \in \mathcal{E}} \la |A|e,e \ra_H = \sum_{e \in \mathcal{E} \setminus \mathcal{E}_1} \la |A|e,e\ra_H = \sum_{e \in \mathcal{E} \setminus \mathcal{E}_1} \la U^*Ae,e \ra_H = \sum_{e \in \mathcal{E} \setminus \mathcal{E}_1} \la Ae,Ue\ra_K.
\]
Of course, $\mathcal{E} \setminus \mathcal{E}_1$ is an orthonormal basis of $(\ker |A|)^{\perp}$, the initial space of $U$, on which $U$ is an isometry by definition.
Therefore, defining $f_e \coloneqq Ue \text{ for } e \in \mathcal{E} \setminus \mathcal{E}_1$, we see that $\la f_e, f_{\tilde{e}} \ra_K = \la Ue,U\tilde{e} \ra_K = \la e,\tilde{e} \ra_H = \delta_{e \tilde{e}}$ whenever $e,\tilde{e} \in \mathcal{E} \setminus \mathcal{E}_1$, i.e., $(f_e)_{e \in \mathcal{E} \setminus \mathcal{E}_1}$ is orthonormal.
It follows (by taking finite subsets $E \subseteq  \mathcal{E} \setminus \mathcal{E}_1$) that
\[
\|A\|_{\cS_1} \leq \sup \Bigg\{ \sum_{j=1}^n |\la A e_j,f_j \ra_K| : n \in \N_0,  \boldsymbol{e} \in O_n(H), \boldsymbol{f} \in O_n(K)\Bigg\}.
\]
For the other inequality, suppose $\|A\|_{\cS_1} < \infty$, and fix $n \in \N$, $\boldsymbol{e} \in O_n(H)$, and $\boldsymbol{f} \in O_n(K)$.
Let $V \colon H \to K$ be the unique partial isometry such that $Ve_j = f_j$ for $1 \leq j \leq n$ and $V \equiv 0$ on $(\spn\{e_1,\ldots,e_n\})^{\perp}$.
If we complete $\{e_1,\ldots,e_n\}$ to an orthonormal basis $\mathcal{E}$ of $H$, then
\[
\sum_{j=1}^n|\la Ae_j,f_j \ra_K| = \sum_{j=1}^n|\la Ae_j,Ve_j \ra_K | = \sum_{j=1}^n|\la V^*Ae_j,e_j \ra_H |  \leq \sum_{e \in \mathcal{E}} |\la V^*Ae,e \ra_H| \leq \|V^*A\|_{\cS_1} \leq \|A\|_{\cS_1}
\]
because $\|V^*\|_{K \to H} = \|V\|_{H \to K} = 1$. Thus
\[
\sup \Bigg\{ \sum_{j=1}^n |\la A e_j,f_j \ra_K| : n \in \N_0, \, \boldsymbol{e} \in O_n(H), \,\boldsymbol{f} \in O_n(K)\Bigg\} \leq \|A\|_{\cS_1},
\]
as desired.
\end{proof}

\begin{thm}[Schatten norm Minkowski integral inequality]\label{thm.Spinteg}
Suppose that $F \colon \Sigma \to B(H;K)$ is pointwise Pettis integrable.
Then
\[
\Bigg\|\int_{\Sigma} F \, d \rho \Bigg\|_{\cS_p} \leq \underline{\int_{\Sigma}}\|F\|_{\cS_p} \, d\rho,
\]
for all $p \in [1,\infty]$.
In particular, if the right hand side is finite, then $\int_{\Sigma} F \,d\rho \in \cS_p(H;K)$.
\end{thm}
\begin{proof}
The case $p=\infty$ is contained in Theorem \ref{thm.wintBHK}\ref{item.wintchar}.
We first prove the case $p=1$, from which the rest of the cases will follow.
Define
\[
A \coloneqq \int_{\Sigma} F \, d\rho \in B(H;K),
\]
and fix $\boldsymbol{e} = (e_1,\ldots,e_n) \in O_n(H)$ and $\boldsymbol{f} = (f_1,\ldots,f_n) \in O_n(K)$.
By definition of the pointwise Pettis integral and Lemma \ref{lem.S1rechar}, we have
\begin{align*}
    \sum_{j=1}^n|\la Ae_j, f_j \ra_K| &= \sum_{j=1}^n\Bigg|\int_{\Sigma}\la F(\sigma)e_j, f_j \ra_K \, \rho(d\sigma) \Bigg| \leq \int_{\Sigma}\underbrace{\sum_{j=1}^n|\la F(\sigma)e_j, f_j \ra_K|}_{\leq \|F(\sigma)\|_{\cS_1}}\, \rho(d\sigma) \leq \underline{\int_{\Sigma}} \|F\|_{\cS_1} \, d\rho.
\end{align*}
Taking\hspace{-0.3mm} a \hspace{-0.3mm}supremum\hspace{-0.3mm} over $n \hspace{-0.3mm}\in\hspace{-0.3mm} \N_0$, $\boldsymbol{e} \hspace{-0.3mm}\in\hspace{-0.3mm} O_n\hspace{-0.3mm}(H)$, and $\boldsymbol{f} \hspace{-0.3mm}\in\hspace{-0.3mm} O_n\hspace{-0.3mm}(K)$ gives, \hspace{-0.3mm}again\hspace{-0.3mm} by \hspace{-0.3mm}Lemma\hspace{-0.3mm} \ref{lem.S1rechar}, $\|A\|_{\cS_1} \hspace{-0.3mm}\leq \underline{\int_{\Sigma}} \|F\|_{\cS_1} d\rho$.

Next, let $p,q \in (1,\infty)$ be such that $\frac{1}{p}+\frac{1}{q}=1$.
If $B \in B(K;H)$, then, by what we have just proven and H\"{o}lder's inequality for the Schatten norms,
\[
\|AB\|_{\cS_1} = \Bigg\|\int_{\Sigma} F(\sigma)B \, \rho(d\sigma) \Bigg\|_{\cS_1} \leq \underline{\int_{\Sigma}} \|F(\sigma)B\|_{\cS_1} \, \rho(d\sigma) \leq \|B\|_{\cS_q}\underline{\int_{\Sigma}}\|F\|_{\cS_p} \, d\rho.
\]
Therefore, if $\underline{\int_{\Sigma}}\|F\|_{\cS_p} \, d\rho < \infty$, then $A \in \cS_p(H;K)$ and
\[
\Bigg\|\int_{\Sigma} F\,d\rho\Bigg\|_{\cS_p} = \|A\|_{\cS_p} = \sup\{\underbrace{|\Tr(AB)|}_{\leq \|AB\|_{\cS_1}} : B \in B(K;H), \,\|B\|_{\cS_q} \leq 1\} \leq \underline{\int_{\Sigma}} \|F\|_{\cS_p} \, d \rho
\]
by duality for the Schatten classes (Theorem \ref{thm.Schatten}\ref{item.Schattendual}).
\end{proof}

As we just saw, the case $p=1$ is really the key to Theorem \ref{thm.Spinteg}.
Since this case is also the most important for our application of interest, we offer a few more words about it.
The proof we have just presented is ``from first principles" in the sense that we did not use any technology from the theory of vector valued integrals;
we only used Lemma \ref{lem.S1rechar} and the definition of the pointwise Pettis integral.
There is, however, an interesting alternative proof that uses Proposition \ref{prop.wintext}\ref{item.weakstarintexist} instead of Lemma \ref{lem.S1rechar}.

\begin{proof}[Second proof of Theorem \ref{thm.Spinteg} when $p=1$]
Since the conclusion is clear when $\underline{\int_{\Sigma}} \|F\|_{\cS_1}\,d\rho = \infty$, we assume $\underline{\int_{\Sigma}} \|F\|_{\cS_1}\,d\rho < \infty$.
In this case, $\|F\|_{\cS_1} < \infty$ almost everywhere (exercise). Since neither $\int_{\Sigma} F \, d\rho$ nor $\underline{\int_{\Sigma}} \|F\|_{\cS_1}\,d\rho$ changes if we modify $F$ on a set of measure zero, we may assume $\|F(\sigma)\|_{\cS_1} < \infty$, for all $\sigma \in \Sigma$;
i.e., $F(\Sigma) \subseteq \cS_1(H;K)$.
We claim in this case that $F \colon \Sigma \to \cS_1(H;K)$ is weak$^*$ integrable when we identify $\cS_1(H;K) \cong \cK(K;H)^*$ as in Theorem \ref{thm.Schatten}\ref{item.Schattendual}.
Indeed, if $B \colon K \to H$ is a finite rank linear operator, then\pagebreak\ $\Tr(F(\cdot)B) \colon \Sigma \to \C$ is measurable by Theorem \ref{thm.wintBHK}\ref{item.wmeaschar}.
Now, if $B \in \cK(K;H)$ is arbitrary, then there is a sequence $(B_n)_{n \in \N}$ of finite rank linear operators $K \to H$ such that $\|B-B_n\|_{K \to H} \to 0$ as $n \to \infty$.
If $\sigma \in \Sigma$, then this gives
\[
|\Tr(F(\sigma)B) - \Tr(F(\sigma)B_n)| = |\Tr(F(\sigma)(B-B_n))| \leq \|F(\sigma)\|_{\cS_1}\|B-B_n\|_{K \to H} \to 0
\]
as $n \to \infty$.
Thus $\Tr(F(\cdot)B) \colon \Sigma \to \C$ is measurable.
Also,
\[
\int_{\Sigma}|\Tr(F(\sigma)B)|\,\rho(d\sigma) \leq \underline{\int_{\Sigma}}\|F(\sigma)B\|_{\cS_1}\,\rho(d\sigma) \leq \|B\|_{K \to H}\underline{\int_{\Sigma}}\|F\|_{\cS_1}\,d\rho < \infty.
\]
Therefore, by Proposition \ref{prop.wintext}\ref{item.weakstarintexist}, $F \colon \Sigma \to \cK(K;H)^*$ is weak$^*$ integrable and
\[
\Bigg\|\int_{\Sigma} F\,d\rho \Bigg\|_{\cS_1} = \Bigg\|\int_{\Sigma} F \,d\rho \Bigg\|_{\cK(K;H)^*} \leq \underline{\int_{\Sigma}}\|F\|_{\cK(K;H)^*}\,d\rho = \underline{\int_{\Sigma}}\|F\|_{\cS_1}\,d\rho.
\]
Modulo the detail, which we leave to the reader, that the weak$^*$ integral of $F$ agrees with its pointwise Pettis integral, this completes the proof.
\end{proof}

\begin{rem}
It is worth mentioning that when $H$ and $K$ are separable, it is possible to prove Theorem \ref{thm.Spinteg} using the basic theory of the Bochner integral because in this case $\cS_p(H;K)$ is separable (when $p < \infty$).
Since we dealt with the general case, additional care --- in the form of either Lemma \ref{lem.S1rechar} or Proposition \ref{prop.wintext}\ref{item.weakstarintexist} --- was required.
\end{rem}

\begin{cor}[Schatten dominated convergence theorem]\label{cor.SchattenDCT}
Fix some $p \in [1,\infty]$.
Let $(F_n)_{n \in \N}$ be a sequence of pointwise Pettis integrable maps $\Sigma \to B(H;K)$ such that $F_n(\Sigma) \subseteq \cS_p(H;K)$, for all $n \in \N$.
Suppose $F \colon \Sigma \to \cS_p(H;K) \subseteq B(H;K)$ is such that $F_n \to F$ pointwise in $\cS_p(H;K)$ as $n \to \infty$.
If
\[
\overline{\int_{\Sigma}}\sup_{n \in \N}\|F_n\|_{\cS_p} \, d\rho < \infty, \numberthis\label{eq.Schattenbounded}
\]
then\hspace{-0.4mm} $F$ \hspace{-0.45mm}is\hspace{-0.45mm} pointwise \hspace{-0.45mm}Pettis\hspace{-0.45mm} (in \hspace{-0.45mm}fact,\hspace{-0.45mm} weak$^*$) \hspace{-0.45mm}integrable\hspace{-0.45mm} and \hspace{-0.45mm}we\hspace{-0.45mm} have \hspace{-0.45mm}the\hspace{-0.45mm} following \hspace{-0.45mm}convergence\hspace{-0.4mm} in \hspace{-0.4mm}$\cS_p\hspace{-0.45mm}(H;\hspace{-0.45mm}K)$ as $n \hspace{-0.45mm}\to\hspace{-0.45mm} \infty$:
\[
\int_{\Sigma} F_n \, d\rho \to \int_{\Sigma} F \, d\rho. \numberthis\label{eq.Schattenconv}
\]
\end{cor}
\begin{proof}
By a standard exercise using the principle of uniform boundedness, $(B(H;K),\sigma$-WOT$)$ is sequentially complete.
Since $\|\cdot\| \leq \|\cdot\|_{\cS_p}$, it follows from Proposition \ref{prop.GelfPetint}\ref{item.wDCT} (applied to the collection of seminorms from Section \ref{sec.optop} defining the $\sigma$-WOT) that $F \colon \Sigma \to B(H;K)$ is weak$^*$ integrable and the convergence \eqref{eq.Schattenconv} occurs in the $\sigma$-WOT.
But if $n \in \N$, then \eqref{eq.Schattenbounded} and Theorem \ref{thm.Spinteg} imply that $\int_{\Sigma} F_n \,d\rho \in \cS_p(H;K)$.
Finally,
\[
\Bigg\|\int_{\Sigma} F_n \, d\rho - \int_{\Sigma}F \, d\rho \Bigg\|_{\cS_p} = \Bigg\|\int_{\Sigma} (F_n-F) \,d\rho \Bigg\|_{\cS_p} \leq \underline{\int_{\Sigma}} \|F_n-F\|_{\cS_p} \, d\rho \to 0
\]
as $n \to \infty$ by Theorem \ref{thm.Spinteg} and Proposition \ref{prop.nonmeasinteg}\ref{item.poorDCT}, where the latter applies because of \eqref{eq.Schattenbounded}.
\end{proof}

Finally, we generalize Theorem \ref{thm.Spinteg} (with $H=K$) to noncommutative $L^p$-norms of a semifinite von Neumann algebra.
(The reader who is uninterested in semifinite von Neumann algebras may skip at this time to Section \ref{sec.MOI}.)
For this purpose, we first prove a version of Lemma \ref{lem.S1rechar} appropriate for this setting;
this is rather standard, but we supply a transparent proof for the reader's convenience.

\begin{lem}\label{lem.L1rechar}
Let $(\cM,\tau)$ be a semifinite von Neumann algebra.
If $a \in \cM$, then
\[
\|a\|_{L^1(\tau)} = \tau(|a|) = \sup\{|\tau(ab)| : b \in \mathcal{L}^1(\tau), \,\|b\| \leq 1\}.
\]
\end{lem}
\begin{proof}
Let $a \in \cM$.
If $b \in \mathcal{L}^1(\tau)$, then, by Theorem \ref{thm.Lp}\ref{item.tauext}--\ref{item.LpHolder},
\[
|\tau(ab)| \leq \|ab\|_{L^1(\tau)} \leq \|a\|_{L^1(\tau)} \|b\|_{L^{\infty}(\tau)}= \tau(|a|)\,\|b\|.
\]
Thus $\tau(|a|) \geq \sup\{|\tau(ab)| : b \in \mathcal{L}^1(\tau), \|b\| \leq 1\}$.
\pagebreak

Now, let $a = u|a|$ be the polar decomposition of $a$.
Suppose $p \in \cM$ is a $\tau$-finite projection, i.e., $p \in \mathrm{Proj}(\cM) \coloneqq \{q \in \cM : q^2=q=q^*\}$ and $\tau(p) < \infty$.
If $b \coloneqq pu^*$, then $b \in \mathcal{L}^1(\tau)$, $\|b\| \leq 1$, and
\[
\tau(ab) = \tau(apu^*) = \tau(u^*ap) = \tau(|a|p) = \tau\big(|a|^{\frac{1}{2}}|a|^{\frac{1}{2}}p\big) = \tau\big(|a|^{\frac{1}{2}}p|a|^{\frac{1}{2}}\big).
\]
If we can show that the net of $\tau$-finite projections (directed by $\leq$) increases to the identity, then the normality of $\tau$ would give
\[
\tau(|a|) = \sup\{\tau\big(|a|^{\frac{1}{2}}p|a|^{\frac{1}{2}}\big) : p \in L^1(\tau) \cap \mathrm{Proj}(\cM)\}.
\]
Combining this with the identity above, we would conclude, as desired, that
\[
\tau(|a|) \leq \sup\{|\tau(ab)| : b \in \mathcal{L}^1(\tau), \,\|b\| \leq 1\}.
\]

To complete the proof, we show $\mathrm{Proj}(L^1(\tau)) \coloneqq L^1(\tau) \cap \mathrm{Proj}(\cM)$ increases to the identity --- in other words, $\sup \mathrm{Proj}(L^1(\tau)) = 1$.
(A priori, this supremum is a projection in $\cM$ by Proposition 1.1 in Chapter V of \cite{takesaki}.)
To this end, suppose $0 \neq q \in \mathrm{Proj}(\cM)$ is arbitrary.
We claim that there exists $0 \neq p \in \mathrm{Proj}(L^1(\tau))$ such that $p \leq q$.
Indeed, by faithfulness and semifiniteness of $\tau$, there is some $0 \leq x \in \cM$ such that $0 \neq x \leq p$ and $\tau(x) < \infty$.
Since $x$ is positive, it is selfadjoint and $\sigma(x) \subseteq [0,\infty)$.
Letting $P^x \colon \cB_{\sigma(x)} \to \cM$ be its projection valued spectral measure, we have that if $\e > 0$ and $G_{\e} \coloneqq \sigma(x) \cap [\e,\infty)$, then
\[
\e\,P^x(G_{\e}) = \int_{\sigma(x)}\e\,1_{G_{\e}} \,dP^x \leq \int_{\sigma(x)}\lambda\,P^x(d\lambda)  = x.
\]
Since $x \neq 0$ and $x$ is normal, $\sigma(x) \neq \{0\}$.
Therefore, there is some $\e > 0$ such that $P^x(G_{\e}) \neq 0$.
For this choice of $\e$, let $p \coloneqq P^x(G_{\e})$.
Then $0 \neq \e\,p \leq x$, so that $\tau(p) \leq \e^{-1}\tau(x) < \infty$, i.e., $p \in \mathrm{Proj}(L^1(\tau))$.
But also, $\e\,p \leq x \leq q$.
Since $p$ and $q$ are both projections, this implies $p \leq q$. This proves the claim.
But now, by definition of $\sup \mathrm{Proj}(L^1(\tau))$, there can be no nonzero $\tau$-finite projection $\leq 1-\sup \mathrm{Proj}(L^1(\tau))$, so $1-\sup \mathrm{Proj}(L^1(\tau)) = 0$ by what we have just proven.
\end{proof}

\begin{thm}[Noncommutative Minkowski inequality for integrals]\label{thm.Lpinteg}
Let $(\cM,\tau)$ be a semifinite von Neumann algebra.
If $F \colon \Sigma \to \cM$ is weak$^*$ integrable, then
\[
\Bigg\|\int_{\Sigma} F \,d\rho \Bigg\|_{L^p(\tau)} \leq \underline{\int_{\Sigma}}\|F\|_{L^p(\tau)}\,d\rho,
\]
for all $p \in [1,\infty]$.
In particular, if the right hand side is finite, then $\int_{\Sigma} F \,d\rho \in \mathcal{L}^p(\tau)$.
\end{thm}
\begin{proof}
Write $a \coloneqq \int_{\Sigma} F \,d\rho$. The case $p=\infty$ is contained in Theorem \ref{thm.wintBHK}\ref{item.wintchar}.
We first prove the case $p=1$, from which the rest of the cases will follow.
Let $b \in \mathcal{L}^1(\tau)$.
Since the map $\cM \ni c \mapsto \tau(cb) \in \C$ is $\sigma$-weakly continuous and $F$ is weak$^*$ integrable (i.e., weakly integrable in the $\sigma$-weak topology), we have
\[
\tau(ab) = \int_{\Sigma}\tau(F(\sigma)\,b)\,\rho(d\sigma)
\]
by definition.
Then, by Lemma \ref{lem.L1rechar} (twice),
\begin{align*}
    \|a\|_{L^1(\tau)} & = \sup\Bigg\{|\tau(ab)| = \Bigg|\int_{\Sigma} \tau(F(\sigma)\,b)\,\rho(d\sigma)\Bigg| : b \in \mathcal{L}^1(\tau), \, \|b\| \leq 1\Bigg\} \\
    & \leq \sup\Bigg\{\int_{\Sigma}|\tau(F(\sigma)\,b)|\,\rho(d\sigma) : b \in \mathcal{L}^1(\tau), \, \|b\| \leq 1\Bigg\} \leq \underline{\int_{\Sigma}}\|F\|_{L^1(\tau)}\,d\rho,
\end{align*}
as desired.

Now, let $p,q \in (1,\infty)$ be such that $\frac{1}{p}+\frac{1}{q} = 1$. If $b \in \mathcal{L}^q(\tau)$, then, by what we have just proven and noncommutative H\"{o}lder's inequality,
\[
\Bigg\|\underbrace{\int_{\Sigma} F(\sigma)\,b \, \rho(d\sigma)}_{ab} \Bigg\|_{L^1(\tau)} \leq \underline{\int_{\Sigma}} \|F(\sigma)\,b\|_{L^1(\tau)} \, \rho(d\sigma) \leq \|b\|_{L^q(\tau)}\underline{\int_{\Sigma}}\|F\|_{L^p(\tau)} \, d\rho.
\]
Therefore, if $\underline{\int_{\Sigma}}\|F\|_{L^p(\tau)} \, d\rho < \infty$, then $\int_{\Sigma} F \, d\rho = a \in \mathcal{L}^p(\tau)$ and
\[
\Bigg\|\int_{\Sigma} F \, d\rho\Bigg\|_{L^p(\tau)} = \sup\{\|ab\|_{L^1(\tau)} : b \in \cM, \,\|b\|_{L^q(\tau)} \leq 1\} \leq \underline{\int_{\Sigma}} \|F\|_{L^p(\tau)} \, d \rho
\]
by the dual characterization of the noncommutative $L^p$-norm (Theorem \ref{thm.Lp}\ref{item.Lpdual}).
\end{proof}

The motivation for the name is, of course, the classical Minkowski inequality for integrals (e.g., 6.19 in \cite{folland}).
In view of \ref{prop.GelfPetint}\ref{item.winttriangle}, it would be just as reasonable to call Theorems \ref{thm.Spinteg} and \ref{thm.Lpinteg}, respectively, the Schatten $p$-norm and noncommutative $L^p$-norm ``(integral) triangle inequalities."

\section{Multiple operator integrals (MOIs)}\label{sec.MOI}

For the duration of Section \ref{sec.MOI}, fix $k \in \N$ and, for each $j \in \{1,\ldots,k+1\}$, a projection valued measure space $(\Om_j,\sF_j,H_j,P_j)$.
Also, write $(\Om,\sF,H,P)$ for their (tensor) product
\[
(\Om_1 \times \cdots \times \Om_{k+1},\sF_1 \otimes \cdots \otimes \sF_{k+1},H_1 \hotimes \cdots \hotimes H_{k+1}, P_1 \otimes \cdots \otimes P_{k+1}).
\]
Please see Theorem \ref{thm.tensprodPVM} for the definition of $P$.

\subsection{Integral projective tensor products of \texorpdfstring{$L^{\infty}$}{}-spaces}\label{sec.IPTP}

We now discuss integral projective tensor products of $L^{\infty}$-spaces.
Formally, the idea is to replace the countable sum in the decomposition \eqref{eq.projdecomp} of elements of the classical projective tensor product with an integral over a $\sigma$-finite measure space.
To make this rigorous, we first observe that Minkowski's inequality for integrals with $p=\infty$ holds for projection valued measures.

\begin{lem}\label{lem.PVMink}
Fix a projection valued measure space $(\Xi,\sG,K,Q)$ and a $\sigma$-finite measure space $(\Sigma,\sH,\rho)$. If $\Phi \colon \Xi \times \Sigma \to [0,\infty]$ is measurable, then
\[
\Bigg\|\int_{\Sigma} \Phi(\cdot,\sigma) \, \rho(d\sigma)\Bigg\|_{L^{\infty}(Q)} \leq \underline{\int_{\Sigma}}\|\Phi(\cdot,\sigma)\|_{L^{\infty}(Q)} \, \rho(d\sigma), \numberthis\label{eq.PVMink}
\]
i.e., $\int_{\Sigma} \Phi(\xi,\sigma) \, \rho(d\sigma) \leq \underline{\int_{\Sigma}}\|\Phi(\cdot,\sigma)\|_{L^{\infty}(Q)} \, \rho(d\sigma)$ for $Q$-almost every $\xi \in \Xi$.
\end{lem}
\begin{proof}
If $\underline{\int_{\Sigma}}\|\Phi(\cdot,\sigma)\|_{L^{\infty}(Q)} \, \rho(d\sigma) = \infty$, then the conclusion is obvious.
We therefore suppose
\[
c \coloneqq \underline{\int_{\Sigma}}\|\Phi(\cdot,\sigma)\|_{L^{\infty}(Q)} \, \rho(d\sigma) < \infty.
\]
Next, note that by (the proof of) Tonelli's theorem, the function
\[
\Xi \ni \xi \mapsto \int_{\Sigma} \Phi(\xi,\sigma) \, \rho(d\sigma) \in [0,\infty]
\]
is measurable.
Thus
\[
G \coloneqq \Big\{\xi \in \Xi : \int_{\Sigma} \Phi(\xi,\sigma) \, \rho(d\sigma) > c \Big\} \in \sG.
\]
Now, let $h \in K$. Since $Q_{h,h} = \la Q(\cdot)h,h \ra_K$ is a finite measure, the classical Minkowski inequality for integrals (e.g., 6.19 in \cite{folland}) gives
\[
\Bigg\|\int_{\Sigma} \Phi(\cdot,\sigma) \, \rho(d\sigma) \Bigg\|_{L^{\infty}(Q_{h,h})} \leq \int_{\Sigma}\|\Phi(\cdot,\sigma)\|_{L^{\infty}(Q_{h,h})} \, \rho(d\sigma) \leq \underline{\int_{\Sigma}}\|\Phi(\cdot,\sigma)\|_{L^{\infty}(Q)} \, \rho(d\sigma) = c.
\]
(Part \hspace{-0.2mm}of what we are using from \hspace{-0.2mm}Minkowski's inequality \hspace{-0.2mm}for integrals is measurability \hspace{-0.2mm}of $\sigma \hspace{-0.2mm}\mapsto\hspace{-0.2mm} \|\Phi(\cdot,\hspace{-0.2mm}\sigma)\|_{\hspace{-0.2mm}L^{\infty}(Q_{h,h}\hspace{-0.2mm})}$.)
In particular, $\la Q(G)h,h \ra_K = Q_{h,h}(G) = 0$.
Since $h$ was arbitrary, we conclude $Q(G) = 0$.
\end{proof}

\begin{defi}[Integral projective tensor products]\label{def.IPTPspecial}
Let $\varphi \colon \Om \to \C$ be a function.
A $\boldsymbol{L_P^{\infty}}$\textbf{-integral projective decomposition} ($L_P^{\infty}$-IPD) of $\varphi$ is a choice $(\Sigma,\rho,\varphi_1,\ldots,\varphi_{k+1})$ of a $\sigma$-finite measure space $(\Sigma,\sH,\rho)$ and measurable functions $\varphi_1 \colon \Om_1 \times \Sigma \to \C,\ldots,\varphi_{k+1} \colon \Om_{k+1} \times \Sigma \to \C$ such that:
\begin{enumerate}[label=(\roman*),font=\normalfont,leftmargin=2\parindent]
    \item $\varphi_j(\cdot,\sigma) \in L^{\infty}(\Om_j,P_j)$, for all $j \in \{1,\ldots,k+1\}$ and $\sigma \in \Sigma$; \label{item.nullset}
    \item $\overline{\int_{\Sigma}} \,\|\varphi_1(\cdot,\sigma)\|_{L^{\infty}(P_1)} \cdots \|\varphi_{k+1}(\cdot,\sigma)\|_{L^{\infty}(P_{k+1})} \, \rho(d\sigma) < \infty$; and\label{item.upperintegofphis}
    \item for $P$-almost every $\boldsymbol{\om} = (\om_1,\ldots,\om_{k+1}) \in \Om$, we have\label{item.equality}
    \[
    \varphi(\boldsymbol{\om}) = \int_{\Sigma} \varphi_1(\om_1,\sigma) \cdots \varphi_{k+1}(\om_{k+1},\sigma) \, \rho(d\sigma).
    \]
    Note that the integral on the right hand side is defined for $P$-almost every $\boldsymbol{\om} = (\om_1,\ldots,\om_{k+1}) \in \Om$ by Lemma \ref{lem.PVMink} and requirement \ref{item.upperintegofphis}.
\end{enumerate}
Now, define $\|\varphi\|_{L^{\infty}(P_1) \iotimes \cdots \iotimes L^{\infty}(P_{k+1})}$ to be the infimum of the set of numbers
\[
\overline{\int_{\Sigma}} \|\varphi_1(\cdot,\sigma)\|_{L^{\infty}(P_1)} \cdots \|\varphi_{k+1}(\cdot,\sigma)\|_{L^{\infty}(P_{k+1})} \, \rho(d\sigma),
\]
where $(\Sigma,\rho,\varphi_1,\ldots,\varphi_{k+1})$ is a $L_P^{\infty}$-integral projective decomposition of $\varphi$.
(We take $\inf \emptyset \coloneqq \infty$.)
Noting that $\|\varphi\|_{L^{\infty}(P_1) \iotimes \cdots \iotimes L^{\infty}(P_{k+1})} = \|\psi\|_{L^{\infty}(P_1) \iotimes \cdots \iotimes L^{\infty}(P_{k+1})}$ if $\varphi=\psi$ $P$-a.e., so that $\|\cdot\|_{L^{\infty}(P_1) \iotimes \cdots \iotimes L^{\infty}(P_{k+1})}$ is well-defined on $L^{\infty}(\Om,P)$, we define $L^{\infty}(\Om_1,P_1) \iotimes \cdots \iotimes L^{\infty}(\Om_{k+1},P_{k+1})$ to be the set of $\varphi \in L^{\infty}(\Om,P)$ such that $\|\varphi\|_{L^{\infty}(P_1) \iotimes \cdots \iotimes L^{\infty}(P_{k+1})} < \infty$ and call it the \textbf{integral projective tensor product (IPTP) of} $\boldsymbol{L^{\infty}(\Om_1,P_1),\ldots,L^{\infty}(\Om_{k+1},P_{k+1})}$.
\end{defi}
\begin{rem}[Measurability issues]\label{rem.measissue}
The literature is rather cavalier with the definition of the IPTP above.
Indeed, if $(\Xi,\sG,K,Q)$ is a projection valued measure space, $(\Sigma,\sH,\rho)$ is a $\sigma$-finite measure space, and $\Phi \colon \Xi \times \Sigma \to \C$ is measurable, then the function $\Sigma \ni \sigma \mapsto \|\Phi(\cdot,\sigma)\|_{L^{\infty}(Q)} \in [0,\infty]$ is \textit{not} necessarily measurable.
In particular, it is important to specify in \eqref{eq.PVMink} and \ref{item.upperintegofphis} which integral (upper or lower) is being used.
This detail, which is important in arguments to come, has been side stepped in the existing literature.

It is worth discussing ``how non-measurable" $\sigma \mapsto \Phi_Q(\sigma) \coloneqq \|\Phi(\cdot,\sigma)\|_{L^{\infty}(Q)}$ can be in various circumstances.
We proceed from least to most pathological.
First, if $Q$ is equivalent to a $\sigma$-finite scalar measure --- as is always the case when $K$ is separable --- then $\Phi_Q$ is actually measurable.
In the next two examples, suppose $Q(G) = 0 \Leftrightarrow G = \emptyset$.
(Please see Example \ref{ex.countingPVM}.) For the second example, suppose $X$ is a complete, separable metric space and $(\Xi,\sG) = (X,\cB_X)$.
Then $\Phi_Q$ is ``almost measurable," i.e., $\Phi_Q$ is measurable with respect to the $\rho$-completion of $\sH$.
This fact follows from Corollary 2.13 in \cite{crauel}, which relies on the nontrivial ``measurable projection theorem" (Theorem III.23 in \cite{castaing}).
Because in this case $\Phi_Q$ is ``almost measurable," the upper and lower integrals of $\Phi_Q$ agree.
This is used implicitly --- and perhaps unknowingly --- in the literature (e.g., \cite{azamovetal,depagtersukochev,doddssub2}) but is never proven or cited as it should be.
Third, let $Y \subseteq [0,1]$ be a non-Lebesgue-measurable set, $(\Xi,\sG) = (Y,\cB_Y)$, and $(\Sigma,\sH,\rho) = ([0,1],\cB_{[0,1]},\,$Lebesgue$)$.
If $\Phi \coloneqq 1_{\Delta \cap (Y \times [0,1])}$, where $\Delta \coloneqq \{(x,x) : x \in [0,1]\}$ is the diagonal, then we have $\Phi_Q(\sigma) = \|\Phi(\cdot,\sigma)\|_{\ell^{\infty}(Y)} = 1_Y(\sigma)$, for all $\sigma \in [0,1]$.
Thus $\Phi_Q$ is not even Lebesgue measurable in this case.
\end{rem}

Here are the basic properties of $L^{\infty}(\Om_1,P_1) \iotimes \cdots \iotimes L^{\infty}(\Om_{k+1},P_{k+1})$.
Special cases of the following proposition have been stated in the literature (e.g., Lemma 4.6 in \cite{depagtersukochev}), but no proofs are written down.
For the sake of completeness --- especially in view of the measurability issues discussed in Remark \ref{rem.measissue} --- we provide a full proof here.
In the statement below, a \textbf{Banach $\boldsymbol{\ast}$-algebra} is a unital Banach algebra endowed with an isometric $\ast$-operation.

\begin{prop}[Basic properties of IPTPs]\label{prop.IPTPspecial}
If $\varphi \in L^{\infty}(\Om_1,P_1) \iotimes \cdots \iotimes L^{\infty}(\Om_{k+1},P_{k+1})$, then
\[
\|\varphi\|_{L^{\infty}(P)} \leq \|\varphi\|_{L^{\infty}(P_1) \iotimes \cdots \iotimes L^{\infty}(P_{k+1})}.
\]
Also, $L^{\infty}(\Om_1,P_1) \iotimes \cdots \iotimes L^{\infty}(\Om_{k+1},P_{k+1}) \subseteq L^{\infty}(\Om,P)$ is a $\ast$-subalgebra, and
\[
\big(L^{\infty}(\Om_1,P_1) \iotimes \cdots \iotimes L^{\infty}(\Om_{k+1},P_{k+1}), \|\cdot\|_{L^{\infty}(P_1) \iotimes \cdots \iotimes L^{\infty}(P_{k+1})}\big)
\]
is a unital Banach $\ast$-algebra.
\end{prop}
\begin{proof}
Let $\sB \coloneqq L^{\infty}(\Om_1,P_1) \iotimes \cdots \iotimes L^{\infty}(\Om_{k+1},P_{k+1})$ and
\[
\|\cdot\|_{\sB} \coloneqq \|\cdot\|_{L^{\infty}(P_1) \iotimes \cdots \iotimes L^{\infty}(P_{k+1})}.
\]
If $\varphi \in \sB$ has $L_P^{\infty}$-IPD $(\Sigma,\rho,\varphi_1,\ldots,\varphi_{k+1})$ and
\[
\Phi(\boldsymbol{\om},\sigma) \coloneqq \varphi_1(\om_1,\sigma)\cdots\varphi_{k+1}(\om_{k+1},\sigma),
\]
then
\begin{align*}
    \|\varphi\|_{L^{\infty}(P)} & \leq \Bigg\|\int_{\Sigma}|\Phi(\cdot,\sigma)|\,\rho(d\sigma) \Bigg\|_{L^{\infty}(P)} \leq \underline{\int_{\Sigma}} \|\Phi(\cdot,\sigma)\|_{L^{\infty}(P)} \, \rho(d\sigma) \\
    & \leq \underline{\int_{\Sigma}} \|\varphi_1(\cdot,\sigma)\|_{L^{\infty}(P_1)}\cdots \|\varphi_{k+1}(\cdot,\sigma)\|_{L^{\infty}(P_{k+1})} \, \rho(d\sigma)
\end{align*}
by definition of $\Phi$, the third requirement in Definition \ref{def.IPTPspecial}, and Lemma \ref{lem.PVMink}.
Using the fact that $\underline{\int_{\Sigma}} \leq \overline{\int_{\Sigma}}$ and taking the infimum over the decompositions $(\Sigma,\rho,\varphi_1,\ldots,\varphi_{k+1})$ gives the first result.
In particular, $\|\varphi\|_{\sB} = 0 \Leftrightarrow \varphi \equiv 0$ $P$-a.e.
Now, we begin the proof that $\sB \subseteq L^{\infty}(\Om,P)$ is a $\ast$-subalgebra and that $(\sB,\|\cdot\|_{\sB})$ is a Banach $\ast$-algebra.

First, it is clear from the definition that $\sB \subseteq L^{\infty}(\Om,P)$ is closed under scalar multiplication and complex conjugation and that $\|\alpha \varphi\|_{\sB} = |\alpha| \, \|\varphi\|_{\sB} = |\alpha| \, \|\overline{\varphi}\|_{\sB}$, for all $\alpha \in \C$ and $\varphi \in \sB$.

Second, let $(\varphi_n)_{n \in \N}$ be a sequence in $\sB$ such that $\sum_{n=1}^{\infty}\|\varphi_n\|_{\sB} < \infty$.
Then
\[
\sum_{n=1}^{\infty}\|\varphi_n\|_{L^{\infty}(P)} \leq \sum_{n=1}^{\infty}\|\varphi_n\|_{\sB} < \infty,
\]
so that $\varphi \coloneqq \sum_{n=1}^{\infty}\varphi_n$ converges in $L^{\infty}(\Om,P)$.
We claim that $\|\varphi\|_{\sB} \leq \sum_{n=1}^{\infty}\|\varphi_n\|_{\sB}$, from which it follows that $\sB \subseteq L^{\infty}(\Om,P)$ is a linear subspace and $(\sB,\|\cdot\|_{\sB})$ is a Banach space.
To see this, fix $\e > 0$ and $n \in \N$.
By definition of $\|\cdot\|_{\sB}$, there exists a $L_P^{\infty}$-IPD $(\Sigma_n,\rho_n,\varphi_{1,n},\ldots,\varphi_{k+1,n})$ of $\varphi_n$ such that
\[
\overline{\int_{\Sigma_n}}\prod_{j=1}^{k+1} \|\varphi_{j,n}(\cdot,\sigma_n)\|_{L^{\infty}(P_j)} \, \rho_n(d\sigma_n) < \|\varphi_n\|_{\sB} + \frac{\e}{2^n}.
\]
This gives
\[
\sum_{n=1}^{\infty}\overline{\int_{\Sigma_n}} \prod_{j=1}^{k+1}\|\varphi_{j,n}(\cdot,\sigma_n)\|_{L^{\infty}(P_j)} \, \rho_n(d\sigma_n) \leq \sum_{n=1}^{\infty} \|\varphi_n\|_{\sB} + \e < \infty.
\]
If we redefine $(\Sigma,\sH,\rho)$ to be the disjoint union of the set of measure spaces $\{(\Sigma_n,\sH_n,\rho_n) : n \in \N\}$ and
\[
\chi_j(\om_j,\sigma) \coloneqq \varphi_{j,n}(\om_j,\sigma), \;\;\; \om_j \in \Om_j, \; \sigma \in \Sigma_n \subseteq \coprod_{\ell \in \N} \Sigma_{\ell} = \Sigma,
\]
for $j \in \{1,\ldots,k+1\}$, then we claim $(\Sigma,\rho,\chi_1,\ldots,\chi_{k+1})$ is a $L_P^{\infty}$-IPD of $\varphi$.
Indeed, item \ref{item.nullset} is easy to check.
Next, by Proposition \ref{prop.nonmeasinteg}\ref{item.disjunupint},
\[
\overline{\int_{\Sigma}}\prod_{j=1}^{k+1}\|\chi_j(\cdot,\sigma)\|_{L^{\infty}(P_j)}\,\rho(d\sigma) = \sum_{n=1}^{\infty}\overline{\int_{\Sigma_n}} \prod_{j=1}^{k+1}\|\varphi_{j,n}(\cdot,\sigma_n)\|_{L^{\infty}(P_j)} \, \rho_n(d\sigma_n) < \infty.
\]
Finally, for $P$-almost every $\boldsymbol{\om} = (\om_1,\ldots,\om_{k+1}) \in \Om$,
\[
\int_{\Sigma}\prod_{j=1}^{k+1}\chi_j(\om_j,\sigma) \, \rho(d\sigma) = \sum_{n=1}^{\infty}\int_{\Sigma_n} \prod_{j=1}^{k+1}\varphi_{j,n}(\om_j,\sigma_n) \, \rho_n(d\sigma_n) = \sum_{n=1}^{\infty}\varphi_n(\boldsymbol{\om}) = \varphi(\boldsymbol{\om}).
\]
From this, we conclude that $\varphi \in \sB$ and
\[
\|\varphi\|_{\sB} \leq \sum_{n=1}^{\infty}\overline{\int_{\Sigma_n}} \prod_{j=1}^{k+1}\|\varphi_{j,n}(\cdot,\sigma_n)\|_{L^{\infty}(P_j)} \, \rho_n(d\sigma_n) \leq \sum_{n=1}^{\infty} \|\varphi_n\|_{\sB} + \e.
\]
Taking $\e \searrow 0$ completes the proof of the claim.
\pagebreak

Third, we show that if $\varphi,\psi \in \sB$, then $\|\varphi\psi\|_{\sB} \leq \|\varphi\|_{\sB}\|\psi\|_{\sB}$, which will complete the proof of the proposition.
To this end, suppose $(\Sigma_1,\rho_1,\varphi_1,\ldots,\varphi_{k+1})$ and $(\Sigma_2,\rho_2,\psi_1,\ldots,\psi_{k+1})$ are $L_P^{\infty}$-IPDs of $\varphi$ and $\psi$, respectively.
Redefine
\begin{align*}
    (\Sigma,\sH,\rho) & \coloneqq (\Sigma_1\times \Sigma_2,\sH_1\otimes \sH_2,\rho_1 \otimes \rho_2) \; \text{ and} \\
    \chi_j(\om_j,\sigma) & \coloneqq \varphi_j(\om_j,\sigma_1) \, \psi_j(\om_j,\sigma_2), \;\;\; (\om_j,\sigma) = (\om_j,\sigma_1,\sigma_2) \in \Om_j \times \Sigma,
\end{align*}
for all $j \in \{1,\ldots,k+1\}$.
We claim $(\Sigma,\rho,\chi_1,\ldots,\chi_{k+1})$ is a $L_P^{\infty}$-IPD of (any representative of) $\varphi\psi$. Once again, item \ref{item.nullset} is clear.
Now, by Fubini's theorem and the definition of the upper integral,
\begin{align*}
     \overline{\int_{\Sigma}}\prod_{j=1}^{k+1}\|\chi_j(\cdot,\sigma)\|_{L^{\infty}(P_j)}\, \rho(d\sigma) & \leq  \overline{\int_{\Sigma_1}}\prod_{j=1}^{k+1}\|\varphi_j(\cdot,\sigma_1)\|_{L^{\infty}(P_j)} \,\rho_1(d\sigma_1) \overline{\int_{\Sigma_2}}\prod_{j=1}^{k+1}\|\psi_j(\cdot,\sigma_2)\|_{L^{\infty}(P_j)}\, \rho_2(d\sigma_2) < \infty.
\end{align*}
Finally, for $P$-almost every $\boldsymbol{\om} = (\om_1,\ldots,\om_{k+1}) \in \Om$,
\[
\varphi(\boldsymbol{\om})\, \psi(\boldsymbol{\om}) = \int_{\Sigma_1}\prod_{j=1}^{k+1}\varphi_j(\om_j,\sigma_1)\, \rho_1(d\sigma_1) \int_{\Sigma_2}\prod_{j=1}^{k+1}\psi_j(\om_j,\sigma_2) \, \rho_2(d\sigma_2) = \int_{\Sigma}\prod_{j=1}^{k+1}\chi_j(\om_j,\sigma)\, \rho(d\sigma)
\]
again by Fubini's theorem.
This proves $\varphi\psi \in \sB$ and, after taking infima, $\|\varphi\psi\|_{\sB} \leq \|\varphi\|_{\sB}\|\psi\|_{\sB}$ as well.
\end{proof}

\begin{ex}[$\ell^{\infty}$-IPTPs]\label{ex.countingPVM}
Let $(\Xi,\sG)$ be a measurable space, and write $\ell^2(\Xi) \coloneqq L^2(\Xi,2^{\Xi},\kappa)$, where $\kappa$ is the counting measure on $\Xi$.
For $G \in \sG$, let $Q(G) \in B(\ell^2(\Xi))$ be multiplication by $1_G$.
Then we call $Q \colon \sG \to B(\ell^2(\Xi))$ the \textbf{projection valued counting measure on} $\boldsymbol{(\Xi,\sG)}$.
Note that $L^{\infty}(\Xi,Q) = \ell^{\infty}(\Xi,\sG)$ with norm $\|\cdot\|_{L^{\infty}(Q)} = \|\cdot\|_{\ell^{\infty}(\Xi)}$ because $Q(G) = 0$ if and only if $G=\emptyset$.

We define
\[
\ell^{\infty}(\Om_1,\sF_1) \iotimes \cdots \iotimes \ell^{\infty}(\Om_{k+1},\sF_{k+1}) \coloneqq L^{\infty}(\Om_1,Q_1) \iotimes \cdots \iotimes L^{\infty}(\Om_{k+1},Q_{k+1}),
\]
where $Q_j$ is the projection valued counting measure on $(\Om_j,\sF_j)$, and
\[
\|\cdot\|_{\ell^{\infty}(\Om_1,\sF_1) \iotimes \cdots \iotimes \ell^{\infty}(\Om_{k+1},\sF_{k+1})} \coloneqq \|\cdot\|_{L^{\infty}(Q_1) \iotimes \cdots \iotimes L^{\infty}(Q_{k+1})}.
\]
It is easy to see that $Q \coloneqq Q_1 \otimes \cdots \otimes Q_{k+1}$ is the projection valued counting measure on $(\Om,\sF)$ when we identify $\ell^2(\Om_1) \hotimes \cdots \hotimes \ell^2(\Om_{k+1}) \cong \ell^2(\Om_1 \times \cdots \times \Om_{k+1}) = \ell^2(\Om)$.
Thus
\[
\ell^{\infty}(\Om_1,\sF_1) \iotimes \cdots \iotimes \ell^{\infty}(\Om_{k+1},\sF_{k+1}) \subseteq L^{\infty}(\Om,Q) = \ell^{\infty}(\Om,\sF).
\]
We call this space the \textbf{integral projective tensor product of} $\boldsymbol{\ell^{\infty}(\Om_1,\sF_1), \ldots , \ell^{\infty}(\Om_{k+1},\sF_{k+1})}$ and refer to $\boldsymbol{\ell^{\infty}}$\textbf{-integral projective decompositions} rather than $L^{\infty}_Q$-integral projective decompositions. 
\end{ex}

Variants of the $\ell^{\infty}$-integral projective tensor product are often used in the literature (e.g., \cite{azamovetal,depagtersukochev,doddssub2}).
As the above example shows, $\ell^{\infty}$-integral projective tensor products are special cases of $L^{\infty}$-integral projective tensor products if one allows non-separable Hilbert spaces.

\subsection{(Well-)Definition of MOIs}\label{sec.welldef}

The goal of this section is to show that if $\varphi \in L^{\infty}(\Om_1,P_1) \iotimes \cdots \iotimes L^{\infty}(\Om_{k+1},P_{k+1})$ and $(\Sigma,\rho,\varphi_1,\ldots,\varphi_{k+1})$ is a $L_P^{\infty}$-IPD of $\varphi$, then the object
\[
\int_{\Sigma} P_1(\varphi_1(\cdot,\sigma)) \, b_1 \cdots P_k(\varphi_k(\cdot,\sigma))\, b_k\,P_{k+1}(\varphi_{k+1}(\cdot,\sigma))\,\rho(d\sigma) \in B(H_{k+1};H_1)
\]
makes sense as a pointwise Pettis (in fact, weak$^*$) integral that is independent of the chosen decomposition $(\Sigma,\rho,\varphi_1,\ldots,\varphi_{k+1})$ of $\varphi$ whenever $b_j \in B(H_{j+1};H_j)$ for all $j \in \{1,\ldots,k\}$.

\begin{defi}[Complex Markov kernel]\label{def.randcompmeas}
Let $(\Xi,\sG)$ and $(\Sigma,\sH)$ be measurable spaces.
A \textbf{complex Markov kernel} (with source $\Sigma$ and target $\Xi$) is a map $\nu \colon \Sigma \to M(\Xi,\sG) = \{$complex measures on $(\Xi,\sG)\}$ such that the function $\Sigma \ni \sigma \mapsto \nu_{\sigma}(G) \coloneqq \nu(\sigma)(G) \in \C$ is $(\sH,\mathcal{B}_{\C})$-measurable, for every $G \in \sG$.
\end{defi}

\begin{lem}\label{lem.randcompmeas}
Fix measurable spaces $(\Xi,\sG)$ and $(\Sigma,\sH)$ and a complex Markov kernel $\nu \colon \Sigma \to M(\Xi,\sG)$.
If $\varphi \colon \Xi \times \Sigma \to \C$ is measurable and $\varphi(\cdot,\sigma) \in L^1(\Xi,|\nu_{\sigma}|)$, for all $\sigma \in \Sigma$, then the following function is $(\sH,\cB_{\C})$-measurable:
\[
\Sigma \ni \sigma \mapsto \int_{\Xi}\varphi(\xi,\sigma) \, \nu_{\sigma}(d\xi) \in \C.
\]
\end{lem}
\begin{proof}[Sketch \hspace{-0.25mm}of\hspace{-0.25mm} proof]\hspace{-1.25mm}
By \hspace{-0.25mm}a\hspace{-0.25mm} truncation \hspace{-0.25mm}argument,\hspace{-0.25mm} it \hspace{-0.25mm}suffices\hspace{-0.25mm} to \hspace{-0.25mm}prove\hspace{-0.25mm} the \hspace{-0.25mm}claim\hspace{-0.25mm} when \hspace{-0.25mm}$\varphi$\hspace{-0.25mm} is \hspace{-0.25mm}bounded.\hspace{-0.25mm}
To \hspace{-0.25mm}this\hspace{-0.25mm} end, \hspace{-0.25mm}let
\[
\mathbb{H} \coloneqq \Big\{\varphi \in \ell^{\infty}(\Xi \times \Sigma,\sG \otimes \sH) : \sigma \mapsto \int_{\Xi} \varphi(\xi,\sigma) \, \nu_{\sigma}(d\xi) \text{ is measurable}\Big\}.
\]
Clearly, $\mathbb{H}$ is a vector space that is closed under complex conjugation.
It is closed under sequential bounded convergence by the dominated convergence theorem.
Now, if $G \in \sG$ and $S \in \sH$, then
\[
\int_{\Xi} 1_{G \times S}(\xi,\sigma) \, \nu_{\sigma}(d\xi) = 1_S(\sigma) \, \nu_{\sigma}(G),
\]
for all $\sigma \in \Sigma$. Thus $1_{G \times S} \in \mathbb{H}$ by definition of a complex Markov kernel.
By the multiplicative system theorem (in the form of Corollary \ref{cor.MST}), we conclude that $\ell^{\infty}(\Xi \times \Sigma,\sG \otimes \sH) \subseteq \mathbb{H}$, as desired.
\end{proof}

\begin{prop}\label{prop.opinteggood}
Fix a projection valued measure space $(\Xi,\sG,K,Q)$, another complex Hilbert space $(L,\la \cdot,\cdot\ra_L)$, a measure space $(\Sigma,\sH,\rho)$, and a measurable function $\varphi \colon \Xi \times \Sigma \to \C$.
If $\varphi(\cdot,\sigma) \in L^{\infty}(\Xi,Q)$, for all $\sigma \in \Sigma$, and $A \colon \Sigma \to B(K;L)$ and $B \colon \Sigma \to B(L;K)$ are pointwise weakly measurable, then the maps
\[
\Sigma \ni \sigma \mapsto A(\sigma)\,Q(\varphi(\cdot,\sigma)) \in B(K;L) \; \text{ and } \; \Sigma \ni \sigma \mapsto Q(\varphi(\cdot,\sigma))\,B(\sigma) \in B(L;K)
\]
are pointwise weakly measurable as well.
If in addition
\[
\underline{\int_{\Sigma}} \|A(\sigma)\|_{B(K;L)} \|\varphi(\cdot,\sigma)\|_{L^{\infty}(Q)} \, \rho(d\sigma) + \underline{\int_{\Sigma}} \|\varphi(\cdot,\sigma)\|_{L^{\infty}(Q)}\|B(\sigma)\|_{B(L;K)} \, \rho(d\sigma) < \infty,
\]
then the aforementioned maps are pointwise Pettis (in fact, weak$^*$) integrable.
\end{prop}
\begin{proof}
Fix $\sigma \in \Sigma$, $k \in K$, and $l \in L$.
Then
\[
\big\la A(\sigma)\,Q(\varphi(\cdot,\sigma))\,k,l \big\ra_L = \big\la Q(\varphi(\cdot,\sigma))\,k,A(\sigma)^*l \big\ra_K = \int_{\Xi}\varphi(\xi,\sigma) \, Q_{k,A(\sigma)^*l}(d\xi).
\]
But
\[
\nu_{\sigma}^A(G) \coloneqq Q_{k,A(\sigma)^*l}(G) = \la Q(G)k,A(\sigma)^*l \ra_K = \la A(\sigma)Q(G)k,l \ra_L
\]
defines a complex Markov kernel $\nu^A \colon \Sigma \to M(\Xi,\sG)$ by the pointwise weak measurability of $A$.
We conclude from Lemma \ref{lem.randcompmeas} that $A(\cdot) \int_{\Xi}\varphi(\xi,\cdot) \, Q(d\xi)$ is pointwise weakly measurable.
A similar argument proves that $\int_{\Xi}\varphi(\xi,\cdot) \, Q(d\xi) \, B(\cdot)$ is pointwise weakly measurable.

The second part follows from Corollary \ref{cor.wstarint} because, for all $\sigma \in \Sigma$, we have
\begin{align*}
    \big\|A(\sigma)\,Q(\varphi(\cdot,\sigma))\big\|_{B(K;L)} & \leq \|A(\sigma)\|_{B(K;L)}\|Q(\varphi(\cdot,\sigma))\|_{B(K)} =  \|A(\sigma)\|_{B(K;L)}\|\varphi(\cdot,\sigma)\|_{L^{\infty}(Q)} \; \text{ and}\\
    \big\|Q(\varphi(\cdot,\sigma))\, B(\sigma)\big\|_{B(L;K)} & \leq \|Q(\varphi(\cdot,\sigma))\|_{B(K)}\|B(\sigma)\|_{B(L;K)} = \|\varphi(\cdot,\sigma)\|_{L^{\infty}(Q)}\|B(\sigma)\|_{B(L;K)}. \qedhere
\end{align*}
\end{proof}

\begin{cor}\label{cor.MOIintgood}
Fix $\varphi \in L^{\infty}(\Om_1,P_1) \iotimes \cdots \iotimes L^{\infty}(\Om_{k+1},P_{k+1})$, a $L_P^{\infty}$-IPD $(\Sigma,\rho,\varphi_1,\ldots,\varphi_{k+1})$ of $\varphi$, and $b_1 \in B(H_2;H_1),\ldots, b_k \in B(H_{k+1};H_k)$. If
\begin{align*}
    F(\sigma) & \coloneqq P_1(\varphi_1(\cdot,\sigma))\,b_1 \cdots P_k(\varphi_k(\cdot,\sigma)) \, b_k \, P_{k+1}(\varphi_{k+1}(\cdot,\sigma)) \in B(H_{k+1};H_1) \\
    & = \int_{\Om_1}\varphi_1(\cdot,\sigma) \, dP_1\, b_1 \cdots \int_{\Om_k}\varphi_k(\cdot,\sigma) \, dP_k \,b_k \int_{\Om_{k+1}}\varphi_{k+1}(\cdot,\sigma)\,dP_{k+1} 
\end{align*}
for $\sigma \in \Sigma$, then $F \colon \Sigma \to B(H_{k+1};H_1)$ is pointwise Pettis (in fact, weak$^*$) integrable and
\[
\Bigg\|\int_{\Sigma}F \,d\rho\Bigg\|_{B(H_{k+1};H_1)} \leq \Bigg(\prod_{j=1}^k\|b_j\|_{B(H_{j+1};H_j)}\Bigg)\underline{\int_{\Sigma}}\prod_{m=1}^{k+1}\|\varphi_m(\cdot,\sigma)\|_{L^{\infty}(P_m)}\,\rho(d\sigma). \numberthis\label{eq.firstbound}
\]
\end{cor}
\begin{proof}
An inductive application of the first part of Proposition \ref{prop.opinteggood} implies that $F$ is pointwise weakly measurable.
The fact that
\[
\|F(\sigma)\|_{B(H_{k+1};H_1)} \leq \Bigg(\prod_{j=1}^k\|b_j\|_{B(H_{j+1};H_j)}\Bigg)\prod_{m=1}^{k+1}\|\varphi_m(\cdot,\sigma)\|_{L^{\infty}(P_m)}
\]
allows us to conclude from Corollary \ref{cor.wstarint} that $F$ is pointwise Pettis (in fact, weak$^*$) integrable.
The bound \eqref{eq.firstbound} then follows from the triangle inequality in Theorem \ref{thm.wintBHK}\ref{item.wintchar}.
\end{proof}

\begin{nota}[MOI, take I]\label{nota.MOIdepondecomp}
Let $\varphi \in L^{\infty}(\Om_1,P_1) \iotimes \cdots \iotimes L^{\infty}(\Om_{k+1},P_{k+1})$ and $(\Sigma,\rho,\varphi_1,\ldots,\varphi_{k+1})$ be a $L_P^{\infty}$-IPD of $\varphi$.
Write $\boldsymbol{P} = (P_1,\ldots,P_{k+1})$ and, whenever $b = (b_1,\ldots,b_k) \in B(H_2;H_1) \times \cdots \times B(H_{k+1};H_k)$, define $I^{\boldsymbol{P}}(\Sigma,\rho,\varphi_1,\ldots,\varphi_{k+1})[b]$ to be the following pointwise Pettis (weak$^*$) integral, from Corollary \ref{cor.MOIintgood}:
\[
\int_{\Sigma} P_1(\varphi_1(\cdot,\sigma))\,b_1 \cdots P_k(\varphi_k(\cdot,\sigma)) \, b_k \, P_{k+1}(\varphi_{k+1}(\cdot,\sigma)) \, \rho(d\sigma) \in B(H_{k+1};H_1).
\]
\end{nota}

By linearity of the integral and \eqref{eq.firstbound}, the map
\[
I^{\boldsymbol{P}}(\Sigma,\rho,\varphi_1,\ldots,\varphi_{k+1}) \colon B(H_2;H_1) \times \cdots \times B(H_{k+1};H_k) \to B(H_{k+1};H_1)
\]
is $k$-linear and bounded.\hspace{-0.2mm}
Our next\hspace{0.4mm}---\hspace{0.4mm}and most important\hspace{0.4mm}---\hspace{0.4mm}task is to prove this map is ultraweakly continuous in each argument.
As described in Section \ref{sec.keying}, this is quite technical when $H_1,\ldots,H_{k+1}$ are not separable.

\begin{lem}\label{lem.Pintcont}
Suppose that $(\Xi,\sG,K,Q)$ is a projection valued measure space.
If $(\varphi_n)_{n \in \N} \in L^{\infty}(\Xi,Q)^{\N}$, $\sup_{n \in \N}\|\varphi_n\|_{L^{\infty}(Q)} < \infty$, and $\varphi_n \to \varphi \in L^{\infty}(\Xi,Q)$ pointwise $Q$-a.e. as $n \to \infty$, then $Q(\varphi_n) \to Q(\varphi)$ in the S$^*$OT as $n \to \infty$.
\end{lem}
\begin{proof}
Since $Q$ is linear and $Q(\psi)^* = Q(\overline{\psi})$, for all $\psi \in L^{\infty}(\Xi,Q)$, it suffices to assume $\varphi \equiv 0$ and, in this case, to prove $Q(\varphi_n) \to 0$ in the SOT as $n \to \infty$.
To this end, let $h \in K$.
Then the dominated convergence theorem gives, as $n \to \infty$,
\[
\|Q(\varphi_n)h\|_K^2 = \la Q(\varphi_n)^*Q(\varphi_n)h,h\ra_K = \la Q(|\varphi_n|^2)h,h \ra_K = \int_{\Xi} |\varphi_n|^2 \, dQ_{h,h} \to 0. \qedhere
\]
\end{proof}

\begin{lem}\label{lem.SOTS1}
Let $K$, $L$ be Hilbert spaces, $(Q_n)_{n \in \N} \in B(K;L)^{\N}$, $Q \in B(K;L)$, and $C \in \cS_1(L;K)$.
If $Q_n \to Q$ in the SOT, then $Q_nC \to QC$ in $\cS_1(L)$ as $n \to \infty$.
If $Q_n^* \to Q^*$ in the SOT, then $CQ_n \to CQ$ in $\cS_1(K)$ as $n \to \infty$.
In particular, if $Q_n \to Q$ in the S$^*$OT, then $Q_nC \to QC$ in $\cS_1(L)$ and $CQ_n \to CQ$ in $\cS_1(K)$ as $n \to \infty$.
\end{lem}
\begin{proof}
Without loss of generality, we can take $Q=0$.
Assume $Q_n \to 0$ in the SOT as $n \to \infty$, and let $D \in \cS_2(L;K)$ and $\mathcal{E} \subseteq L$ be an orthonormal basis.
Then $\|Q_nD\|_{\cS_2}^2 = \sum_{e \in \mathcal{E}} \|Q_nDe\|_L^2 \to 0$ as $n \to \infty$ by the dominated convergence theorem.
More explicitly, we have $\|Q_nDe\|_L \to 0$ as $n \to \infty$, for all $e \in \mathcal{E}$; and $\|Q_nDe\|_L^2 \leq \|De\|_K^2 \sup_{j \in \N} \|Q_j\|_{B(K;L)}^2 \in L^1(\mathcal{E},$ counting$)$ because $D \in \cS_2(L;K)$.
Next, assume $Q_n^* \to 0$ in the SOT as $n \to \infty$.
Then $\|DQ_n\|_{\cS_2} = \|Q_n^*D^*\|_{\cS_2} \to 0$ as $n \to \infty$ by the previous argument.
Finally, let $C = U|C|$ be the polar decomposition of $C \in \cS_1(L;K)$.
Then $|C|^{\frac{1}{2}} \in \cS_2(L)$ and $U|C|^{\frac{1}{2}} \in \cS_2(L;K)$, so
\begin{align*}
    \|Q_nC\|_{\cS_1} & = \big\|Q_nU|C|^{\frac{1}{2}}|C|^{\frac{1}{2}}\big\|_{\cS_1} \leq \big\|Q_nU|C|^{\frac{1}{2}}\big\|_{\cS_2}\big\||C|^{\frac{1}{2}}\big\|_{\cS_2} \to 0 \; \text{ or} \\
    \|CQ_n\|_{\cS_1} & = \big\|U|C|^{\frac{1}{2}}|C|^{\frac{1}{2}}Q_n\big\|_{\cS_1} \leq \big\|U|C|^{\frac{1}{2}}\big\|_{\cS_2} \big\||C|^{\frac{1}{2}}Q_n\big\|_{\cS_2} \to 0
\end{align*}
if $Q_n \to 0$ or $Q_n^* \to 0$ in the SOT as $n \to \infty$, respectively, by H\"{o}lder's inequality for the Schatten norms and what we just proved.
\end{proof}

We are now finally prepared to prove Theorem \ref{thm.traceofinteg}.

\begin{proof}[Proof \hspace{-0.25mm}of\hspace{-0.25mm} Theorem \ref{thm.traceofinteg}]\hspace{-1.25mm}
The \hspace{-0.25mm}first\hspace{-0.25mm} conclusion \hspace{-0.25mm}is\hspace{-0.25mm} clear \hspace{-0.25mm}from\hspace{-0.25mm} Proposition \hspace{-0.25mm}\ref{prop.opinteggood}\hspace{-0.25mm} (to \hspace{-0.25mm}see\hspace{-0.25mm} that $A(\cdot)\int_{\Xi} \varphi(\xi,\cdot)\,Q(d\xi)$ and $\int_{\Xi} \varphi(\xi,\cdot)\,Q(d\xi)\,A(\cdot)$ are pointwise Pettis integrable) and Theorem \ref{thm.Spinteg} (to see that their pointwise Pettis integrals belong to $\cS_1(K)$).
For the second, we use the multiplicative system theorem.
Let
\[
\mathbb{H} \coloneqq \{\varphi \in \ell^{\infty}(\Xi \times \Sigma, \sG \otimes \sH) : \eqref{eq.traceofinteg} \text{ holds for } A \text{ and } A^*\}.
\]
(Note $A^*$ and $A$ satisfy the same hypotheses.)
Clearly $\mathbb{H}$ is a vector space.
To show $\mathbb{H}$ is closed under complex conjugation, note that if $\varphi \in \mathbb{H}$ and $B \in \{A,A^*\}$, then
\begin{align*}
    \Tr\Bigg(\int_{\Sigma} B(\sigma) \, Q\Big(\overline{\varphi(\cdot,\sigma)}\Big) \,\rho(d\sigma)\Bigg) & = \Tr\Bigg(\Bigg(\int_{\Sigma} Q(\varphi(\cdot,\sigma)) \, B(\sigma)^*  \,\rho(d\sigma)\Bigg)^*\Bigg) = \overline{\Tr\Bigg(\int_{\Sigma} Q(\varphi(\cdot,\sigma)) \, B(\sigma)^*  \,\rho(d\sigma)\Bigg)} \\
    & = \overline{\Tr\Bigg(\int_{\Sigma} B(\sigma)^* \, Q(\varphi(\cdot,\sigma)) \, \rho(d\sigma)\Bigg)} =  \Tr\Bigg(\Bigg(\int_{\Sigma} B(\sigma)^* \, Q(\varphi(\cdot,\sigma)) \, \rho(d\sigma)\Bigg)^*\Bigg) \\
    & = \Tr\Bigg(\int_{\Sigma} Q\Big(\overline{\varphi(\cdot,\sigma)}\Big)\,B(\sigma)\,\rho(d\sigma)  \Bigg).
\end{align*}
Therefore, $\overline{\varphi} \in \mathbb{H}$.
Next, we show $\mathbb{H}$ is closed under sequential bounded convergence.
Suppose $(\varphi_n)_{n \in \N} \in \mathbb{H}^{\N}$ converges boundedly to $\varphi \in \ell^{\infty}(\Xi \times \Sigma, \sG \otimes \sH)$. We claim that
\begin{align*}
    & \int_{\Sigma} B(\sigma) \, Q(\varphi_n(\cdot,\sigma)) \,\rho(d\sigma) \to \int_{\Sigma} B(\sigma)  \, Q(\varphi(\cdot,\sigma)) \,\rho(d\sigma) \; \text{ and}\\
    & \int_{\Sigma} Q(\varphi_n(\cdot,\sigma)) \, B(\sigma)  \,\rho(d\sigma) \to \int_{\Sigma} Q(\varphi(\cdot,\sigma)) \, B(\sigma)  \,\rho(d\sigma) 
\end{align*}
in $\cS_1(K)$ as $n \to \infty$ for $B \in \{A,A^*\}$.
Indeed, let $\sigma \in \Sigma$.
By Lemma \ref{lem.Pintcont}, $Q(\varphi_n(\cdot,\sigma)) \to Q(\varphi(\cdot,\sigma))$ in the S$^*$OT as $n \to \infty$.
By Lemma \ref{lem.SOTS1},
\begin{align*}
    & B(\sigma) \,Q(\varphi_n(\cdot,\sigma))  \to B(\sigma)  \,Q(\varphi(\cdot,\sigma)) \; \text{ and } \; Q(\varphi_n(\cdot,\sigma)) \, B(\sigma)  \to Q(\varphi(\cdot,\sigma)) \, B(\sigma) 
\end{align*}
in $\cS_1(K)$ as $n \to \infty$.
Finally, note that
\[
\overline{\int_{\Sigma}}\sup_{n \in \N}\big\|B(\sigma) \,Q(\varphi_n(\cdot,\sigma)) \big\|_{\cS_1}\,\rho(d\sigma)+\overline{\int_{\Sigma}}\sup_{n \in \N}\big\| Q(\varphi_n(\cdot,\sigma)) \,B(\sigma)\big\|_{\cS_1} \,\rho(d\sigma)
\]
is bounded above by $2\rho(\Sigma)\sup_{\sigma \in \Sigma}\|B(\sigma)\|_{\cS_1}\,\sup_{n \in \N}\|\varphi_n\|_{\ell^{\infty}(\Xi \times \Sigma)} < \infty$.
The claim then follows from Corollary \ref{cor.SchattenDCT}.
This establishes \eqref{eq.traceofinteg} for $\varphi$ and $B$ by allowing us to take $n \to \infty$ in \eqref{eq.traceofinteg} for $\varphi_n$ and $B$.
Thus $\varphi \in \mathbb{H}$, as desired.

Finally, if $G \in \sG$, $S \in \sH$, $\varphi \coloneqq 1_{G \times S}$, and $B \in \{A,A^*\}$, then
\begin{align*}
    \Tr\Bigg(\int_{\Sigma} B(\sigma)\, Q(\varphi(\cdot,\sigma))\,\rho(d\sigma)\Bigg) & =  \Tr\Bigg(\int_{\Sigma} B(\sigma) \,1_S(\sigma) \,Q(G) \,\rho(d\sigma)\Bigg)  = \Tr\Bigg(\int_{\Sigma} 1_S(\sigma) \,B(\sigma) \, \rho(d\sigma) \,Q(G) \Bigg) \\
    & = \Tr\Bigg(Q(G)\int_{\Sigma} 1_S(\sigma) \,B(\sigma) \, \rho(d\sigma) \Bigg) = \Tr\Bigg(\int_{\Sigma} Q(G)\,1_S(\sigma) \,B(\sigma) \, \rho(d\sigma) \Bigg) \\
    & = \Tr\Bigg( \int_{\Sigma} Q(\varphi(\cdot,\sigma)) \,B(\sigma) \,\rho(d\sigma)\Bigg)
\end{align*}
by Theorem \ref{thm.Schatten}\ref{item.Trflip}.
Thus $\varphi \in \mathbb{H}$.
We conclude from the multiplicative system theorem (Corollary \ref{cor.MST}) that $\ell^{\infty}(\Xi \times \Sigma, \sG \otimes \sH) \subseteq \mathbb{H}$, as desired.
\end{proof}
\begin{rem}\label{rem.easyTrsep}
There is an easy argument available when $K$ is separable.
Indeed, since any orthonormal basis is countable, one can show that if $(\Sigma,\sH,\rho)$ is any measure space and $A \colon \Sigma \to \cS_1(K) \subseteq B(K)$ is pointwise weakly measurable with $\underline{\int_{\Sigma}} \|A(\sigma)\|_{\cS_1}\,\rho(d\sigma) < \infty$, then $\Tr\big(\int_{\Sigma}A(\sigma) \, \rho(d\sigma)\big) = \int_{\Sigma}\Tr(A(\sigma)) \, \rho(d\sigma)$.
The result (actually, a more general result) therefore follows from pointwise application of Theorem \ref{thm.Schatten}\ref{item.Trflip}.
\end{rem}

\begin{thm}\label{thm.S1integrability}
Fix $\varphi \in L^{\infty}(\Om_1,P_1) \iotimes \cdots \iotimes L^{\infty}(\Om_{k+1},P_{k+1})$, a $L_P^{\infty}$-integral projective decomposition $(\Sigma,\rho,\varphi_1,\ldots,\varphi_{k+1})$ of $\varphi$, and
\[
b = (b_1,\ldots,b_k) \in B(H_2;H_1) \times \cdots \times B(H_{k+1};H_k).
\pagebreak
\]
If $b_j \in \cS_1(H_{j+1};H_j)$ for some $j \in \{1,\ldots,k\}$, then
\[
\Sigma \ni \sigma \mapsto P_1(\varphi_1(\cdot,\sigma))\,b_1\cdots P_k(\varphi_k(\cdot,\sigma))\,b_k\,P_{k+1}(\varphi_{k+1}(\cdot,\sigma)) \in \cS_1(H_{k+1};H_1)
\]
is Gel'fand--Pettis integrable as a map $\Sigma \to (\cS_1(H_{k+1};H_1),\|\cdot\|_{\cS_1})$, and its Gel'fand--Pettis integral (over $\Sigma$) is $I^{\boldsymbol{P}}(\Sigma,\rho,\varphi_1,\ldots,\varphi_{k+1})[b] \in \cS_1(H_{k+1};H_1)$.
\end{thm}
\begin{proof}
First, notice that $I^{\boldsymbol{P}}(\Sigma,\rho,\varphi_1,\ldots,\varphi_{k+1})[b] \in \cS_1(H_{k+1};H_1)$ by Theorem \ref{thm.Spinteg}.
Now, for $\sigma \in \Sigma$, write
\[
F(\sigma) \coloneqq P_1(\varphi_1(\cdot,\sigma))\,b_1\cdots P_k(\varphi_k(\cdot,\sigma))\,b_k\,P_{k+1}(\varphi_{k+1}(\cdot,\sigma)) \in \cS_1(H_{k+1};H_1).
\]
By definition of Gel'fand--Pettis integrals and the identification (Theorem \ref{thm.Schatten}\ref{item.Schattendual}) of $B(H_1;H_{k+1})$ as $\cS_1(H_{k+1};H_1)^*$, we must show that if $b_{k+1} \in B(H_1;H_{k+1})$, then:
\begin{enumerate}[label=(\arabic*), leftmargin=2\parindent]
    \item $\Sigma \ni \sigma \mapsto \Tr(F(\sigma)\,b_{k+1}) \in \C$ is measurable;\label{item.Trmeas}
    \item $\int_{\Sigma}|\Tr(F(\sigma)\,b_{k+1})|\,\rho(d\sigma) < \infty$; and\label{item.Trint}
    \item $\Tr\big(I^{\boldsymbol{P}}(\Sigma,\rho,\varphi_1,\ldots,\varphi_{k+1})[b]\,b_{k+1}\big) = \int_{\Sigma} \Tr(F(\sigma)\,b_{k+1})\,\rho(d\sigma)$.\label{item.Trwint}
\end{enumerate}
This suffices because if $S \in \sH$, then we can apply the above to
\[
(S,\rho|_S,\varphi_1|_{\Om_1 \times S},\ldots,\varphi_{k+1}|_{\Om_{k+1} \times S})
\]
to conclude that the $\cS_1$-weak integral over $S$ of $F$ is
\[
I^{\boldsymbol{P}}(S,\rho|_S,\varphi_1|_{\Om_1 \times S},\ldots,\varphi_{k+1}|_{\Om_{k+1} \times S})[b].
\]
We take each item in turn.

\ref{item.Trmeas} If $\sigma \in \Sigma$, then
\begin{align*}
    \Tr(F(\sigma)\,b_{k+1}) & = \Tr(P_1(\varphi_1(\cdot,\sigma))\,b_1\cdots P_k(\varphi_k(\cdot,\sigma))\,b_k\,P_{k+1}(\varphi_{k+1}(\cdot,\sigma))\,b_{k+1}) \\
    & = \Tr(\underbrace{P_{j+1}(\varphi_{j+1}(\cdot,\sigma))\,b_{j+1}\cdots P_{k+1}(\varphi_{k+1}(\cdot,\sigma))\,b_{k+1}\,P_1(\varphi_1(\cdot,\sigma))\,b_1\cdots P_j(\varphi_j(\cdot,\sigma))}_{\coloneqq F_j(\sigma)}\,b_j)
\end{align*}
by Theorem \ref{thm.Schatten}\ref{item.Trflip}. But $F_j \colon \Sigma \to B(H_{j+1};H_j)$ is pointwise Pettis integrable by Corollary \ref{cor.MOIintgood}, and $b_j \in \cS_1(H_{j+1};H_j)$.
Thus $\Tr(F(\cdot)\,b_{k+1}) = \Tr(F_j(\cdot)\,b_j)$ is measurable by Theorem \ref{thm.wintBHK}\ref{item.wmeaschar}.

\ref{item.Trint} We have
\begin{align*}
    \int_{\Sigma} |\Tr(F(\sigma)\,b_{k+1})|\,\rho(d\sigma) & \leq \underline{\int_{\Sigma}} \|F(\sigma)\,b_{k+1}\|_{\cS_1}\,\rho(d\sigma) \leq \|b_{k+1}\|\underline{\int_{\Sigma}} \|F\|_{\cS_1}\,d\rho \\
    & \leq \|b_j\|_{\cS_1}\prod_{p \neq j}\|b_p\| \underline{\int_{\Sigma}}\prod_{j_0=1}^{k+1} \|\varphi_{j_0}(\cdot,\sigma)\|_{L^{\infty}(P_{j_0})}\,\rho(d\sigma) < \infty,
\end{align*}
as desired.
(Note that we have been sloppy with our notation for different operator norms above.)

\ref{item.Trwint} This step is the most involved;
we reduce it to the result of Theorem \ref{thm.traceofinteg}.
First, let $(\Sigma_n)_{n \in \N} \in \sH^{\N}$ be a countable partition of $\Sigma$ by sets of finite $\rho$-measure.
Then
\[
\int_{\Sigma} \Tr(F(\sigma)\,b_{k+1})\,\rho(d\sigma) = \sum_{n=1}^{\infty}\int_{\Sigma_n} \Tr(F(\sigma)\,b_{k+1})\,\rho(d\sigma)
\]
by the last step.
Now, if $F_0 \colon \Sigma \to B(H_{k+1};H_1)$ is pointwise Pettis integrable, then 
\[
\int_{\Sigma} F_0 \, d\rho = \text{WOT-}\sum_{n=1}^{\infty}\int_{\Sigma_n}F_0 \, d\rho \; \text{ and } \; \sum_{n=1}^{\infty} \Bigg\|\int_{\Sigma_n}F_0 \, d\rho\Bigg\|_{\cS_1} \leq \sum_{n=1}^{\infty} \underline{\int_{\Sigma_n}} \|F_0\|_{\cS_1} \, d\rho = \underline{\int_{\Sigma}}\|F_0\|_{\cS_1}\,d\rho
\]
by definition of the pointwise Pettis integral, Theorem \ref{thm.Spinteg}, and Proposition \ref{prop.nonmeasinteg}\ref{item.disjunupint}.
This implies
\[
\Tr\big(I^{\boldsymbol{P}}(\Sigma,\rho,\varphi_1,\ldots,\varphi_{k+1})[b]\,b_{k+1}\big) = \sum_{n=1}^{\infty} \Tr\big(I^{\boldsymbol{P}}(\Sigma_n,\rho|_{\Sigma_n},\varphi_1|_{\Om_1 \times \Sigma_n},\ldots,\varphi_{k+1}|_{\Om_{k+1} \times \Sigma_n})[b]\,b_{k+1}\big).
\]
It therefore suffices to assume $\rho(\Sigma) < \infty$.

Second, let $m \in \{1,\ldots,k+1\}$, $n \in \N$,
\[
\varphi_{m,n} \coloneqq \varphi_m \,1_{\{(\om_m,\sigma)\in\Om_m \times \Sigma : |\varphi_m(\om_m,\sigma)| \leq n\}},\pagebreak
\]
and $\sigma \in \Sigma$.
Then $\varphi_{m,n}$ is bounded, $\varphi_{m,n}(\cdot,\sigma) \to \varphi_m(\cdot,\sigma)$ $P_m$-almost everywhere as $n \to \infty$, and
\[
\sup_{n \in \N}\|\varphi_{m,n}(\cdot,\sigma)\|_{L^{\infty}(P_m)} \leq \|\varphi_m(\cdot,\sigma)\|_{L^{\infty}(P_m)} < \infty.
\]
Thus $P_m(\varphi_{m,n}(\cdot,\sigma)) \to P_m(\varphi_m(\cdot,\sigma))$ in the S$^*$OT as $n \to \infty$ by Lemma \ref{lem.Pintcont}.
Now, let
\begin{align*}
    & (\psi_{1,n},\ldots,\psi_{k+1,n}) \coloneqq (\varphi_1,\ldots,\varphi_{m-1},\varphi_{m,n},\varphi_{m+1},\ldots,\varphi_{k+1}) \; \text{ and} \\
    & F_{m,n}(\sigma) \coloneqq P_1(\psi_{1,n}(\cdot,\sigma))\,b_1\cdots P_k(\psi_{k,n}(\cdot,\sigma))\,b_k\,P_{k+1}(\psi_{k+1,n}(\cdot,\sigma)) \in B(H_{k+1};H_k). 
\end{align*}
Since $b_j \in \cS_1(H_{j+1};H_j)$, Lemma \ref{lem.SOTS1} gives that $F_{m,n}(\sigma) \to F(\sigma)$ in $\cS_1(H_{k+1};H_1)$ as $n \to \infty$.
But
\[
\overline{\int_{\Sigma}}\sup_{n \in \N}\|F_{m,n}\|_{\cS_1}\,d\rho \leq \|b_j\|_{\cS_1}\prod_{p \not\in\{j,k+1\}}\|b_p\|\overline{\int_{\Sigma}}\prod_{j_0=1}^{k+1}\|\varphi_{j_0}(\cdot,\sigma)\|_{L^{\infty}(P_{j_0})}\,\rho(d\sigma) < \infty.
\]
Therefore, by Corollary \ref{cor.SchattenDCT},
\[
I^{\boldsymbol{P}}(\Sigma,\rho,\psi_{1,n},\ldots,\psi_{k+1,n})[b] = \int_{\Sigma} F_{m,n}\,d\rho \to \int_{\Sigma} F\,d\rho = I^{\boldsymbol{P}}(\Sigma,\rho,\varphi_1,\ldots,\varphi_{k+1})[b] \numberthis\label{eq.DCTapp}
\]
in $\cS_1(H_{k+1};H_1)$ as $n \to \infty$, and, by the dominated convergence theorem,
\[
\int_{\Sigma} \Tr(F_{m,n}(\sigma)\,b_{k+1})\,\rho(d\sigma) \to \int_{\Sigma}\Tr(F(\sigma)\,b_{k+1})\,\rho(d\sigma)
\]
as $n \to \infty$.
It therefore suffices to assume $\varphi_m$ is bounded.

Finally, assume $\rho(\Sigma) < \infty$ and $\varphi_1,\ldots,\varphi_{k+1}$ are all bounded.
Also, retain the definition of the map $F_j \colon \Sigma \to B(H_{j+1};H_j)$ from the first step of the proof.
Then
\begin{align*}
    \Tr\big(I^{\boldsymbol{P}}(\Sigma,\rho,\varphi_1,\ldots,\varphi_{k+1})[b]\, b_{k+1}\big) & = \Tr\Bigg(\int_{\Sigma}\prod_{j_0=1}^{k+1}P_{j_0}(\varphi_{j_0}(\cdot,\sigma))\,b_{j_0}\, \rho(d\sigma)\Bigg)\\
    & = \Tr\Bigg(\int_{\Sigma}\prod_{j_1=1}^j P_{j_1}(\varphi_{j_1}(\cdot,\sigma))  \, b_{j_1}\prod_{j_2=j+1}^{k+1} P_{j_2}(\varphi_{j_2}(\cdot,\sigma)) \, b_{j_2} \, \rho(d\sigma)\Bigg) \\
    & = \Tr\Bigg(\int_{\Sigma}\prod_{j_2=j+1}^{k+1} P_{j_2}(\varphi_{j_2}(\cdot,\sigma)) \,b_{j_2} \prod_{j_1=1}^j P_{j_1}(\varphi_{j_1}(\cdot,\sigma)) \, b_{j_1}\, \rho(d\sigma)\Bigg) \numberthis\label{eq.manip1} \\
    & = \Tr\Bigg(\Bigg(\int_{\Sigma}F_j\,d\rho\Bigg)\, b_j\Bigg) = \int_{\Sigma} \underbrace{\Tr(F_j(\sigma)\,b_j)}_{=\Tr(F(\sigma)\,b_{k+1})}\,\rho(d\sigma), \numberthis\label{eq.manip2}
\end{align*}
where \eqref{eq.manip1} follows from repeated, alternating applications of Theorems \ref{thm.traceofinteg} and \ref{thm.Schatten}\ref{item.Trflip};
and \eqref{eq.manip2} follows from the $\sigma$-weak continuity of $c \mapsto \Tr(cb_j)$ and the calculation from the first step.
\end{proof}
\begin{rem}
By inspection of the proof, justifying \eqref{eq.manip1} is where essentially all the difficulty lies.
Using the generalization of Theorem \ref{thm.traceofinteg} for separable $K$ alluded to in Remark \ref{rem.easyTrsep}, \eqref{eq.manip1} can be justified much more easily and directly (i.e., without need to go through all the reductions from the third step of the proof) when $H_1,\ldots,H_{k+1}$ are separable.
\end{rem}

It is interesting to note that the ``truncation argument" from the third step of the proof above is the only place in this section where we have actually used that the \textit{upper} (as opposed to the lower) $\rho$-integral of $\sigma \mapsto \prod_{j=1}^{k+1}\|\varphi_j(\cdot,\sigma)\|_{L^{\infty}(P_j)}$ is finite.
Specifically, we only used it to prove \eqref{eq.DCTapp}.
In all other arguments, finiteness of the lower integral sufficed.

\begin{cor}[Ultraweak continuity of MOI]\label{cor.uwcont}
Let $\varphi \,\in\, L^{\infty}(\Om_1,\,P_1) \,\iotimes\, \cdots \,\iotimes\, L^{\infty}(\Om_{k+1},\,P_{k+1})$ and $(\Sigma,\rho,\varphi_1,\ldots,\varphi_{k+1})$ be a $L_P^{\infty}$-IPD of $\varphi$.
If
\[
b = (b_1,\ldots,b_k) \in B(H_2;H_1) \times \cdots \times B(H_{k+1};H_k),
\]
$b_{k+1} \in \cS_1(H_1;H_{k+1})$, and $\gamma \in S_{k+1}$ is a cyclic permutation of $\{1,\ldots,k+1\}$, then
\[
\Tr\big(I^{\boldsymbol{P}}(\Sigma,\rho,\varphi_1,\ldots,\varphi_{k+1})[b]\,b_{k+1}\big) = \Tr\big(I^{\boldsymbol{P}_{\gamma}}(\Sigma,\rho,\varphi_{\gamma(1)},\ldots,\varphi_{\gamma(k+1)})[b_{\gamma}]\,b_{\gamma(k+1)}\big),\pagebreak
\]
where $\boldsymbol{P}_{\gamma} = (P_{\gamma(1)},\ldots,P_{\gamma(k+1)})$ and $b_{\gamma} = (b_{\gamma(1)},\ldots,b_{\gamma(k)})$.
In particular, the map
\[
I^{\boldsymbol{P}}(\Sigma,\rho,\varphi_1,\ldots,\varphi_{k+1}) \colon B(H_2;H_1) \times \cdots \times B(H_{k+1};H_k) \to B(H_{k+1};H_1)
\]
is ($k$-linear and) ultraweakly continuous in each argument. 
\end{cor}
\begin{proof}
Fix $b \in B(H_2;H_1) \times \cdots \times B(H_{k+1};H_k)$, $b_{k+1} \in \cS_1(H_1;H_{k+1})$, and $j \in \{1,\ldots,k\}$.
Let $\gamma \in S_{k+1}$ be the cyclic permutation $j_0 \mapsto j_0+j$ mod $k+1$, and define
\begin{align*}
    F(\sigma) & \coloneqq P_1(\varphi_1(\cdot,\sigma))\,b_1\cdots P_k(\varphi_k(\cdot,\sigma))\,b_k\,P_{k+1}(\varphi_{k+1}(\cdot,\sigma)) \in B(H_{k+1};H_1) \; \text{ and} \\
    F_j(\sigma) & \coloneqq P_{j+1}(\varphi_{j+1}(\cdot,\sigma))\,b_{j+1}\cdots P_{k+1}(\varphi_{k+1}(\cdot,\sigma))\,b_{k+1}\,P_1(\varphi_1(\cdot,\sigma))\,b_1\cdots P_{j-1}(\varphi_{j-1}(\cdot,\sigma))\,b_{j-1}\,P_j(\varphi_j(\cdot,\sigma)) \\
    & = P_{\gamma(1)}(\varphi_{\gamma(1)}(\cdot,\sigma))\,b_{\gamma(1)}\cdots P_{\gamma(k)}(\varphi_{\gamma(k)}(\cdot,\sigma))\,b_{\gamma(k)}\,P_{\gamma(k+1)}(\varphi_{\gamma(k+1)}(\cdot,\sigma)) \in B(H_j;H_{j+1}),
\end{align*}
for all $\sigma \in \Sigma$.
Then $I^{\boldsymbol{P}}(\Sigma,\rho,\varphi_1,\ldots,\varphi_{k+1})[b] = \int_{\Sigma} F\,d\rho$ and
\begin{align*}
    \Tr\Bigg(\Bigg(\int_{\Sigma} F\,d\rho\Bigg)b_{k+1}\Bigg) & = \int_{\Sigma} \Tr(F(\sigma)\,b_{k+1})\,\rho(d\sigma) \numberthis\label{eq.uwint} \\
    & = \int_{\Sigma} \Tr(F_j(\sigma)\,b_j)\,\rho(d\sigma) \numberthis\label{eq.symmobs} \\
    & = \Tr\Bigg(\Bigg(\int_{\Sigma} F_j\,d\rho\Bigg)b_j\Bigg) \numberthis\label{eq.S1wint} \\
    & = \Tr\big(I^{\boldsymbol{P}_{\gamma}}(\Sigma,\rho,\varphi_{\gamma(1)},\ldots,\varphi_{\gamma(k+1)})[b_{\gamma}]\,b_{\gamma(k+1)}\big),
\end{align*}
where \eqref{eq.uwint} holds because $c \mapsto \Tr(cb_{k+1})$ is $\sigma$-weakly continuous, \eqref{eq.symmobs} holds by Theorem \ref{thm.Schatten}\ref{item.Trflip}, and \eqref{eq.S1wint} holds by Theorem \ref{thm.S1integrability} (plus the fact that the map $\cS_1 \ni c \mapsto \Tr(cb_j) \in \C$ is bounded linear).
This completes the proof for the cyclic permutation $\gamma$.
Since $j \in \{1,\ldots,k\}$ was arbitrary, we are done.
\end{proof}

We now reap the benefits of this technical work:
the ultraweak continuity we have just proven will allow us to show that $I^{\boldsymbol{P}}(\Sigma,\rho,\varphi_1,\ldots,\varphi_{k+1})$ as defined in Notation \ref{nota.MOIdepondecomp} does not depend on the chosen $L^{\infty}_P$-IPD of $\varphi$ and is therefore a reasonable definition of \eqref{eq.formalMOI}.

\begin{thm}[Well-definition of MOI]\label{thm.MOIwelldef}
If $\varphi \in L^{\infty}(\Om_1,P_1) \iotimes \cdots \iotimes L^{\infty}(\Om_{k+1},P_{k+1})$, then
\[
I^{\boldsymbol{P}}(\Sigma,\rho,\varphi_1,\ldots,\varphi_{k+1}) = I^{\boldsymbol{P}}(\tilde{\Sigma},\tilde{\rho},\tilde{\varphi}_1,\ldots,\tilde{\varphi}_{k+1})
\]
whenever $(\Sigma,\rho,\varphi_1,\ldots,\varphi_{k+1})$ and $(\tilde{\Sigma},\tilde{\rho},\tilde{\varphi}_1,\ldots,\tilde{\varphi}_{k+1})$ are $L_P^{\infty}$-IPDs of $\varphi$.
\end{thm}
\begin{proof}
By the argumentwise ultraweak continuity of $I^{\boldsymbol{P}}(\Sigma,\rho,\varphi_1,\ldots,\varphi_{k+1})$ and $I^{\boldsymbol{P}}(\tilde{\Sigma},\tilde{\rho},\tilde{\varphi}_1,\ldots,\tilde{\varphi}_{k+1})$ (i.e., Corollary \ref{cor.uwcont}) and the ultraweak density of finite rank operators, it suffices to prove
\[
I^{\boldsymbol{P}}(\Sigma,\rho,\varphi_1,\ldots,\varphi_{k+1})[b] = I^{\boldsymbol{P}}(\tilde{\Sigma},\tilde{\rho},\tilde{\varphi}_1,\ldots,\tilde{\varphi}_{k+1})[b] \numberthis\label{eq.welldef}
\]
whenever $b = (b_1,\ldots,b_k) \in B(H_2;H_1) \times \cdots \times B(H_{k+1};H_k)$ is such that $b_j$ has finite rank for all $j \in \{1,\ldots,k\}$.
By $k$-linearity, it therefore also suffices to prove \eqref{eq.welldef} when $b_j$ has rank at most one for all $j \in \{1,\ldots,k\}$.

To this end, write $m \coloneqq k+1$ and, for $1 \leq j \leq m-1$, let $b_j = \la \cdot , h_j \ra_{H_{j+1}} k_j$, where $k_j \in H_j$ and $h_j \in H_{j+1}$.
Then $P_j(\varphi_j(\cdot,\sigma)) \, b_j = \la \cdot, h_j \ra_{H_{j+1}} P_j(\varphi_j(\cdot,\sigma))\, k_j$, for all $\sigma \in \Sigma$ and $j \in \{1,\ldots,m-1\}$.
If also $k_m \in H_m$, then
\begin{align*}
    F(\sigma)k_m & \coloneqq P_1(\varphi_1(\cdot,\sigma)) \, b_1 \cdots P_{m-1}(\varphi_{m-1}(\cdot,\sigma)) \, b_{m-1} \, P_m(\varphi_m(\cdot,\sigma))k_m \\
    & = \prod_{j=2}^m \big\la P_j(\varphi_j(\cdot,\sigma)) k_j,h_{j-1} \big\ra_{H_j} P_1(\varphi_1(\cdot,\sigma)) k_1 \\
    & = \Bigg(\prod_{j=2}^m \int_{\Om_j}\varphi_j(\cdot,\sigma) \, d\big(P_j\big)_{k_j,h_{j-1}}\Bigg) P_1(\varphi_1(\cdot,\sigma)) k_1. \numberthis\label{eq.firstexp}
\end{align*}
Next, if $h_0 \in H_1$ and
\begin{align*}
    \nu & \coloneqq (P_1 \otimes \cdots \otimes P_m)_{k_1 \otimes \cdots \otimes k_m, h_0 \otimes \cdots \otimes h_{m-1}} \in M(\Om,\sF) \\
    & = (P_1)_{k_1,h_0} \otimes (P_2)_{k_2,h_1} \otimes \cdots \otimes (P_m)_{k_m,h_{m-1}},\displaybreak
\end{align*}
then
\[
\int_{\Om} \prod_{j=1}^m|\varphi_j(\om_j,\sigma)| \, |\nu|(d\boldsymbol{\om}) \leq \prod_{j=1}^m\|\varphi_j(\cdot,\sigma)\|_{L^{\infty}(P_j)}\|k_j\|_{H_j} \|h_{j-1}\|_{H_j} < \infty, \numberthis\label{eq.bound1}
\]
for all $\sigma \in \Sigma$.
Therefore, by \eqref{eq.firstexp} and Fubini's theorem,
\begin{align*}
    \la F(\sigma)k_m, h_0 \ra_{H_1} & = \Bigg(\prod_{j=2}^m \int_{\Om_j}\varphi_j(\cdot,\sigma) \, d\big(P_j\big)_{k_j,h_{j-1}}\Bigg) \big\la P_1(\varphi_1(\cdot,\sigma)) k_1, h_0 \big\ra_{H_1} \\
    & = \prod_{j=1}^m \int_{\Om_j}\varphi_j(\cdot,\sigma) \, d\big(P_j\big)_{k_j,h_{j-1}} = \int_{\Om} \varphi_1(\om_1,\sigma)\cdots \varphi_m(\om_m,\sigma) \, \nu(d\boldsymbol{\om}), \numberthis\label{eq.secondexp}
\end{align*}
for all $\sigma \in \Sigma$.
Now, by \eqref{eq.bound1},
\[
\int_{\Sigma}\int_{\Om} |\varphi_1(\om_1,\sigma)|\cdots |\varphi_m(\om_m,\sigma)| \, |\nu|(d\boldsymbol{\om}) \, \rho(d\sigma) \leq \Bigg(\prod_{j=1}^m\|k_j\|_{H_j} \|h_{j-1}\|_{H_j}\Bigg) \underline{\int_{\Sigma}}\prod_{j=1}^m\|\varphi_j(\cdot,\sigma)\|_{L^{\infty}(P_j)} \, \rho(d\sigma) < \infty.
\]
Thus, by definition of pointwise Pettis integrals, \eqref{eq.secondexp}, and Fubini's theorem,
\begin{align*}
    \big\la I^{\boldsymbol{P}}(\Sigma,\rho,\varphi_1,\ldots,\varphi_m)[b]k_m,h_0 \big\ra_{H_1} & = \int_{\Sigma}\int_{\Om} \varphi_1(\om_1,\sigma)\cdots \varphi_m(\om_m,\sigma) \, \nu(d\boldsymbol{\om}) \, \rho(d\sigma) \\
    & = \int_{\Om}\int_{\Sigma} \varphi_1(\om_1,\sigma)\cdots \varphi_m(\om_m,\sigma) \, \rho(d\sigma)\, \nu(d\boldsymbol{\om}) \numberthis\label{eq.thirdexp} 
\end{align*}
Finally, note that $\nu \ll P = P_1 \otimes \cdots \otimes P_m$ in the sense that
\[
\{G \in \sF : P(G) = 0\} \subseteq \{G \in \sF : \nu(\tilde{G}) = 0 \text{ when }\sF \ni \tilde{G} \subseteq G\}.
\]
Therefore, the definition of $L^{\infty}_P$-IPD implies
\[
\varphi(\boldsymbol{\om}) = \int_{\Sigma} \varphi_1(\om_1,\sigma)\cdots \varphi_m(\om_m,\sigma) \, \rho(d\sigma)
\]
for $\nu$-almost every (i.e., $|\nu|$-almost every) $\boldsymbol{\om} \in \Om$.
Thus \eqref{eq.thirdexp} becomes
\[
\big\la I^{\boldsymbol{P}}(\Sigma,\rho,\varphi_1,\ldots,\varphi_{k+1})[b]k_m,h_0 \big\ra_{H_1} = \int_{\Om} \varphi \, d\nu,
\]
and $\int_{\Om} \varphi \, d\nu = \la P(\varphi)(k_1 \otimes \cdots \otimes k_m), h_0 \otimes \cdots \otimes h_{m-1} \ra_{H_1 \hotimes \cdots \hotimes H_m}$.
Since the right hand side is independent of the chosen $L^{\infty}_P$-IPD, we are done.
\end{proof}

We are finally allowed to make the following long awaited definition.

\begin{defi}[MOI, take II]\label{def.MOI}
If $\varphi \in L^{\infty}(\Om_1,P_1) \iotimes \cdots \iotimes L^{\infty}(\Om_{k+1},P_{k+1})$, then we define
\begin{align*}
    \big(I^{\boldsymbol{P}}\varphi\big)[b] & = \int_{\Om_{k+1}}\cdots\int_{\Om_1} \varphi(\om_1,\ldots,\om_{k+1}) \, P_1(d\om_1) \, b_1 \cdots P_k(d\om_k) \, b_k \,P_{k+1}(d\om_{k+1}) \\
    & \coloneqq \int_{\Sigma}P_1(\varphi_1(\cdot,\sigma))\, b_1 \cdots P_k(\varphi_k(\cdot,\sigma))\, b_k \,P_{k+1}(\varphi_{k+1}(\cdot,\sigma))\, \rho(d\sigma),
\end{align*}
for all $b = (b_1,\ldots,b_k) \in B(H_2;H_1) \times \cdots \times B(H_{k+1};H_k)$ and any $L^{\infty}_P$-integral projective decomposition $(\Sigma,\rho,\varphi_1,\ldots,\varphi_{k+1})$ of $\varphi$.
We call $I^{\boldsymbol{P}}\varphi$ the \textbf{multiple operator integral} (MOI) of $\varphi$ with respect to $\boldsymbol{P} = (P_1,\ldots,P_{k+1})$.
We also write
\[
(P_1 \otimes \cdots \otimes P_{k+1})(\varphi)\sh[b_1,\ldots,b_k] = P(\varphi) \sh b \coloneqq \big(I^{\boldsymbol{P}}\varphi\big)[b].
\]
\end{defi}
\begin{rem}[$\#$ Operation]\label{rem.hash}
The $\sh$ in the definition above \textit{formally} stands for the algebraic action
\[
\# \colon B(H_1) \otimes \cdots \otimes B(H_{k+1}) \to \Hom(B(H_2;H_1) \otimes \cdots \otimes B(H_{k+1};H_k); B(H_{k+1};H_1))
\]
determined (linearly) by
\[
(a_1 \otimes \cdots \otimes a_{k+1})\sh [b_1 \otimes \cdots \otimes b_k] \coloneqq a_1b_1\cdots a_1b_ka_{k+1}\pagebreak
\]
for $a_1 \in B(H_1),\ldots,a_{k+1} \in B(H_{k+1})$ and $b_1 \in B(H_2;H_1),\ldots,b_k \in B(H_{k+1};H_k)$.
Now, the von Neumann algebra tensor product $B(H_1) \wotimes \cdots \wotimes B(H_{k+1})$ is naturally isomorphic to $B(H) = B(H_1 \hotimes \cdots \hotimes H_{k+1})$.
Therefore, morally speaking, ``the multiple operator integral $(I^{P_1,\ldots,P_{k+1}}\varphi)[b_1,\ldots,b_k]$ is
\[
P(\varphi) = \int_{\Om_1 \times \cdots \times \Om_{k+1}} \varphi \, d(P_1 \otimes \cdots \otimes P_{k+1}) \in B(H) = B(H_1) \wotimes \cdots \wotimes B(H_{k+1})
\]
acting on $b_1 \otimes \cdots \otimes b_k$ via $\#$," even though this does not actually make sense (i.e., $\#$ may not extend to the von Neumann algebra tensor product).
We continue this discussion in Remark \ref{rem.morehash}.
\end{rem}

We end this section by restricting the MOI we just defined to a von Neumann algebra.
Notice first that Theorem \ref{thm.MOIwelldef} and Corollary \ref{cor.MOIintgood} give
\[
\big\|\big(I^{\boldsymbol{P}}\varphi\big)[b]\big\|_{B(H_{k+1};H_1)} \leq \|\varphi\|_{L^{\infty}(P_1) \iotimes \cdots \iotimes L^{\infty}(P_{k+1})}\prod_{j=1}^k\|b_j\|_{B(H_{j+1};H_j)}, \numberthis\label{eq.opnormestim}
\]
for all $b = (b_1,\ldots,b_k) \in B(H_2;H_1) \times \cdots \times B(H_{k+1};H_k)$.

\begin{lem}\label{lem.PintinM}
Suppose $(\Xi,\sG,K,Q)$ is a projection valued measure space and $\cM \subseteq B(K)$ is a von Neumann algebra.
If $Q(G) \in \cM$, for all $G \in \sG$, then $Q(f) \in \cM$, for all $f \in L^{\infty}(\Xi,Q)$.
\end{lem}
\begin{proof}
Of course, it suffices to prove $Q(f) \in \cM$, for all $f \in \ell^{\infty}(\Xi,\sG)$.
To this end, let $\mathbb{H} \coloneqq \{f \in \ell^{\infty}(\Xi,\sG) : Q(f) \in \cM\}$.
Since $\cM$ is a linear subspace closed under taking adjoints, $\mathbb{H}$ is a vector space that is closed under complex conjugation.
Since $\cM$ is closed in the SOT, Lemma \ref{lem.Pintcont} implies that $\mathbb{H}$ is closed under sequential bounded convergence.
By assumption $\mathbb{M} \coloneqq \{1_G : G \in \sG\} \subseteq \mathbb{H}$.
Since $\mathbb{M}$ is closed under multiplication and complex conjugation, the multiplicative system theorem gives $\ell^{\infty}(\Xi,\sG) = \ell^{\infty}(\Xi,\sigma(\mathbb{M})) \subseteq \mathbb{H}$, as desired.
\end{proof}

\begin{thm}[MOIs in $\cM$]\label{thm.MOIsinM}
Suppose $H_1 = \cdots = H_{k+1} = K$, $\cM \subseteq B(K)$ is a von Neumann algebra, and $P_j$ takes values in $\cM$ for all $j \in \{1,\ldots,k+1\}$.
If $\varphi \in L^{\infty}(\Om_1,P_1) \iotimes \cdots \iotimes L^{\infty}(\Om_{k+1},P_{k+1})$ and $(\Sigma,\rho,\varphi_1,\ldots,\varphi_{k+1})$ is a $L_P^{\infty}$-IPD of $\varphi$ (where $P \coloneqq P_1 \otimes \cdots \otimes P_{k+1}$), then
\[
\big(I^{\boldsymbol{P}}\varphi\big)[b] = \int_{\Sigma} P_1(\varphi_1(\cdot,\sigma))\,b_1\cdots P_k(\varphi_k(\cdot,\sigma))\,b_k \, P_{k+1}(\varphi_{k+1}(\cdot,\sigma)) \, \rho(d\sigma) \numberthis\label{eq.wstarintMOI}
\]
is a weak$^*$ integral in $\cM$ whenever $b = (b_1,\ldots,b_k) \in \cM^k$.
Moreover, $I^{\boldsymbol{P}}\varphi \colon B(K)^k \to B(K)$ restricts to a bounded $k$-linear map $\cM^k \to \cM$ with operator norm at most $\|\varphi\|_{L^{\infty}(P_1) \iotimes \cdots \iotimes L^{\infty}(P_{k+1})}$ that is $\sigma$-weakly continuous in each argument.
Finally, $I^{\boldsymbol{P}}\varphi$ is independent of the representation of $\cM$ in the sense that if $\cN$ is another von Neumann algebra, and $\pi \colon \cM \to \cN$ is an algebraic $\ast$-isomorphism, then
\[
\pi\big(\big(I^{P_1,\ldots,P_{k+1}}\varphi\big)[b_1,\ldots,b_k]\big) = \big(I^{\pi \circ P_1,\ldots,\pi \circ P_{k+1}}\varphi\big)[\pi(b_1),\ldots,\pi(b_k)],
\]
for all $b_1,\ldots,b_k \in \cM$.
\end{thm}
\begin{proof}
By Lemma \ref{lem.PintinM}, the definition of $I^{\boldsymbol{P}}\varphi$, and Corollary \ref{cor.wstarint}, the expression \eqref{eq.wstarintMOI} defining $\big(I^{\boldsymbol{P}}\varphi\big)[b]$ is a weak$^*$ integral in $\cM$ whenever $b \in \cM^k$.
We know from \eqref{eq.opnormestim} that the restriction $I^{\boldsymbol{P}}\varphi \colon \cM^k \to \cM$ has the claimed operator norm bound.
Corollary \ref{cor.uwcont} implies the restriction $I^{\boldsymbol{P}}\varphi \colon \cM^k \to \cM$ is $\sigma$-weakly continuous in each argument because the $\sigma$-weak operator topology on $\cM$ is the subspace topology induced by the $\sigma$-weak topology on $B(K)$, and the latter is the same as the ultraweak topology.
The final claim follows from Theorem \ref{thm.wintBHK}\ref{item.indepofrep} and the fact that $\pi(P_j(f)) = (\pi \circ P_j)(f)$, for all $f \in L^{\infty}(\Om_j,P_j)$, by another multiplicative system theorem argument, which we leave to the reader.
\end{proof}
\begin{rem}[The general semifinite case]\label{rem.semifinitecase}
Let $(\cM,\tau)$ be a semifinite von Neumann algebra.
The arguments in this section are robust in the sense that they can be used to prove the following generalizations (in the case $H_j=K$ for all $j \in \{1,\ldots,k+1\}$) of Theorems \ref{thm.traceofinteg} and \ref{thm.S1integrability}:
\begin{enumerate}[label=(\roman*),font=\normalfont,leftmargin=2\parindent]
    \item Let $(\Xi,\sG,K,Q)$ be a projection valued measure space such that $Q$ takes values in $\cM$, $(\Sigma,\sH,\rho)$ be a finite measure space, $A \colon \Sigma \to \cM$ be weak$^*$ measurable, and $\varphi \in \ell^{\infty}(\Xi \times \Sigma, \sG \otimes \sH)$.
    If
    \[
    \sup_{\sigma \in \Sigma}\max\big\{\|A(\sigma)\|_{L^1(\tau)},\|A(\sigma)\|_{L^{\infty}(\tau)}\big\} < \infty,\pagebreak
    \]
    then
    \begin{align*}
        & \int_{\Sigma} A(\sigma)\,Q(\varphi(\cdot,\sigma))\,\rho(d\sigma),\int_{\Sigma} Q(\varphi(\cdot,\sigma))\,A(\sigma)\,\rho(d\sigma) \in \mathcal{L}^1(\tau) \; \text{ and} \\
        & \tau\Bigg(\int_{\Sigma} Q(\varphi(\cdot,\sigma))\,A(\sigma)\,\rho(d\sigma)\Bigg) = \tau\Bigg(\int_{\Sigma} A(\sigma)\,Q(\varphi(\cdot,\sigma))\,\rho(d\sigma)\Bigg).
    \end{align*}
    \item Suppose we are in the setup of Theorem \ref{thm.MOIsinM}.
    If $j \in \{1,\ldots,k\}$ and $b_j \in \mathcal{L}^1(\tau)$, then the map
    \[
    \Sigma \ni \sigma \mapsto P_1(\varphi_1(\cdot,\sigma))\,b_1\cdots P_k(\varphi_k(\cdot,\sigma))\,b_k\,P_{k+1}(\varphi_{k+1}(\cdot,\sigma)) \in L^1(\tau)
    \]
    is Gel'fand--Pettis integrable as a map $\Sigma \to (L^1(\tau),\|\cdot\|_{L^1(\tau)})$, and its Gel'fand--Pettis integral is the multiple operator integral \eqref{eq.wstarintMOI}.
\end{enumerate}
To prove these, one uses the same arguments but with Theorem \ref{thm.Lpinteg} in place of Theorem \ref{thm.Spinteg}, a replacement for Lemma \ref{lem.SOTS1} (e.g., Lemma 2.5 in \cite{azamovetal}), and duality/basic properties of $L^1$ instead of $\cS_1$.
Facts such as the two above are of use when proving trace formulas;
please see Section 5.5 of \cite{skripka} for a survey of some existing results on trace formulas.
\end{rem}

\subsection{Algebraic properties and noncommutative \texorpdfstring{$L^p$}{} estimates}\label{sec.genprop}

In this section, we prove linearity and multiplicativity properties of the MOI defined in the previous section.
Then we prove Schatten $p$-norm and --- in the case of a semifinite von Neumann algebra --- noncommutative $L^p$-norm estimates for MOIs.

\begin{prop}[Algebraic properties of MOIs]\label{prop.linandmult}
Suppose $1 \leq m \leq k$.
\begin{enumerate}[label=(\roman*),font=\normalfont,leftmargin=2\parindent]
    \item If $\varphi,\psi \in L^{\infty}(\Om_1,P_1) \iotimes \cdots \iotimes L^{\infty}(\Om_{k+1},P_{k+1})$ and $\alpha \in \C$, then $I^{\boldsymbol{P}}(\varphi+\alpha\,\psi) = I^{\boldsymbol{P}}\varphi+\alpha\,I^{\boldsymbol{P}}\psi$.\label{item.lin}
    \item If \[
    \psi_1 \in L^{\infty}(\Om_1,P_1) \iotimes \cdots \iotimes L^{\infty}(\Om_m,P_m) \; \text{ and } \; \psi_2 \in L^{\infty}(\Om_{m+1},P_{m+1}) \iotimes \cdots \iotimes L^{\infty}(\Om_{k+1},P_{k+1}),
    \]
    then (the $P$-almost everywhere equivalence class of) the function
    \[
    (\psi_1 \otimes \psi_2)(\om_1,\ldots,\om_{k+1}) \coloneqq \psi_1(\om_1,\ldots,\om_m)\,\psi_2(\om_{m+1},\ldots,\om_{k+1})
    \]
    belongs to $L^{\infty}(\Om_1,P_1) \iotimes \cdots \iotimes L^{\infty}(\Om_{k+1},P_{k+1})$ and
    \[
    \big(I^{\boldsymbol{P}}(\psi_1 \otimes \psi_2)\big)[b] = \big(I^{P_1,\ldots,P_m}\psi_1\big)[b_1,\ldots,b_{m-1}]\,b_m\,\big(I^{P_{m+1},\ldots,P_{k+1}}\psi_2\big)[b_{m+1},\ldots,b_k],
    \]
    for all $b = (b_1,\ldots,b_k) \in B(H_2;H_1) \times \cdots \times B(H_{k+1};H_k)$.\label{item.mult1}
    \item If $\psi \in L^{\infty}(\Om_m,P_m) \iotimes L^{\infty}(\Om_{m+1},P_{m+1})$,
    \[
    \tilde{\psi}(\boldsymbol{\om}) \coloneqq \psi(\om_m,\om_{m+1}) \; \text{ for } \; \boldsymbol{\om} = (\om_1,\ldots,\om_{k+1}) \in \Om,
    \]
    and $\varphi \in L^{\infty}(\Om_1,P_1) \iotimes \cdots \iotimes L^{\infty}(\Om_{k+1},P_{k+1})$, then
    \[
    \big(I^{\boldsymbol{P}}(\varphi\tilde{\psi})\big)[b] = \big(I^{\boldsymbol{P}}\varphi\big)\big[b_1,\ldots,b_{m-1},\big(I^{P_m,P_{m+1}}\psi\big)[b_m],b_{m+1},\ldots,b_k\big],
    \]
    for all $b = (b_1,\ldots,b_k) \in B(H_2;H_1) \times \cdots \times B(H_{k+1};H_k)$.\label{item.mult2}
\end{enumerate}
\end{prop}
\begin{proof}
We take each item in turn.

\ref{item.lin} Let $\varphi, \psi \in L^{\infty}(\Om_1,P_1) \iotimes \cdots \iotimes L^{\infty}(\Om_{k+1},P_{k+1})$ and $\alpha \in \C$.
It is easy to see that $I^{\boldsymbol{P}}(\alpha\varphi) = \alpha \, I^{\boldsymbol{P}}\varphi$.
To prove additivity, let $(\Sigma_1,\rho_1,\varphi_1,\ldots,\varphi_{k+1})$ and $(\Sigma_2,\rho_2,\psi_1,\ldots,\psi_{k+1})$ be $L_P^{\infty}$-IPDs of $\varphi$ and $\psi$, respectively.
Then we take $(\Sigma,\sH,\rho)$ to be the disjoint union of the measure spaces $(\Sigma_1,\sH_1,\rho_1)$ and $(\Sigma_2,\sH_2,\rho_2)$ and, for $j \in \{1,\ldots,k+1\}$, define $\chi_j \colon \Om_j \times \Sigma \to \C$ by
\[
\chi_j(\om_j,\sigma) \coloneqq \begin{cases}
\varphi_j(\om_j,\sigma) & \text{if } \sigma \in \Sigma_1 \subseteq \Sigma,\\
\psi_j(\om_j,\sigma) & \text{if } \sigma \in \Sigma_2 \subseteq \Sigma,
\end{cases}\pagebreak
\]
for all $(\om_j,\sigma) \in \Om_j \times \Sigma$.
Then, as is argued in the proof of Proposition \ref{prop.IPTPspecial}, $(\Sigma,\rho,\chi_1,\ldots,\chi_{k+1})$ is a $L^{\infty}_P$-IPD of $\varphi+\psi$.
Thus, by definition of the disjoint union measure space and pointwise Pettis integrals,
\[
I^{\boldsymbol{P}}(\varphi+\psi) = I^{\boldsymbol{P}}(\Sigma,\rho,\chi_1,\ldots,\chi_{k+1}) = I^{\boldsymbol{P}}(\Sigma_1,\rho_1,\varphi_1,\ldots,\varphi_{k+1}) + I^{\boldsymbol{P}}(\Sigma_2,\rho_2,\psi_1,\ldots,\psi_{k+1}) = I^{\boldsymbol{P}}\varphi + I^{\boldsymbol{P}}\psi.
\]
Thus $\varphi \mapsto I^{\boldsymbol{P}}\varphi$ is linear.

\ref{item.mult1} Let $\psi_1 \in L^{\infty}(\Om_1,P_1) \iotimes \cdots \iotimes L^{\infty}(\Om_m,P_m)$ and $\psi_2 \in L^{\infty}(\Om_{m+1},P_{m+1}) \iotimes \cdots \iotimes L^{\infty}(\Om_{k+1},P_{k+1})$.
If $(\Sigma_1,\rho_1,\varphi_1,\ldots,\varphi_m)$ and $(\Sigma_2,\rho_2,\varphi_{m+1},\ldots,\varphi_{k+1})$ are, respectively, $L^{\infty}_{P_1 \otimes \cdots \otimes P_m}$- and $L^{\infty}_{P_{m+1}\otimes \cdots \otimes P_{k+1}}$-IPDs of $\psi_1$ and $\psi_2$, then
\[
(\Sigma_1,\rho_1,\varphi_1,\ldots,\varphi_m,\underbrace{1,\cdots, 1}_{k+1-m}) \; \text{ and } \; (\Sigma_2,\rho_2,\underbrace{1,\ldots,1}_{m},\varphi_{m+1},\ldots,\varphi_{k+1})
\]
are, respectively, $L^{\infty}_P$-IPDs of $\psi_1 \otimes 1$ and $1 \otimes \psi_2$.
But then
\[
\psi_1 \otimes \psi_2 = (\psi_1 \otimes 1)(1 \otimes \psi_2) \in L^{\infty}(\Om_1,P_1) \iotimes \cdots \iotimes L^{\infty}(\Om_{k+1},P_{k+1})
\]
because $L^{\infty}(\Om_1,P_1) \iotimes \cdots \iotimes L^{\infty}(\Om_{k+1},P_{k+1})$ is an algebra.
And actually, by the argument from the proof of Proposition \ref{prop.IPTPspecial}, if we redefine
\begin{align*}
    (\Sigma,\sH,\rho) & \coloneqq (\Sigma_1 \times \Sigma_2,\sH_1 \otimes \sH_2,\rho_1 \otimes \rho_2) \; \text{ and} \\
    \chi_j(\om_j,\sigma) & \coloneqq \begin{cases}
    \varphi_j(\om_j, \sigma_1) & \text{if } 1 \leq j \leq m, \\
    \varphi_j(\om_j, \sigma_2) & \text{if } m+1 \leq j \leq k+1,
    \end{cases}
\end{align*}
for $\om_j \in \Om_j$ and $\sigma = (\sigma_1,\sigma_2) \in \Sigma_1 \times \Sigma_2 = \Sigma$, then $(\Sigma,\rho,\chi_1,\ldots,\chi_{k+1})$ is a $L^{\infty}_P$-IPD of $\psi_1 \otimes \psi_2$.
This observation implies the result.
Indeed, let $h_1 \in H_{k+1}$, $h_2 \in H_1$,
\begin{align*}
    h_3 & \coloneqq (I^{P_{m+1},\ldots,P_{k+1}}\psi_2)[b_{m+1},\ldots,b_k]h_1, \; \text{ and} \\
    T & \coloneqq \big(I^{P_1,\ldots,P_m}\psi_1\big)[b_1,\ldots,b_{m-1}]\,b_m\,\big(I^{P_{m+1},\ldots,P_{k+1}}\psi_2\big)[b_{m+1},\ldots,b_k].
\end{align*}
Then
\begin{align*}
    \la Th_1,h_2 \ra_{H_1} & = \big\la \big(I^{P_1,\ldots,P_m}\psi_1\big)[b_1,\ldots,b_{m-1}]b_mh_3,h_2 \big\ra_{H_1} \\
    & = \int_{\Sigma_1}\Bigg\la\Bigg( \prod_{j=1}^m P_j(\varphi_j(\cdot,\sigma_1)) \, b_j\Bigg)h_3,h_2 \Bigg\ra_{H_1} \, \rho_1(d\sigma_1) \\
    & = \int_{\Sigma_1}\int_{\Sigma_2}\Bigg\la \Bigg( \prod_{j_1=1}^m P_{j_1}(\varphi_{j_1}(\cdot,\sigma_1)) \, b_{j_1}\Bigg)\Bigg( \prod_{j_2=m+1}^k P_{j_2}(\varphi_{j_2}(\cdot,\sigma_2)) \, b_{j_2}\Bigg) \\
    & \hspace{50mm} \times P_{k+1}(\varphi_{k+1}(\cdot,\sigma_2))h_1,h_2\Bigg\ra_{H_1}\rho_2(d\sigma_2) \,\rho_1(d\sigma_1)  \\
    & = \int_{\Sigma} \big\la P_1(\chi_1(\cdot,\sigma)) \,b_1 \cdots P_k(\chi_k(\cdot,\sigma))\,b_k\,P_{k+1}(\chi_{k+1}(\cdot,\sigma))h_1,h_2 \big\ra_{H_1} \, \rho(d\sigma) \\
    & = \big\la \big(I^{\boldsymbol{P}}(\psi_1 \otimes \psi_2)\big)[b_1,\ldots,b_k]h_1,h_2 \big\ra_{H_1}
\end{align*}
by definition and Fubini's theorem.
This completes the proof of the first multiplicativity claim.

\ref{item.mult2} Let $(\Sigma_1,\rho_1,\varphi_1,\ldots,\varphi_{k+1})$ be a $L_P^{\infty}$-IPD of $\varphi$ and $(\Sigma_2,\rho_2,\psi_m,\varphi_{m+1})$ be a $L_{P_m \otimes P_{m+1}}^{\infty}$-IPD of $\psi$.
Then
\[
(\Sigma_2,\rho_2,\underbrace{1,\ldots,1}_{m-1},\psi_m,\psi_{m+1},\underbrace{1,\cdots, 1}_{k-m})
\]
is a $L_P^{\infty}$-IPD of $\tilde{\psi}$ (as defined in the statement of this item).
Once again, by the argument from the proof of Proposition \ref{prop.IPTPspecial}, if we redefine
\begin{align*}
    (\Sigma,\sH,\rho) & \coloneqq (\Sigma_1 \times \Sigma_2,\sH_1 \otimes \sH_2, \rho_1 \otimes \rho_2) \; \text{ and} \\
    \chi_j(\om_j,\sigma) & \coloneqq \begin{cases}
    \varphi_j(\om_j, \sigma_1) & \text{if } 1 \leq j \leq m-1, \\
    \varphi_j(\om_j, \sigma_1)\,\psi_j(\om_j,\sigma_2) & \text{if } m \leq j \leq m+1, \\
    \varphi_j(\om_j, \sigma_1) & \text{if } m+2 \leq j \leq k+1,
    \end{cases}
\end{align*}
for $\om_j \in \Om_j$ and $\sigma = (\sigma_1,\sigma_2) \in \Sigma_1 \times \Sigma_2 = \Sigma$, then $(\Sigma,\rho,\chi_1,\ldots,\chi_{k+1})$ is a $L_P^{\infty}$-IPD of $\varphi \tilde{\psi}$.
Now, if $h_1 \in H_{k+1}$, $h_2 \in H_1$, $b_m^{\psi} \coloneqq \big(I^{P_m,P_{m+1}}\psi\big)[b_m]$, and
\[
T \coloneqq \big(I^{P_1,\ldots,P_{k+1}}\varphi\big)\big[b_1,\ldots,b_{m-1},b_m^{\psi},b_{m+1},\ldots,b_k\big],
\]
then
\begin{align*}
    \la Th_1,h_2 \ra_{H_1} & = \big\la \big(I^{P_1,\ldots,P_{k+1}}\varphi\big)\big[b_1,\ldots,b_{m-1},b_m^{\psi},b_{m+1},\ldots,b_k\big]h_1,h_2 \big\ra_{H_1} \\
    & = \int_{\Sigma_1}\Bigg\la\Bigg( \prod_{j_1=1}^{m-1}P_{j_1}(\varphi_{j_1}(\cdot,\sigma_1)) \, b_{j_1}\Bigg) P_m(\varphi_m(\cdot,\sigma_1))\,b_m^{\psi} \\
    & \hspace{15mm} \times \Bigg(\prod_{j_2=m+1}^k P_{j_2}(\varphi_{j_2}(\cdot,\sigma_1)) \, b_{j_2}\Bigg) P_{k+1}(\varphi_{k+1}(\cdot,\sigma_1))h_1,h_2 \Bigg\ra_{H_1} \,\rho_1(d\sigma_1) \\
    & = \int_{\Sigma_1}\int_{\Sigma_2}\Bigg\la\Bigg(\prod_{j_1=1}^{m-1} P_{j_1}(\varphi_{j_1}(\cdot,\sigma_1))\, b_{j_1}\Bigg)P_m(\varphi_m(\cdot,\sigma_1)) \\
    & \hspace{15mm}  \times P_m(\psi_m(\cdot,\sigma_2))\,b_m\,P_{m+1}(\psi_{m+1}(\cdot,\sigma_2))\, P_{m+1}(\varphi_{m+1}(\cdot,\sigma_1)) \,b_{m+1} \\
    & \hspace{15mm}  \times\Bigg(\prod_{j_2=m+2}^kP_{j_2}(\varphi_{j_2}(\cdot,\sigma_1))\, b_{j_2}\Bigg) P_{k+1}(\varphi_{k+1}(\cdot,\sigma_1))h_1,h_2 \Bigg\ra_{H_1} \rho_2(d\sigma_2)\,\rho_1(d\sigma_1) \\
    & = \int_{\Sigma} \big\la P_1(\chi_1(\cdot,\sigma)) \,b_1 \cdots P_k(\chi_k(\cdot,\sigma))\,b_k\,P_{k+1}(\chi_{k+1}(\cdot,\sigma))h_1,h_2 \big\ra_{H_1} \,\rho(d\sigma) \\
    & = \big\la \big(I^{\boldsymbol{P}}(\psi_1 \otimes \psi_2)\big)[b_1,\ldots,b_k]h_1,h_2 \big\ra_{H_1}
\end{align*}
by definition, multiplicativity of integration with respect to a projection valued measure, and Fubini's theorem.
This completes the proof of the second multiplicativity claim.
\end{proof}

\begin{prop}[Schatten estimates on MOIs]\label{prop.MOISchattenestim}
If $\varphi \,\in\, L^{\infty}(\Om_1,\,P_1) \,\iotimes\, \cdots \,\iotimes\, L^{\infty}(\Om_{k+1},\,P_{k+1})$ and $p,p_1,\ldots,p_k \in [1,\infty]$ are such that $\frac{1}{p} = \frac{1}{p_1}+\cdots+\frac{1}{p_k}$, then
\[
\big\|(I^{\boldsymbol{P}}\varphi)[b]\big\|_{\cS_p} \leq \|\varphi\|_{L^{\infty}(P_1) \iotimes \cdots \iotimes L^{\infty}(P_{k+1})}\|b_1\|_{\cS_{p_1}}\cdots\|b_k\|_{\cS_{p_k}},
\]
for all $b = (b_1,\ldots,b_k) \in B(H_2;H_1) \times \cdots \times B(H_{k+1};H_k)$.
(As usual, $0 \cdot \infty \coloneqq 0$.)
\end{prop}
\begin{proof}
Let $(\Sigma,\rho,\varphi_1,\ldots,\varphi_{k+1})$ be a $L^{\infty}_P$-IPD of $\varphi$.
By definition, Theorem \ref{thm.Spinteg}, and H\"{o}lder's inequality,
\begin{align*}
    \big\|(I^{\boldsymbol{P}}\varphi)[b]\big\|_{\cS_p} & = \Bigg\|\int_{\Sigma}P_1(\varphi_1(\cdot,\sigma))\,b_1 \cdots P_k(\varphi_k(\cdot,\sigma)) \, b_k \, P_{k+1}(\varphi_{k+1}(\cdot,\sigma)) \, \rho(d\sigma)\Bigg\|_{\cS_p} \\
    & \leq \underline{\int_{\Sigma}}\big\|P_1(\varphi_1(\cdot,\sigma))\,b_1 \cdots P_k(\varphi_k(\cdot,\sigma)) \, b_k \, P_{k+1}(\varphi_{k+1}(\cdot,\sigma))\big\|_{\cS_p} \, \rho(d\sigma) \\
    & \leq \|b_1\|_{\cS_{p_1}}\cdots\|b_k\|_{\cS_{p_k}}\underline{\int_{\Sigma}}\prod_{j=1}^{k+1}\|\varphi_j(\cdot,\sigma)\|_{L^{\infty}(P_j)}\,\rho(d\sigma).
\end{align*}
Using $\underline{\int_{\Sigma}} \leq \overline{\int_{\Sigma}}$ and then taking the infimum over all $L_P^{\infty}$-IPDs $(\Sigma,\rho,\varphi_1,\ldots,\varphi_{k+1})$ gives the desired result.
\end{proof}

By the same proof, using Theorem \ref{thm.Lpinteg} in place of Theorem \ref{thm.Spinteg} and noncommutative H\"{o}lder's inequality in place of H\"{o}lder's inequality for the Schatten norms, we get the following.

\begin{prop}[Noncommutative $L^p$ estimates on MOIs]\label{prop.MOILpestim}
Suppose that $H_1 = \cdots = H_{k+1} = K$, $(\cM \subseteq B(K),\tau)$ is a semifinite von Neumann algebra, and $P_j$ takes values in $\cM$ for all $j \in \{1,\ldots,k+1\}$.
If $\varphi \in L^{\infty}(\Om_1,P_1) \iotimes \cdots \iotimes L^{\infty}(\Om_{k+1},P_{k+1})$ and $p,p_1,\ldots,p_k \in [1,\infty]$ are such that $\frac{1}{p} = \frac{1}{p_1}+\cdots+\frac{1}{p_k}$, then
\[
\big\|(I^{\boldsymbol{P}}\varphi)[b]\big\|_{L^p(\tau)} \leq \|\varphi\|_{L^{\infty}(P_1) \iotimes \cdots \iotimes L^{\infty}(P_{k+1})}\|b_1\|_{L^{p_1}(\tau)}\cdots\|b_k\|_{L^{p_k}(\tau)},
\]
for all $b = (b_1,\ldots,b_k) \in \cM^k$.
(As usual, $0 \cdot \infty \coloneqq 0$.)
In particular, $I^{\boldsymbol{P}}\varphi$ extends to a bounded $k$-linear map $L^{p_1}(\tau) \times \cdots \times L^{p_k}(\tau) \to L^p(\tau)$ with operator norm at most $\|\varphi\|_{L^{\infty}(P_1) \iotimes \cdots \iotimes L^{\infty}(P_{k+1})}$.
\end{prop}

\subsection{Relation to other definitions}\label{sec.otherdefs}

For completeness, we now review one frequently used alternative definition --- due to Pavlov and Birman--Solomyak --- of \eqref{eq.formalMOI} and prove that it agrees with our definition from the previous section when both definitions apply.
This alternative definition requires the construction of a certain vector measure.
(Please see Section \ref{sec.proofofPavlov} for the terminology in the following statement.)

\begin{thm}[Pavlov, Birman--Solomyak]\label{thm.Pavlov}
If $b=(b_1,\ldots,b_k) \in \cS_2(H_2;H_1) \times \cdots \times \cS_2(H_{k+1};H_k)$, then there exists a unique vector measure $P \sh b \colon \sF \to \cS_2(H_{k+1};H_1)$ with bounded semivariation such that
\[
(P\sh b)(G_1 \times \cdots \times G_{k+1}) = P_1(G_1)\,b_1\cdots P_k(G_k)\,b_k\,P_{k+1}(G_{k+1})
\]
for all $G_1 \in \sF_1,\ldots,G_{k+1} \in \sF_{k+1}$.
The semivariation $\|P \sh b\|_{\operatorname{svar}}$ of $P\sh b$ is at most $\|b_1\|_{\cS_2}\cdots\|b_k\|_{\cS_2}$, and $P \sh b \ll P$ in the sense that $\{G \in \sF : P(G) = 0\} \subseteq \{G \in \sF : (P\sh b)(\tilde{G}) = 0$ whenever $\sF \ni \tilde{G} \subseteq G\}$.
\end{thm}
\begin{rem}\label{rem.morehash}
The notation for the vector measure in Theorem \ref{thm.Pavlov} is not standard.
It is inspired by the $\#$ operation discussed in Remark \ref{rem.hash}.
As the notation suggests, morally speaking, ``$P \sh b$ is the projection valued measure $P = P_1 \otimes \cdots \otimes P_{k+1}$ acting on $b = b_1 \otimes \cdots \otimes b_k$ via $\#$."
Indeed, the condition uniquely characterizing $P \sh b$ can be rewritten genuinely as $(P \sh b)(G) = P(G) \sh b$ for all $G = G_1\times \cdots \times G_{k+1}$ with $G_1 \in \sF_1,\ldots,G_{k+1} \in \sF_{k+1}$ because in this case $P(G) \in B(H_1) \otimes \cdots \otimes B(H_{k+1}) \subseteq B(H)$.
Therefore, morally speaking, integrating a function $\varphi$ with respect to this vector measure may also be viewed as ``$\int_{\Om} \varphi \, dP$ acting on $b$ via $\#$," which also matches the interpretation discussed in Remark \ref{rem.hash}.
\end{rem}

The proof Pavlov originally provided in \cite{pavlov} that this construction is possible has an error.
Birman--Solomyak pointed out and sketched a correction of this error in \cite{birmansolomyakTensProd}.
For the benefit of the reader, we provide a complete proof, including the necessary background in vector measure theory in Section \ref{sec.proofofPavlov}.
In any case, following Pavlov, Theorem \ref{thm.Pavlov} allows us to define \eqref{eq.formalMOI} as
\[
\int_{\Om} \varphi \, d(P \sh b) \in \cS_2(H_{k+1};H_1)
\]
for \textit{all} $\varphi \in L^{\infty}(\Om,P)$, where $\Om = \Om_1 \times \cdots \times \Om_{k+1}$, but only $b \in \cS_2(H_2;H_1) \times \cdots \times \cS_2(H_{k+1};H_k)$.
(Please see pages 5--6 in \cite{diesteluhl} for the definition of the integral above.)
In this case,
\[
\Bigg\|\int_{\Om}\varphi \, d(P\sh b)\Bigg\|_{\cS_2} \leq \|\varphi\|_{L^{\infty}(P\sh b)}\|P\sh b\|_{\mathrm{svar}} \leq \|\varphi\|_{L^{\infty}(P)}\|b_1\|_{\cS_2}\cdots\|b_k\|_{\cS_2}. \numberthis\label{eq.Pavlovbound}
\]
We now show this definition agrees with the one we developed in Section \ref{sec.welldef} when they both apply.

\begin{thm}[Agreement with Pavlov MOI]\label{thm.pavlovMOIagree}
If $\varphi \in L^{\infty}(\Om_1,P_1) \iotimes \cdots \iotimes L^{\infty}(\Om_{k+1},P_{k+1})$, then
\[
\big(I^{\boldsymbol{P}}\varphi\big)[b] = \int_{\Om}\varphi \, d(P \sh b),
\]
for all $b \in \cS_2(H_2;H_1) \times \cdots \times \cS_2(H_{k+1};H_k)$.
\end{thm}
\begin{proof}
First, note that
\[
\cS_2(H_2;H_1) \times \cdots \times \cS_2(H_{k+1};H_k) \ni b \mapsto \int_{\Om} \varphi \, d(P\sh b) \in \cS_2(H_{k+1};H_1)
\]
is a bounded $k$-linear map with operator norm at most $\|\varphi\|_{L^{\infty}(P)}$ by \eqref{eq.Pavlovbound} and the $k$-linearity of the condition uniquely characterizing $P \sh b$.
By Proposition \ref{prop.MOISchattenestim},
\[
\cS_2(H_2;H_1) \times \cdots \times \cS_2(H_{k+1};H_k) \ni b \mapsto \big(I^{\boldsymbol{P}}\varphi\big)[b] \in \cS_2(H_{k+1};H_1)
\]
is a bounded $k$-linear map with operator norm at most $\|\varphi\|_{L^{\infty}(P_1) \iotimes \cdots \iotimes L^{\infty}(P_{k+1})}$.
Since finite rank operators are dense in $\cS_2$, it therefore suffices to prove
\[
\big(I^{\boldsymbol{P}}\varphi\big)[b] = \int_{\Om}\varphi \, d(P \sh b),
\]
for all $b = (b_1,\ldots,b_k) \in \cS_2(H_2;H_1) \times \cdots \times \cS_2(H_{k+1};H_k)$ such that $b_1,\ldots,b_k$ all have rank at most one.
\pagebreak

Now, recall $\cS_2(H_1;H_{k+1}) \cong \cS_2(H_{k+1};H_1)^*$ via the map $B \mapsto (A \mapsto \Tr(AB))$.
Therefore, $\int_{\Om}\varphi \, d(P\sh b)$ is determined by the requirement that
\[
\Tr\Bigg(\int_{\Om}\varphi \, d(P\sh b)\,b_{k+1}\Bigg) = \int_{\Om} \varphi(\boldsymbol{\om}) \, \Tr((P \sh b)(d\boldsymbol{\om})\,b_{k+1}), \numberthis\label{eq.Pavlovdet}
\]
for all $b_{k+1} \in \cS_2(H_1;H_{k+1})$.
Once again, since finite rank operators are dense in $\cS_2$ and the above equation is bounded linear in $b_{k+1}$, $\int_{\Om} \varphi \, d(P \sh b)$ is determined by \eqref{eq.Pavlovdet} for all $b_{k+1} \colon H_1 \to H_{k+1}$ with rank at most one.
It therefore suffices to prove that
\[
\Tr\big(\big(I^{\boldsymbol{P}}\varphi\big)[b]\,b_{k+1}\big) = \int_{\Om} \varphi(\boldsymbol{\om}) \, \Tr((P \sh b)(d\boldsymbol{\om})\,b_{k+1}),
\]
for all $b = (b_1,\ldots,b_k) \in \cS_2(H_2;H_1) \times \cdots \times \cS_2(H_{k+1};H_k)$ and $b_{k+1} \in \cS_2(H_1;H_{k+1})$ such that $b_1,\ldots,b_{k+1}$ all have rank at most one.
Now, write $m \coloneqq k+1$, $T \coloneqq (I^{\boldsymbol{P}}\varphi)[b]$, $b_j \coloneqq \la \cdot, h_j \ra_{H_{j+1}}k_j$ for $1 \leq j \leq m-1$, and $b_m \coloneqq \la \cdot , h_0\ra_{H_1}k_m$.
Then, by the calculation done in the proof of Theorem \ref{thm.MOIwelldef},
\[
\Tr\big(\big(I^{\boldsymbol{P}}\varphi\big)[b]\,b_m\big) = \Tr(T \circ (\la \cdot, h_0 \ra_{H_1}k_m)) = \la Tk_m,h_0 \ra_{H_1} = \int_{\Om}\varphi \, d\nu,
\]
where $\nu = P_{k_1 \otimes \cdots \otimes k_m, h_0 \otimes \cdots \otimes h_{m-1}} = (P_1)_{k_1,h_0} \otimes (P_2)_{k_2,h_1} \otimes \cdots \otimes (P_m)_{k_m,h_{m-1}}$.
But now, by definition of the vector measure $P \sh b$, if $G = G_1 \times \cdots \times G_{k+1} \in \sF$ with $G_1 \in \sF_1,\ldots,G_{k+1} \in \sF_{k+1}$, then
\begin{align*}
    \Tr((P \sh b)(G)\,b_{k+1}) & = \Tr(P_1(G_1)\,b_1\cdots P_{k+1}(G_{k+1})\,b_{k+1}) \\
    & =  \prod_{j=1}^m \big\la P_j(G_j)k_j, h_{j-1} \big\ra_{H_j} = \prod_{j=1}^m ( P_j)_{k_j,h_{j-1}}(G_j) \\
    & = ((P_1)_{k_1,h_0} \otimes (P_2)_{k_2,h_1} \otimes \cdots \otimes (P_m)_{k_m,h_{m-1}})(G) = \nu(G).
\end{align*}
(This is a special case of the calculation resulting in \eqref{eq.firstexp} and \eqref{eq.secondexp} from the proof of Theorem \ref{thm.MOIwelldef}.)
It follows that $\Tr((P \sh b)(\cdot)\,b_m) = \nu$ as complex measures on $(\Om,\sF)$.
This completes the proof.
\end{proof}

We end this section by relating this to Birman--Solomyak's original definition of DOIs, i.e., the case $k=1$.
Before doing so, however, we make an observation.
Redefine $H \coloneqq H_1$ and $K \coloneqq H_2$.
It is a standard fact that $H \hotimes K^* \cong \cS_2(K;H)$ isometrically via the bounded linear map determined by $h \otimes \ell \mapsto \ell(\cdot) h$.
This identification gives us a natural isometric isomorphism $\# \colon B(H \hotimes K^*) \to B(\cS_2(K;H))$ that is a homeomorphism with respect to all the usual topologies --- in particular, the WOT.
Viewing $B(H) \otimes B(K^*)$ as a subset of $B(H \hotimes K^*)$, one can show this map is the unique WOT-continuous linear extension of the linear map determined by
\[
B(H) \otimes B(K^*) \ni a \otimes b^{\operatorname{t}} \mapsto (c \mapsto acb) \in B(\cS_2(K;H)),
\]
where for $b \in B(K)$ the \textbf{transpose} $b^{\operatorname{t}} \in B(K^*)$ is defined by $\ell \mapsto \ell \circ b$, i.e., the adjoint of $b$ without identifying $K^*$ with $K$ via the Riesz representation theorem.
(What is being said here is that the operation $\#$ from Remark \ref{rem.hash} \textit{does} extend to the von Neumann algebra tensor product $B(H) \wotimes B(K^*) = B(H \hotimes K^*)$ when the codomain is taken to be $B(\cS_2(K;H))$.)
Noting that both the transpose $P_2^{\operatorname{t}} \colon \sF_2 \to B(K^*)$ and the composition $\tilde{P} \coloneqq \#(P_1 \otimes P_2^{\operatorname{t}}) \colon \sF_1 \otimes \sF_2 \to B(\cS_2(K;H))$ are still projection valued measures, we can therefore define, following Birman--Solomyak \cite{birmansolomyakDSOI1},
\[
T_{\varphi}^{P_1,P_2}(b) \coloneqq \#\Bigg(\int_{\Om_1 \times \Om_2} \varphi \, d(P_1 \otimes P_2^{\operatorname{t}})\Bigg)b = \tilde{P}(\varphi) b \in \cS_2(K;H),
\]
for all $\varphi \in L^{\infty}(\Om_1 \times \Om_1,P_1 \otimes P_2) = L^{\infty}(\Om_1 \times \Om_2,P_1 \otimes P_2^{\operatorname{t}}) = L^{\infty}(\Om_1 \times \Om_2, \tilde{P})$ and $b \in \cS_2(K;H)$.
One can show (e.g., starting with finite rank $b$ and then approximating in $\cS_2$) that
\[
T_{\varphi}^{P_1,P_2}(b) = \int_{\Om_1 \times \Om_2} \varphi \, d((P_1 \otimes P_2)\sh b),
\]
i.e., this agrees with Pavlov's definition.
\pagebreak

Now, Birman--Solomyak defined $T_{\varphi}^{P_1,P_2}(b)$ for $b \in B(K;H)$ as follows. Recall $B(H;K) \cong \cS_1(K;H)^*$ isometrically via the map $B \mapsto (A \mapsto \Tr(AB))$.
Thus $B(K;H)$ is isometrically conjugate isomorphic to $\cS_1(K;H)^*$ via the map $C \mapsto (A \mapsto \Tr(AC^*))$.
Therefore, if $T \colon \cS_1(K;H) \to \cS_1(K;H)$ is a bounded linear map, then we may speak of its adjoint $T^* \colon B(K;H) \to B(K;H)$, which is characterized by
\[
\Tr(T(A)\,C^*) = \Tr(A\,T^*(C)^*),
\]
for all $A \in \cS_1(K;H)$ and $C \in B(K;H)$.
Now, if $\varphi \in L^{\infty}(\Om_1 \times \Om_2, P_1 \otimes P_2)$ satisfies
\[
T_{\varphi}^{P_1,P_2}(\cS_1(K;H)) \subseteq \cS_1(K;H) \subseteq \cS_2(K;H)
\]
and if $T_{\varphi}^{P_1,P_2}|_{\cS_1(K;H)} \colon \cS_1(K;H) \to \cS_1(K;H)$ is bounded --- as is the case when $\varphi \in L^{\infty}(\Om_1,P_1) \iotimes L^{\infty}(\Om_2,P_2)$ by Theorem \ref{thm.pavlovMOIagree} and Proposition \ref{prop.MOISchattenestim} --- then it is easy to show that
\[
T_{\overline{\varphi}}^{P_1,P_2}(\cS_1(K;H)) \subseteq \cS_1(K;H) \; \text{ and } \; \big\|T_{\overline{\varphi}}^{P_1,P_2}\big\|_{B(\cS_1(K;H))} = \big\|T_{\varphi}^{P_1,P_2}\big\|_{B(\cS_1(K;H))} < \infty.
\]
In this situation, Birman--Solomyak define, for all $b \in B(K;H)$,
\[
T_{\varphi}^{P_1,P_2}(b) \coloneqq \big(T_{\overline{\varphi}}^{P_1,P_2}\big|_{\cS_1(K;H)}\big)^*(b) \in B(K;H).
\]

Now, let $\varphi \in L^{\infty}(\Om_1,P_1) \iotimes L^{\infty}(\Om_2,P_2)$.
By Corollary \ref{cor.uwcont}, if $b_1 \in B(K;H)$ and $b_2 \in \cS_1(K;H)$, then
\begin{align*}
    \Tr\Bigg(\int_{\Om_2}\int_{\Om_1}\varphi(\om_1,\om_2)\,P_1(d\om_1) \, b_1 \, P_2(d\om_2)\, b_2^*\Bigg) & = \Tr\Bigg(\int_{\Om_1}\int_{\Om_2}\varphi(\om_1,\om_2)\,P_2(d\om_2) \, b_2^* \, P_1(d\om_1) \,b_1\Bigg) \\
    & = \Tr\Bigg(b_1\Bigg(\int_{\Om_2}\int_{\Om_1}\overline{\varphi(\om_1,\om_2)}\,P_1(d\om_1) \, b_2 \, P_2(d\om_2)\Bigg)^* \Bigg).
\end{align*}
This says precisely that 
\[
\big(I^{P_1,P_2}\overline{\varphi}|_{\cS_1(K;H)}\big)^* = I^{P_1,P_2}\varphi.
\]
Since we already know our definition of MOIs agrees with that of Pavlov when they both apply, and therefore $(I^{P_1,P_2}\varphi)[b] = T_{\varphi}^{P_1,P_2}(b)$ when $b \in \cS_2(K;H)$, we obtain the following theorem.

\begin{thm}[Agreement $\,$with $\,$Birman--Solomyak$\,$ DOI]\label{thm.BSagree}
If $\varphi \,\in\, L^{\infty}(\Om_1,\,P_1) \,\iotimes\, L^{\infty}(\Om_2,\,P_2)$, then $I^{P_1,P_2}\varphi = T_{\varphi}^{P_1,P_2}$ on all of $B(K;H)$. \qed
\end{thm}

\appendix
\section{Proofs of auxiliary MOI results}\label{app.otherMOI}

\subsection{Proof of Theorem \ref{thm.tensprodPVM}}\label{sec.tensprodPVM}

In this section, we present the proof of Theorem \ref{thm.tensprodPVM}.
As this result is analogous to the construction of product measures, it should be no surprise that the proof goes through an extension theorem (Theorem \ref{thm.extendpremeas}) for projection valued measures.

\begin{defi}
Let $\mathscr{E}$ be a collection of subsets of a set $\Om$ such that $\emptyset,\Om \in \mathscr{E}$, $H$ be a complex Hilbert space, and $P^0 \colon \mathscr{E} \to B(H)$ be a function.
We call $P^0$
\begin{enumerate}[label=(\roman*),font=\normalfont,leftmargin=2\parindent]
    \item \textbf{projection valued} if $P^0(\Om) = \id_H = 1$ and $P^0(G)^2=P^0(G)=P^0(G)^*$ for all $G \in \mathscr{E}$;
    \item a \textbf{projection valued protomeasure} if $\mathscr{E}$ is an elementary family and $P^0$ is projection valued and countably additive (on $\mathscr{E}$) in the WOT; and
    \item a \textbf{projection valued premeasure} if $\mathscr{E}$ is an algebra and $P^0$ is projection valued and countably additive (on $\mathscr{E}$) in the WOT.
\end{enumerate}
\end{defi}
\begin{rem}
It can be shown (as in the proof of Theorem 1 in Section 5.1.1 of \cite{birmansolomyakBook}) that if $\mathscr{E}$ is a ring of sets and $P^0$ is projection valued and finitely additive, then $P^0(G_1 \cap G_2) = P^0(G_1) \, P^0(G_2)$ for all $G_1,G_2 \in \mathscr{E}$.
\end{rem}

As in classical measure theory, one can extend a protomeasure in a unique way to a measure.
\pagebreak

\begin{lem}\label{lem.extendprotomeas}
Let $\mathscr{E}$ be an elementary family of subsets of a set $\Om$, $H$ be a complex Hilbert space, and $P^{00} \colon \mathscr{E} \to B(H)$ be a projection valued protomeasure such that $P^{00}(G_1) \, P^{00}(G_2) = 0$ whenever $G_1,G_2 \in \mathscr{E}$ and $G_1 \cap G_2 = \emptyset$.
Then $P^{00}$ extends uniquely to a projection valued premeasure $P^0 \colon \alpha(\mathscr{E}) \to B(H)$, where $\alpha(\mathscr{E}) \subseteq 2^{\Om}$ is the algebra generated by $\mathscr{E}$.
\end{lem}
\begin{proof}[Sketch \hspace{-0.15mm}of\hspace{-0.15mm} proof]\hspace{-1mm}
Recall \hspace{-0.15mm}$\alpha(\mathscr{E})$\hspace{-0.15mm} is \hspace{-0.15mm}the\hspace{-0.15mm} set \hspace{-0.15mm}of\hspace{-0.15mm} finite \hspace{-0.15mm}disjoint\hspace{-0.15mm} unions \hspace{-0.15mm}of\hspace{-0.15mm} elements \hspace{-0.15mm}of\hspace{-0.15mm} $\mathscr{E}$.
\hspace{-0.15mm}Therefore,\hspace{-0.15mm} if $G_1,\ldots,G_m \hspace{-0.15mm}\in\hspace{-0.15mm} \mathscr{E}$ are pairwise disjoint and $G = \bigcup_{n=1}^m G_n \in \alpha(\mathscr{E})$, then we must have $P^0(G) \coloneqq \sum_{n=1}^m P^{00}(G_n)$.
By a standard argument, $P^0$ is well-defined and finitely additive because $P^{00}$ is finitely additive on $\mathscr{E}$, and $P^0$ is WOT-countably additive on $\alpha(\mathscr{E})$ because $P^{00}$ is WOT-countably additive on $\mathscr{E}$.
Finally, $P^0$ is projection valued because $P^0(G)$ is, by assumption on $P^{00}$, a sum of pairwise orthogonal projections. 
\end{proof}

\begin{thm}\label{thm.extendpremeas}
Let $\sA$ be an algebra of subsets of $\Om$, $H$ be a complex Hilbert space, and $P^0 \colon \mathscr{A} \to B(H)$ be a projection valued premeasure.
Then $P^0$ extends uniquely to a projection valued measure $P \colon \sigma(\sA) \to B(H)$, where $\sigma(\sA) \subseteq 2^{\Om}$ is the $\sigma$-algebra generated by $\sA$.
\end{thm}

For a proof, please see Section 5.2 (specifically, Theorems 3 and 4.(2)) of \cite{birmansolomyakBook}.
It proceeds as in classical measure theory, using a projection valued analog of Carath\'{e}odory's theorem, which concerns itself with the \textbf{projection valued outer measure}
\[
P^*(G) \coloneqq \inf\big\{P^0(G_1) : G_1 \supseteq G, \, G_1 \in \sA\big\}.
\]
In fact, the whole proof amounts to transferring the result of Carath\'{e}odory's theorem for the outer measures $\mu_{h,h}^*(G) \coloneqq \inf\big\{\big\la P^0(G_1)h,h \big\ra_H: G_1 \supseteq G, \, G_1 \in \sA\big\}$ to a result about $P^*$.

\begin{proof}[Proof of Theorem \ref{thm.tensprodPVM}]
Write $m \coloneqq k+1$, $\Om \coloneqq \Om_1 \times \cdots \times \Om_m$, and define
\[
\mathscr{E} \coloneqq \{G_1 \times \cdots \times G_m \subseteq \Om : G_1 \in \sF_1, \ldots, G_m \in \sF_m\}
\]
to be the set of measurable rectangles.
Now, define
\[
P^{00}(G_1 \times \cdots \times G_m) \coloneqq P_1(G_1) \otimes \cdots \otimes P_m(G_m) \in B(H_1 \hotimes \cdots \hotimes H_m),
\]
for all $G_1 \times \cdots \times G_m \in \mathscr{E}$.
Recall $\mathscr{E}$ is an elementary family.
We claim $P^{00}$ is a projection valued protomeasure such that $P^{00}(G \cap \tilde{G}) = P^{00}(G) \, P^{00}(\tilde{G})$ for $G,\tilde{G} \in \mathscr{E}$.
If so, then an appeal to Lemma \ref{lem.extendprotomeas} and Theorem \ref{thm.extendpremeas} completes the proof because $\sigma(\alpha(\mathscr{E})) = \sigma(\mathscr{E}) = \sF_1 \otimes \cdots \otimes \sF_m$.

If $G \coloneqq G_1 \times \cdots \times G_m, \tilde{G} \coloneqq \tilde{G}_1 \times \cdots \times \tilde{G}_m \in \mathscr{E}$, then
\begin{align*}
    P^{00}(G \cap \tilde{G}) & = P^{00}((G_1 \cap \tilde{G}_1) \times \cdots \times (G_m \cap \tilde{G}_m)) \\
    & = P_1(G_1 \cap \tilde{G}_1) \otimes \cdots \otimes P_m(G_m \cap \tilde{G}_m) \\
    & = (P_1(G_1) \, P_1(\tilde{G}_1)) \otimes \cdots \otimes (P_m(G_m) \, P_m(\tilde{G}_m)) \\
    & = (P_1(G_1) \otimes \cdots \otimes P_m(G_m))(P_1(\tilde{G}_1) \otimes \cdots \otimes P_m(\tilde{G}_m)) \\
    & = P^{00}(G) \, P^{00}(\tilde{G}).
\end{align*}
Also,
\begin{align*}
    P^{00}(G)^* & = (P_1(G_1) \otimes \cdots \otimes P_m(G_m))^* = P_1(G_1)^* \otimes \cdots \otimes P_m(G_m)^* \\
    & = P_1(G_1) \otimes \cdots \otimes P_m(G_m) = P^{00}(G).
\end{align*}
Since clearly $P^{00}(\emptyset) = 0$ and $P^{00}(\Om) = 1$, we only have WOT-countable additivity left to prove.

To this end, write $\la \cdot ,\cdot \ra \coloneqq \la \cdot, \cdot \ra_{H_1 \otimes \cdots \otimes H_m}$ for the tensor inner product.
By definition, we need to show that the assignment
\[
\mathscr{E} \ni G \mapsto P_{h,k}^{00}(G) \coloneqq \big\la P^{00}(G)h, k \big\ra \in \C
\]
is countably additive for all $h,k \in H_1 \hotimes \cdots \hotimes H_m$.
Choosing first pure tensors $h = h_1 \otimes \cdots \otimes h_m$, $k = k_1 \otimes \cdots \otimes k_m$ for $h_1,k_1 \in H_1,\ldots,h_m,k_m \in H_m$, we have
\begin{align*}
    P_{h,k}^{00}(G_1 \times \cdots \times G_m) & = \la (P_1(G_1) \otimes \cdots \otimes P_m(G_m))(h_1 \otimes \cdots \otimes h_m), k_1 \otimes \cdots \otimes k_m \ra \\
    & = \la P_1(G_1) h_1,k_1 \ra_{H_1}\cdots \la P_m(G_m) h_m,k_m \ra_{H_m} \\
    & = (P_1)_{h_1,k_1}(G_1) \cdots (P_m)_{h_m,k_m}(G_m) \\
    & = ((P_1)_{h_1,k_1}\otimes \cdots \otimes (P_m)_{h_m,k_m})(G).\displaybreak
\end{align*}
We conclude that $P^{00}_{h,k}$ is countably additive for pure tensors $h,k$ and therefore also for $h,k \in H_1 \otimes \cdots \otimes H_m$.
Now, let $(G_n)_{n \in \N} \in \mathscr{E}^{\N}$ be a pairwise disjoint sequence such that $\bigcup_{n \in \N}G_n \in \mathscr{E}$ and $h,k \in H_1 \hotimes \cdots \hotimes H_m$ be arbitrary.
First, we show $\sum_{n=1}^{\infty}\la P^{00}(G_n)h,k\ra$ is absolutely convergent.
Indeed, $(P^{00}(G_n))_{n \in \N}$ is a sequence of pairwise orthogonal projections, so Bessel's inequality implies
\[
\Bigg(\sum_{n=1}^{\infty}\big\|P^{00}(G_n)h\big\|^2\Bigg)^{\frac{1}{2}} \leq \|h\|.
\]
Therefore, by the Cauchy--Schwarz inequality (twice),
\begin{align*}
    \sum_{n=1}^{\infty}\big|\big\la P^{00}(G_n)h,k\big\ra\big| & = \sum_{n=1}^{\infty}\big|\big\la P^{00}(G_n)h,P^{00}(G_n)k\big\ra\big| \leq \sum_{n=1}^{\infty}\big\|P^{00}(G_n)h\big\|\,\big\|P^{00}(G_n)k\big\| \\
    & \leq \Bigg(\sum_{n=1}^{\infty}\big\|P^{00}(G_n)h\big\|^2 \Bigg)^{\frac{1}{2}}\Bigg(\sum_{n=1}^{\infty}\big\|P^{00}(G_n)k\big\|^2\Bigg)^{\frac{1}{2}}  \leq \|h\| \, \|k\|.
\end{align*}
Next, choose sequences $(h_j)_{j \in \N},(k_j)_{j \in \N}$ in $H_1 \otimes \cdots \otimes H_m$ such that $h_j \to h$ and $k_j \to k$ in $H_1 \hotimes \cdots \hotimes H_m$ as $j \to \infty$.
Then
\begin{align*}
    \Bigg|\big\la P^{00}(G)h_j,k_j\big\ra -  \sum_{n=1}^{\infty}\big\la P^{00}(G_n)h,k\big\ra\Bigg| & = \Bigg|\sum_{n=1}^{\infty}\big\la P^{00}(G_n)h_j,k_j\big\ra - \sum_{n=1}^{\infty}\big\la P^{00}(G_n)h,k\big\ra\Bigg| \\
    & = \Bigg|\sum_{n=1}^{\infty}\big(\big\la P^{00}(G_n)(h_j-h),k_j\big\ra+
\big\la P^{00}(G_n)h,k_j-k\big\ra\big)\Bigg| \\
    & \leq \|h_j-h\| \, \|k_j\| + \|h\| \, \|k_j-k\| \to 0
\end{align*}
as $j \to \infty$.
But we know $\la P^{00}(G)h_j,k_j\ra \to \la P^{00}(G)h,k\ra$ as $j \to \infty$ as well.
We therefore conclude that $\la P^{00}(G)h,k\ra = \sum_{n=1}^{\infty}\la P^{00}(G_n)h,k\ra$, as desired.
\end{proof}

\subsection{Proof of Theorem \ref{thm.Pavlov}}\label{sec.proofofPavlov}

In this section, we prove Theorem \ref{thm.Pavlov} using the approach of Pavlov and Birman--Solomyak from the beginning of \cite{pavlov} and the end of \cite{birmansolomyakTensProd}, respectively.
The construction is similar in spirit to that of the tensor product of projection valued measures.
One uses a Carath\'{e}odory-type theorem to extend a certain measure from the measurable rectangles to the whole product $\sigma$-algebra.
However, the difference is that building tensor products of projection valued measures requires only the (relatively easy) projection valued Carath\'{e}odory's theorem, while the construction in Theorem \ref{thm.Pavlov} requires (the much more difficult) Carath\'{e}odory's theorem for vector measures.
Therefore, before we launch into the proof of Theorem \ref{thm.Pavlov}, we must review --- following some of Chapter I of \cite{diesteluhl} --- some vector measure theory.

\begin{defi}
Let $\mathscr{E}$ be collection of subsets of a set $\Om$ that contains both $\emptyset$ and $\Om$, $V$ be a Banach space, $\mu \colon \mathscr{E} \to V$ be a function, and $(G_n)_{n \in \N} \in \mathscr{E}^{\N}$ be a pairwise disjoint sequence.
We call $\mu$
\begin{enumerate}[label=(\roman*),font=\normalfont,leftmargin=2\parindent]
    \item \textbf{finitely additive} if $\mu\big(\bigcup_{k=1}^n G_k \big) = \sum_{k=1}^n\mu(G_k)$ whenever $\bigcup_{k=1}^nG_k \in \mathscr{E}$;
    \item \textbf{weakly countably additive} (respectively, \textbf{countably additive}) if $\mu\big(\bigcup_{n \in \N} G_n \big) = \sum_{n=1}^{\infty}\mu(G_n)$ in the weak (respectively, norm) topology whenever $\bigcup_{n \in \N}G_n \in \mathscr{E}$;
    \item \textbf{strongly additive} if $\sum_{n=1}^{\infty}\mu(G_n)$ always exists in the norm topology;
    \item a \textbf{finitely additive vector measure} if $\mathscr{E}$ is an algebra and $\mu$ is finitely additive;
    and
    \item a \textbf{vector measure} if $\mathscr{E}$ is a $\sigma$-algebra and $\mu$ is countably additive.
\end{enumerate}
\end{defi}

Of course, if $\mu \colon \mathscr{E} \to V$ is a (finitely additive) vector measure, then $\ell \circ \mu$ is a (finitely additive) complex measure for all $\ell \in V^*$.
Also, we write $|\nu|$ for the total variation of a (finitely additive) complex measure $\nu$.

\begin{defi}
Let $\sA$ be an algebra of subsets of a set $\Om$, $V$ be a Banach space, and $\mu \colon \sA \to V$ be a finitely additive vector measure.
Write
\[
\|\mu\|(G) \coloneqq \sup\{|\ell \circ \mu|(G) : \ell \in V^*, \; \|\ell\|_{V^*} \leq 1\} \in [0,\infty],
\]
for all $G \in \sA$, and $\|\mu\|_{\mathrm{svar}} \coloneqq \|\mu\|(\Om)$.
The function $\|\mu\|$ is called the \textbf{semivariation} of $\mu$.
If $\|\mu\|_{\mathrm{svar}} < \infty$, then we say $\mu$ is of \textbf{bounded semivariation}.
\end{defi}

\begin{nota}
For a Banach space $V$ and a measurable space $(\Om,\sF)$, let $M(\Om,\sF;V)$ be the set of $V$-valued vector measures on $(\Om,\sF)$ of bounded semivariation.
\end{nota}

We now record some results about extending vector measures.
The following lemma has essentially the same proof as Lemma \ref{lem.extendprotomeas}, so we omit it.

\begin{lem}\label{lem.extendvecprotomeas}
Let $\mathscr{E}$ be an elementary family of subsets of a set $\Om$, $V$ be a Banach space, and $\mu^{00} \colon \mathscr{E} \to V$ be finitely (respectively, (weakly) countably) additive.
Then $\mu^{00}$ extends uniquely to a finitely (respectively, (weakly) countably) additive function $\mu^0 \colon \alpha(\mathscr{E}) \to V$, where $\alpha(\mathscr{E})$ is the algebra generated by $\mathscr{E}$.
\end{lem}

\begin{thm}[Carath\'{e}odory--Hahn--Kluv\'{a}nek]\label{thm.vecCaratheodory}
Let $\sA$ be an algebra of subsets of a set $\Om$, $V$ be a Banach space, and $\mu^0 \colon \sA \to V$ be weakly countably additive and of bounded semivariation.
If $\mu^0$ is strongly additive, then $\mu^0$ extends uniquely to a vector measure $\mu \colon \sigma(\sA) \to V$ with $\|\mu\|_{\mathrm{svar}} = \|\mu^0\|_{\mathrm{svar}}$.
\end{thm}

For a proof, please see Theorem 2 in Section I.5 of \cite{diesteluhl}.
Already, this takes some machinery (a ``control measure," the object in item (ii) of the aforementioned Theorem 2) to prove.
It often takes additional machinery to verify strong additivity of a given finitely additive measure;
for instance, here is a useful characterization of the situation in which you \textit{never} have to check strong additivity.

\begin{thm}[Diestel--Faires]
Write $c_0 \coloneqq \{(a_n)_{n \in \N} \in \ell^{\infty}(\N) : a_n \to 0$ as $n \to \infty\}$, and let $V$ be a Banach space.
The following are equivalent:
\begin{enumerate}[label=(\roman*),font=\normalfont,leftmargin=2\parindent]
    \item $V$ contains a copy of $c_0$ (i.e., there is a linear map $T \colon c_0 \to V$ and constants $\e,C > 0$ such that $\e\|a\|_{\ell^{\infty}(\N)} \leq \|Ta\|_V \leq C\|a\|_{\ell^{\infty}(\N)}$ for all $a \in c_0$); and
    \item there exists a finitely additive vector measure $\mu^0 \colon \sA \to V$, where $\sA$ is an algebra of sets, of bounded semivariation that is not strongly additive.
\end{enumerate}
\end{thm}

For a proof, please see Theorem 2 and the example that follows it in Section I.4 of \cite{diesteluhl}.

\begin{cor}\label{cor.strongaddcrit}
If $V$ is a weakly sequentially complete Banach space, then every finitely additive $V$-valued vector measure of bounded semivariation is strongly additive.
In particular, this is true for reflexive $V$.
\end{cor}
\begin{proof}[Sketch of proof]
By the Diestel--Faires theorem, it suffices to show $V$ cannot contain a copy of $c_0$.
But notice that any closed linear subspace of $V$ is weakly sequentially complete, so it suffices to show $c_0$ is not weakly sequentially complete.
\hspace{-0.15mm}To\hspace{-0.15mm} this \hspace{-0.15mm}end,\hspace{-0.15mm} let $e_n \hspace{-0.15mm}\in\hspace{-0.15mm} c_0$ be the $n^{\text{th}}$ standard basis vector and $s_n \hspace{-0.15mm}\coloneqq\hspace{-0.15mm} \sum_{j=1}^ne_j \hspace{-0.15mm}\in\hspace{-0.15mm} c_0$.
It is an easy exercise to show $(s_n)_{n \in \N}$ is weakly Cauchy but not weakly convergent.

It remains to prove that if $V$ is reflexive, then $V$ is weakly sequentially complete.
Indeed, let $(v_n)_{n \in \N}$ be a weakly Cauchy sequence and $\ev \colon V \to V^{**}$ be the natural embedding.
By the completeness of $\C$, if $\ell \in V^*$, then there exists $\eta(\ell) \in \C$ such that $\ell(v_n) \to \eta(\ell)$ as $n \to \infty$.
Clearly the map $\eta \colon V^* \to \C$ is linear.
By the principle of uniform boundedness, $\eta$ is bounded.
Since $V$ is reflexive, $\eta = \ev(v)$ for some $v \in V$, and by definition $v_n \to v$ weakly as $n \to \infty$.
\end{proof}

\begin{cor}\label{cor.vecCaratheodory}
Let $\sA$ be an algebra of subsets of a set $\Om$, $V$ be a Banach space, and $\mu^0 \colon \sA \to V$ be weakly countably additive and of bounded semivariation.
If $V$ is weakly sequentially complete, then $\mu^0$ extends uniquely to a vector measure $\mu \colon \sigma(\sA) \to V$ with $\|\mu\|_{\mathrm{svar}} = \|\mu^0\|_{\mathrm{svar}}$.
\end{cor}
\begin{proof}
By Corollary \ref{cor.strongaddcrit}, $\mu^0$ is strongly additive.
Then the Carath\'{e}odory--Hahn--Kluv\'{a}nek theorem gives the desired extension.
\end{proof}

We are now almost ready to prove Theorem \ref{thm.Pavlov}.
Before beginning in earnest, we take care of a combinatorial detail that arises in the proof.

\begin{lem}\label{lem.comb}
Let $\sA_j$ be an algebra of subsets of $\Om_j$, for $1 \leq j \leq k+1$, and
\[
\mathscr{E} \coloneqq \{G_1 \times \cdots \times G_{k+1} : G_1 \in \sA_1,\ldots,G_{k+1} \in \sA_{k+1}\}
\]
be the elementary family of \textbf{rectangles} in $\Om \coloneqq \Om_1 \times \cdots \times \Om_{k+1}$.
Suppose $\nu \colon \mathscr{E} \to \C$ is finitely additive.
If $R_1,\ldots,R_n \in \mathscr{E}$ is a partition of $\Om$ by rectangles and we write $[m] \coloneqq \{1,\ldots,m\}$, then there is a partition
\[
\big\{G^{\boldsymbol{\ell}} \coloneqq G_1^{\ell_1} \times \cdots \times G_{k+1}^{\ell_{k+1}} : \boldsymbol{\ell} = (\ell_1,\ldots,\ell_{k+1}) \in [n_1] \times \cdots \times [n_{k+1}] =: [\boldsymbol{n}] \big\}
\]
of $\Om$, where $G_j^1,\ldots,G_j^{n_j} \in \sA_j$ is a partition of $\Om_j$ for all $j \in [k+1]$, such that
\[
\sum_{m=1}^n|\nu(R_m)| \leq \sum_{\boldsymbol{\ell} \in [\boldsymbol{n}]} \big|\nu\big(G^{\boldsymbol{\ell}}\big)\big|.
\]
\end{lem}
\begin{proof}
The key observation is that if $\tilde{G}^1,\ldots,\tilde{G}^n \in \sA_1$ is a cover (\textit{not} necessarily a partition) of $\Om_1$, then there is a partition $G^1,\ldots,G^N \in \sA_1$ of $\Om_1$ such that for all $m \in [n]$, $\tilde{G}^m$ is a disjoint union of some of the $G$'s.
We prove this by induction on $n$, the case $n=1$ being trivial.
Assume the result for all sets, all algebras, and all covers of length less than $n$.
Then, given a cover $\tilde{G}^1,\ldots,\tilde{G}^n \in \sA_1$ of $\Om_1$, we get a partition $G_0^1,\ldots,G_0^{N_0} \in \sA_1$ of $\bigcup_{m=2}^n \tilde{G}^m$ with the property that $\tilde{G}^m$ is a disjoint union of $G_0$'s for all $m \in \{2,\ldots,n\}$.
Let $\mathscr{P} \coloneqq \{G_0^1,\ldots,G_0^{N_0}\}$.
Then the desired partition of $\Om_1$ is
\[
\{P \in \mathscr{P} : \tilde{G}^1 \cap P = \emptyset\} \cup \{\tilde{G}^1 \cap P : P \in \mathscr{P} \text{ and } \tilde{G}^1 \cap P \neq \emptyset\} \cup \Bigg\{\tilde{G}^1 \setminus \Bigg(\bigcup_{P \in \mathscr{P}}P\Bigg) \Bigg\}.
\]
Enumera\hspace{-0.15mm}ting \hspace{-0.4mm}t\hspace{-0.15mm}he\hspace{-0.4mm} (\hspace{-0.4mm}nonempty\hspace{-0.4mm}) \hspace{-0.4mm}sets\hspace{-0.4mm} in \hspace{-0.4mm}the\hspace{-0.4mm} above \hspace{-0.4mm}family\hspace{-0.4mm} $G^1\hspace{-0.6mm},\ldots,\hspace{-0.4mm}G^N \hspace{-0.55mm}\in\hspace{-0.55mm} \sA_1$, \hspace{-0.4mm}we\hspace{-0.4mm} are \hspace{-0.4mm}done\hspace{-0.4mm} proving \hspace{-0.4mm}this\hspace{-0.35mm} initial \hspace{-0.4mm}observation.

Now,\hspace{-0.2mm} writing \hspace{-0.2mm}$R_m \hspace{-0.2mm}= \hspace{-0.2mm}\tilde{G}_1^m \hspace{-0.2mm}\times\hspace{-0.2mm} \cdots \hspace{-0.2mm}\times\hspace{-0.2mm} \tilde{G}_{k+1}^m$\hspace{-0.2mm}, apply \hspace{-0.2mm}the\hspace{-0.2mm} observation from \hspace{-0.2mm}the\hspace{-0.2mm} previous \hspace{-0.2mm}paragraph\hspace{-0.2mm} to $\tilde{G}_j^1,\ldots,\tilde{G}_j^n\hspace{-0.2mm} \in\hspace{-0.2mm} \sA_j$ to obtain a partition $G_j^1,\ldots,G_j^{n_j} \in \sA_j$ of $\Om_j$ such that for all $m \in [n]$, $\tilde{G}_j^m$ is a disjoint union of some $G_j$'s.
By the finite additivity assumption, we then get
\[
\nu(R_m) = \sum_{\ell_1 : G_1^{\ell_1} \subseteq \tilde{G}_1^m} \cdots \sum_{\ell_{k+1} : G_{k+1}^{\ell_{k+1}} \subseteq \tilde{G}_{k+1}^m} \nu\big(G_1^{\ell_1} \times \cdots \times G_{k+1}^{\ell_{k+1}} \big).
\]
Because $R_1,\ldots,R_n$ are pairwise disjoint, if $(\ell_1,\ldots,\ell_{k+1}) \in [n_1] \times \cdots \times [n_{k+1}]$ is such that
\[
G_1^{\ell_1} \subseteq \tilde{G}_1^m,\ldots,G_{k+1}^{\ell_{k+1}} \subseteq \tilde{G}_{k+1}^m,
\]
then it cannot be that $G_1^{\ell_1} \subseteq \tilde{G}_1^{m_0},\ldots,G_{k+1}^{\ell_{k+1}} \subseteq \tilde{G}_{k+1}^{m_0}$ for some $m_0 \neq m$ unless $G_1^{\ell_1} \times \cdots \times G_{k+1}^{\ell_{k+1}}$ is empty --- in which case $\nu\big(G_1^{\ell_1} \times \cdots \times G_{k+1}^{\ell_{k+1}} \big) = 0$.
This ``no double counting" observation and the above identity imply
\begin{align*}
    \sum_{m=1}^n|\nu(R_m)| & \leq \sum_{m=1}^n\sum_{\ell_1 : G_1^{\ell_1} \subseteq \tilde{G}_1^m} \cdots \sum_{\ell_{k+1} : G_{k+1}^{\ell_{k+1}} \subseteq \tilde{G}_{k+1}^m} \big|\nu\big(G_1^{\ell_1} \times \cdots \times G_{k+1}^{\ell_{k+1}} \big)\big| \leq \sum_{\boldsymbol{\ell} \in [\boldsymbol{n}]} \big|\nu\big(G^{\boldsymbol{\ell}}\big)\big|,
\end{align*}
as desired.
\end{proof}

\begin{proof}[Proof of Theorem \ref{thm.Pavlov}]
Let
\begin{align*}
    b & = (b_1,\ldots,b_k) \in \cS_2(H_2;H_1) \times \cdots \times \cS_2(H_{k+1};H_k) \; \text{ and} \\
    \mathscr{E} & \coloneqq \{G_1 \times \cdots \times G_{k+1} : G_1 \in \sF_1, \ldots, G_{k+1} \in \sF_{k+1}\}.
\end{align*}
For $G_1\times \cdots \times G_{k+1} \in \mathscr{E}$, define
\begin{align*}
    \mu_b^{00}(G_1 \times \cdots \times G_{k+1}) & \coloneqq P_1(G_1)\,b_1\cdots P_k(G_k)\,b_k\,P_{k+1}(G_{k+1}) \in \cS_2(H_{k+1};H_1)\\
    & = (P_1(G_1) \otimes  \cdots  \otimes P_{k+1}(G_{k+1})) \sh [b_1 \otimes \cdots \otimes b_k]
\end{align*}
in the notation of Remark \ref{rem.hash} (of which we shall make some use).
We break up the proof into five steps.
\begin{enumerate}[label=Step \arabic*., leftmargin=3.36\parindent]
    \item Prove that $\mu_b^{00}$ is finitely additive and therefore, by Lemma \ref{lem.extendvecprotomeas}, extends to a finitely additive vector measure $\mu_b^0 \colon \alpha(\mathscr{E}) \to \cS_2(H_{k+1};H_1)$.\makeatletter\def\@currentlabel{Step 1}\makeatother\label{step1}
    \item Prove that for any partition $G_j^1,\ldots,G_j^{n_j} \in \sF_j$ of $\Om_j$ (for each $j \in \{1,\ldots,k+1\}$) and any $b_{k+1} \in \cS_2(H_1;H_{k+1})$, we have that\makeatletter\def\@currentlabel{Step 2}\makeatother\label{step2}
    \[
    \sum_{(\ell_1,\ldots,\ell_m) \in [n_1]\times \cdots \times [n_m]} \big|\Tr\big(\mu_b^{00}\big(G_1^{\ell_1} \times \cdots \times G_{k+1}^{\ell_{k+1}}\big)\,b_{k+1}\big)\big| \leq \|b_1\|_{\cS_2} \cdots \|b_{k+1}\|_{\cS_2}.
    \]
    \item Conclude $\big\|\mu_b^0\big\|_{\mathrm{svar}} \leq \|b_1\|_{\cS_2} \cdots \|b_k\|_{\cS_2}$.\makeatletter\def\@currentlabel{Step 3}\makeatother\label{step3}
    \item Prove that $\mu_b^{00}$ is weakly countably additive, which, again by Lemma \ref{lem.extendvecprotomeas}, means that $\mu_b^0$ is also weakly countably additive.
    Then apply Corollary \ref{cor.vecCaratheodory} to get $P \sh b$ from $\mu_b^0$.\makeatletter\def\@currentlabel{Step 4}\makeatother\label{step4}
    \item Prove $P \sh b \ll P$.\makeatletter\def\@currentlabel{Step 5}\makeatother\label{step5}
\end{enumerate}
Let us begin.
We shall use $(\Om,\sF,H,P)$ as shorthand for the (tensor) product $(\Om_1 \times \cdots \times \Om_{k+1},\sF_1 \otimes \cdots \otimes \sF_{k+1},$

\noindent $H_1 \hotimes \cdots \hotimes H_{k+1}, P_1 \otimes \cdots \otimes P_{k+1})$.

\ref{step1}.
There are a number of direct ways to see $\mu_b^{00}$ is finitely additive.
We provide a cute proof using Theorem \ref{thm.tensprodPVM} and the $\#$ operation from Remark \ref{rem.hash}.
By definition of $P$, if $G \in \mathscr{E}$, then we have $P(G) \in B(H_1) \otimes \cdots \otimes B(H_{k+1}) \subseteq B(H)$.
By definition, $\mu_b^{00}(G) = P(G)\sh [b_1 \otimes \cdots \otimes b_k]$ for $G \in \mathscr{E}$.
But $P$ is finitely additive on $\mathscr{E}$ because $P \colon \sigma(\mathscr{E})= \sF \to B(H)$ is a projection valued measure.
Thus $\mu_b^{00}$ is finitely additive by the linearity of $\#$.
Moreover, the finitely additive extension $\mu_b^0 \colon \alpha(\mathscr{E}) \to \cS_2(H_{k+1};H_1)$ is also given by the formula
\[
\mu_b^0(G) = P(G) \sh [b_1 \otimes \cdots \otimes b_k]
\]
for $G \in \alpha(\mathscr{E})$.
This formula makes sense because $\alpha(\mathscr{E})$ is the set of \textit{finite} disjoint unions of elements of $\mathscr{E}$, so $P(G)$ is a \textit{finite} sum of pure tensors and thus lies in $B(H_1) \otimes \cdots \otimes B(H_{k+1})$ when $G \in \alpha(\mathscr{E})$.

\ref{step2}.
Let
\[
\Delta \coloneqq \big\{G^{\boldsymbol{\ell}} \coloneqq G_1^{\ell_1} \times \cdots \times G_{k+1}^{\ell_{k+1}} : \boldsymbol{\ell} = (\ell_1,\ldots,\ell_{k+1}) \in [n_1] \times \cdots \times [n_{k+1}] =: [\boldsymbol{n}] \big\}
\]
be the partition of $\Om = \Om_1 \times \cdots \times \Om_{k+1}$ obtained from the fixed partitions of $\Om_1,\ldots,\Om_{k+1}$.
For ease of notation, write
\[
T^{\boldsymbol{\ell}} \coloneqq \Tr\big(\mu_b^{00}\big(G^{\boldsymbol{\ell}}\big)\,b_{k+1}\big) \; \text{ and } \; |\Delta| \coloneqq \sum_{\boldsymbol{\ell} \in [\boldsymbol{n}]} \big|T^{\boldsymbol{\ell}}\big|.
\]
Then the estimate $|\Delta| \leq \|b_1\|_{\cS_2} \cdots \|b_{k+1}\|_{\cS_2}$ is the goal of this step.

To begin, note that if $\boldsymbol{\ell} \in [\boldsymbol{n}]$, then
\begin{align*}
    T^{\boldsymbol{\ell}} & = \Tr\Big(P_1\big(G_1^{\ell_1}\big) \, b_1 \cdots P_k\big(G_k^{\ell_k}\big) \, b_k \,P_{k+1}\big(G_{k+1}^{\ell_{k+1}}\big) \, b_{k+1} \Big) \\
    & = \Tr\Big(P_1\big(G_1^{\ell_1}\big)^2 b_1 \cdots P_k\big(G_k^{\ell_k}\big)^2 b_k P_{k+1}\big(G_{k+1}^{\ell_{k+1}}\big)^2 b_{k+1} \Big) \\
    & = \Tr\Big(\Big[P_1\big(G_1^{\ell_1}\big) \, b_1 \, P_2\big(G_2^{\ell_2}\big)\Big] \cdots \Big[P_k\big(G_k^{\ell_k}\big) \, b_k\, P_{k+1}\big(G_{k+1}^{\ell_{k+1}}\big)\Big]\Big[ P_{k+1}\big(G_{k+1}^{\ell_{k+1}}\big)\,b_{k+1} \, P_1\big(G_1^{\ell_1}\big)\Big]\Big).
\end{align*}
If $j \in [k+1]$, $\ell_j \leq n_j$, and $\ell_{j+1} \leq n_{j+1}$ (adding mod $k+1$), then we define $\Pi_j^{\ell_j,\ell_{j+1}} \in B(\cS_2(H_{j+1};H_j))$ by
\[
\Pi_j^{\ell_j,\ell_{j+1}}(c) \coloneqq P_j\big(G_j^{\ell_j}\big) \, c \, P_{j+1}\big(G_{j+1}^{\ell_{j+1}}\big) = \big( P_j\big(G_j^{\ell_j}\big) \otimes P_{j+1} \big(G_{j+1}^{\ell_{j+1}}\big)\big)\sh c,
\]
for $c \in \cS_2(H_{j+1};H_j)$.
Then $T^{\boldsymbol{\ell}} = \Tr\big(\Pi_1^{\ell_1,\ell_2}(b_1) \cdots \Pi_k^{\ell_k,\ell_{k+1}}(b_k)\,\Pi_{k+1}^{\ell_{k+1},\ell_1}(b_{k+1})\big)$, so
\[
\big|T^{\boldsymbol{\ell}}\big| \leq \big\|\Pi_1^{\ell_1,\ell_2}(b_1) \cdots \Pi_k^{\ell_k,\ell_{k+1}}(b_k)\,\Pi_{k+1}^{\ell_{k+1},\ell_1}(b_{k+1})\big\|_{\cS_1} \leq \prod_{j=1}^{k+1} \big\|\Pi_j^{\ell_j,\ell_{j+1}}(b_j)\big\|_{\cS_2}
\]
(remembering to reduce mod $k+1$), for all $\boldsymbol{\ell} \in [\boldsymbol{n}]$.

Next, since
\[
\big\{G_j^{\ell_j} \times G_{j+1}^{\ell_{j+1}} : \ell_j \in [n_j],\,\ell_{j+1} \in [n_{j+1}]\big\}
\]
is a partition of $\Om_j \times \Om_{j+1}$ by rectangles, it is easy to see
\[
\big\{\Pi_j^{\ell_j,\ell_{j+1}} : \ell_j \in [n_j],\,\ell_{j+1} \in [n_{j+1}]\big\}\pagebreak
\]
is a collection of mutually orthogonal projections in $B(\cS_2(H_{j+1};H_j))$ such that
\begin{align*}
    \sum_{\ell_j=1}^{n_j}\sum_{\ell_{j+1} = 1}^{n_{j+1}} \Pi_j^{\ell_j,\ell_{j+1}}(c) & =\Bigg(\sum_{\ell_j=1}^{n_j}\sum_{\ell_{j+1} = 1}^{n_{j+1}}P_j\big(G_j^{\ell_j}\big) \otimes P_{j+1} \big(G_{j+1}^{\ell_{j+1}}\big) \Bigg)\sh c \\
    & = \Bigg(\Bigg(\sum_{\ell_j=1}^{n_j}P_j\big(G_j^{\ell_j}\big)\Bigg) \otimes \Bigg( \sum_{\ell_{j+1} = 1}^{n_{j+1}}P_{j+1} \big(G_{j+1}^{\ell_{j+1}}\big)\Bigg)\Bigg) \sh c \\
    & = (P_j(\Om_j) \otimes P_{j+1}(\Om_{j+1}))\sh c = c,
\end{align*}
so that whenever $c \in \cS_2(H_{j+1};H_j)$, we have
\[
\sum_{\ell_j=1}^{n_j}\sum_{\ell_{j+1} = 1}^{n_{j+1}} \big\|\Pi_j^{\ell_j,\ell_{j+1}}(c)\big\|_{\cS_2}^2 = \|c\|_{\cS_2}^2.
\]
Therefore, if $k+1$ is even, then by the Cauchy--Schwarz inequality,
\begin{align*}
    |\Delta| & \leq \sum_{\boldsymbol{\ell} \in [\boldsymbol{n}]}\prod_{j=1}^{k+1} \big\|\Pi_j^{\ell_j,\ell_{j+1}}(b_j)\big\|_{\cS_2} \\
    & = \sum_{\boldsymbol{\ell} \in [\boldsymbol{n}]} \prod_{p=1}^{(k+1)/2} \big\|\Pi_{2p-1}^{\ell_{2p-1},\ell_{2p}}(b_{2p-1})\big\|_{\cS_2}\prod_{q=1}^{(k+1)/2} \big\|\Pi_{2q}^{\ell_{2q},\ell_{2q+1}}(b_{2q})\big\|_{\cS_2} \\
    & \leq \Bigg(\sum_{\boldsymbol{\ell} \in [\boldsymbol{n}]} \prod_{p=1}^{(k+1)/2} \big\|\Pi_{2p-1}^{\ell_{2p-1},\ell_{2p}}(b_{2p-1})\big\|_{\cS_2}^2 \Bigg)^{\frac{1}{2}}\Bigg(\sum_{\boldsymbol{\ell} \in [\boldsymbol{n}]}\prod_{q=1}^{(k+1)/2} \big\|\Pi_{2q}^{\ell_{2q},\ell_{2q+1}}(b_{2q})\big\|_{\cS_2}^2\Bigg)^{\frac{1}{2}} \\
    & = \Bigg(\prod_{p=1}^{(k+1)/2}\|b_{2p-1}\|_{\cS_2}^2\Bigg)^{\frac{1}{2}}\Bigg( \prod_{q=1}^{(k+1)/2}\|b_{2q}\|_{\cS_2}^2\Bigg)^{\frac{1}{2}} = \|b_1\|_{\cS_2} \cdots \|b_{k+1}\|_{\cS_2},
\end{align*}
as desired.

If $k+1$ is odd, then we must estimate the expression in a slightly different way.
Just as we write\vspace{-0.25mm}
\[
\boldsymbol{\ell} = (\ell_1,\ldots,\ell_{k+1}) \in [n_1] \times \cdots \times [n_{k+1}] = [\boldsymbol{n}],\vspace{-0.25mm}
\]
we shall use the shorthand\vspace{-0.25mm}
\[
\boldsymbol{\tilde{\ell}} = (\ell_2,\ldots,\ell_k) \in [n_2] \times \cdots \times [n_k] = [\boldsymbol{\tilde{n}}].\vspace{-0.25mm}
\]
By the Cauchy--Schwarz inequality and above,
\begin{align*}
    |\Delta| & \leq \sum_{\boldsymbol{\ell} \in [\boldsymbol{n}]}\prod_{j=1}^{k+1} \big\|\Pi_j^{\ell_j,\ell_{j+1}}(b_j)\big\|_{\cS_2} = \sum_{\ell_1=1}^{n_1}\sum_{\ell_{k+1} = 1}^{n_{k+1}}\big\|\Pi_{k+1}^{\ell_{k+1},\ell_1}(b_{k+1})\big\|_{\cS_2}\sum_{\boldsymbol{\tilde{\ell}} \in [\boldsymbol{\tilde{n}}]} \prod_{j=1}^k \big\|\Pi_j^{\ell_j,\ell_{j+1}}(b_j)\big\|_{\cS_2} \\
    & \leq \Bigg(\sum_{\ell_1=1}^{n_1}\sum_{\ell_{k+1} = 1}^{n_{k+1}}\big\|\Pi_{k+1}^{\ell_{k+1},\ell_1}(b_{k+1})\big\|_{\cS_2}^2\Bigg)^{\frac{1}{2}}\Bigg(  \sum_{\ell_1=1}^{n_1}\sum_{\ell_{k+1} = 1}^{n_{k+1}}\Bigg(\sum_{\boldsymbol{\tilde{\ell}} \in [\boldsymbol{\tilde{n}}]} \prod_{j=1}^k \big\|\Pi_j^{\ell_j,\ell_{j+1}}(b_j)\big\|_{\cS_2}\Bigg)^2 \Bigg)^{\frac{1}{2}} \\
    & = \|b_{k+1}\|_{\cS_2}\Bigg(  \sum_{\ell_1=1}^{n_1}\sum_{\ell_{k+1} = 1}^{n_{k+1}}\Bigg(\underbrace{\sum_{\boldsymbol{\tilde{\ell}} \in [\boldsymbol{\tilde{n}}]} \prod_{j=1}^k \big\|\Pi_j^{\ell_j,\ell_{j+1}}(b_j)\big\|_{\cS_2}}_{=: \, S^{\ell_1,\ell_{k+1}}}\Bigg)^2 \Bigg)^{\frac{1}{2}}.
\end{align*}
Since $k\hspace{-0.3mm}+\hspace{-0.3mm}1$ is odd, \hspace{-0.3mm}$k$\hspace{-0.3mm} is even.
We \hspace{-0.3mm}can\hspace{-0.3mm} therefore \hspace{-0.3mm}estimate\hspace{-0.3mm} as \hspace{-0.3mm}in\hspace{-0.3mm} the \hspace{-0.3mm}even\hspace{-0.3mm} case.
If $\ell_1 \hspace{-0.3mm}\in\hspace{-0.3mm} [n_1\hspace{-0.3mm}]$ and $\ell_{k+1} \hspace{-0.3mm}\in\hspace{-0.3mm} [n_{k+1}\hspace{-0.3mm}]$, then
\begin{align*}
    S^{\ell_1,\ell_{k+1}} & = \sum_{\boldsymbol{\tilde{\ell}} \in [\boldsymbol{\tilde{n}}]} \prod_{p=1}^{k/2} \big\|\Pi_{2p-1}^{\ell_{2p-1},\ell_{2p}}(b_{2p-1})\big\|_{\cS_2}\prod_{q=1}^{k/2} \big\|\Pi_{2q}^{\ell_{2q},\ell_{2q+1}}(b_{2q})\big\|_{\cS_2} \\
    & \leq \Bigg(\sum_{\boldsymbol{\tilde{\ell}} \in [\boldsymbol{\tilde{n}}]} \prod_{p=1}^{k/2} \big\|\Pi_{2p-1}^{\ell_{2p-1},\ell_{2p}}(b_{2p-1})\big\|_{\cS_2}^2 \Bigg)^{\frac{1}{2}}\Bigg(\sum_{\boldsymbol{\tilde{\ell}} \in [\boldsymbol{\tilde{n}}]} \prod_{q=1}^{k/2} \big\|\Pi_{2q}^{\ell_{2q},\ell_{2q+1}}(b_{2q})\big\|_{\cS_2}^2\Bigg)^{\frac{1}{2}}.\displaybreak
\end{align*}
Also,
\begin{align*}
    \sum_{\boldsymbol{\tilde{\ell}} \in [\boldsymbol{\tilde{n}}]} \prod_{p=1}^{k/2} \big\|\Pi_{2p-1}^{\ell_{2p-1},\ell_{2p}}(b_{2p-1})\big\|_{\cS_2}^2 & = \sum_{\ell_2=1}^{n_2}\big\|\Pi_1^{\ell_1,\ell_2}(b_1)\big\|_{\cS_2}^2\prod_{p=2}^{k/2}\|b_{2p-1}\|_{\cS_2}^2 \; \text{ and}\\
    \sum_{\boldsymbol{\tilde{\ell}} \in [\boldsymbol{\tilde{n}}]} \prod_{q=1}^{k/2} \big\|\Pi_{2q}^{\ell_{2q},\ell_{2q+1}}(b_{2q})\big\|_{\cS_2}^2 & = \prod_{q=1}^{(k-1)/2}\|b_{2q}\|_{\cS_2}\sum_{\ell_k=1}^{n_k}\big\|\Pi_k^{\ell_k,\ell_{k+1}}(b_k)\big\|_{\cS_2}^2.
\end{align*}
Therefore,
\[
S^{\ell_1,\ell_{k+1}} \leq \|b_2\|_{\cS_2} \cdots \|b_{k-1}\|_{\cS_2}\Bigg( \sum_{\ell_2=1}^{n_2}\big\|\Pi_1^{\ell_1,\ell_2}(b_1)\big\|_{\cS_2}^2 \Bigg)^{\frac{1}{2}}\Bigg( \sum_{\ell_k=1}^{n_k}\big\|\Pi_k^{\ell_k,\ell_{k+1}}(b_k)\big\|_{\cS_2}^2 \Bigg)^{\frac{1}{2}},
\]
whence it follows that
\begin{align*}
    |\Delta| & \leq \|b_{k+1}\|_{\cS_2}\prod_{j=2}^{k-1}\|b_j\|_{\cS_2}\Bigg(\sum_{\ell_1=1}^{n_1}\sum_{\ell_2=1}^{n_2}\big\|\Pi_1^{\ell_1,\ell_2}(b_1)\big\|_{\cS_2}^2\sum_{\ell_k=1}^{n_k}\sum_{\ell_{k+1} = 1}^{n_{k+1}}\big\|\Pi_k^{\ell_k,\ell_{k+1}}(b_k)\big\|_{\cS_2}^2 \Bigg)^{\frac{1}{2}} \\
    & = \|b_{k+1}\|_{\cS_2}\Bigg(\prod_{j=2}^{k-1}\|b_j\|_{\cS_2}\Bigg)\|b_1\|_{\cS_2}\|b_k\|_{\cS_2} = \|b_1\|_{\cS_2} \cdots \|b_{k+1}\|_{\cS_2},
\end{align*}
as desired.
This completes \ref{step2}.

\ref{step3}.
By the Riesz representation theorem and the definition of the inner product on $\cS_2(H_{k+1};H_1)$, for all $\ell \in \cS_2(H_{k+1};H_1)^*$, there exists unique $B \in \cS_2(H_{k+1};H_1)$ such that $\ell(A) = \la A,B \ra_{\cS_2} = \Tr(B^*A) = \Tr(AB^*)$ for all $A \in \cS_2(H_{k+1};H_1)$.
Writing $b_{k+1} \coloneqq B^* \in \cS_2(H_1;H_{k+1})$ and
\[
\nu_{b,b_{k+1}}(G) = \nu_{b_1,\ldots,b_{k+1}}(G) \coloneqq \Tr\big(\mu_b^0\big(G\big)\,b_{k+1}\big)
\]
for $G \in \alpha(\mathscr{E})$, this tells us the goal of this step is to prove $\|\nu_{b,b_{k+1}}\|_{\text{var}} = |\nu_{b,b_{k+1}}|(\Om) \leq \|b_1\|_{\cS_2} \cdots \|b_{k+1}\|_{\cS_2}$.

To begin, we make the simple observation that if $\nu \colon \sA \to \C$ is a finitely additive complex measure and $\mathscr{E}_0 \subseteq \sA$ is an elementary family generating $\sA$ as an algebra, then
\[
|\nu|(G) = \sup\Bigg\{\sum_{m=1}^n |\nu(R_m)| : R_1,\ldots,R_n \in \mathscr{E}_0 \text{ is a partition of } G\Bigg\}
\]
for all $G \in \alpha(\mathscr{E})$.
Applying this observation to our case, we have
\[
\big|\nu_{b_1,\ldots,b_{k+1}}\big|(G) = \sup\Bigg\{\sum_{m=1}^n \big|\Tr\big(\mu_b^{00}\big(R_m\big)\,b_{k+1}\big)\big| : R_1,\ldots,R_n \in \mathscr{E} \text{ is a partition of } G\Bigg\},
\]
for all $G \in \alpha(\mathscr{E})$. Therefore, by Lemma \ref{lem.comb}, we have
\[
\big\|\nu_{b_1,\ldots,b_{k+1}}\big\|_{\text{var}} = \sup\Bigg\{\sum_{\boldsymbol{\ell} \in [\boldsymbol{n}]} \big|\Tr\big(\mu_b^{00}\big(G^{\boldsymbol{\ell}}\big)\,b_{k+1}\big)\big| : \Delta = \big\{G^{\boldsymbol{\ell}} : \boldsymbol{\ell} \in [\boldsymbol{n}]\big\} \text{ as in \ref{step2}} \Bigg\}.
\]
It then follows from \ref{step2} that $\|\nu_{b_1,\ldots,b_{k+1}}\|_{\text{var}} \leq \|b_1\|_{\cS_2} \cdots \|b_{k+1}\|_{\cS_2}$.
This completes \ref{step3}.

\ref{step4}.
According to the comments at the beginning of and notation in \ref{step3}, the goal of this step is to prove that $\nu_{b_1,\ldots,b_{k+1}}$ is countably additive for all $b_j \in \cS_2(H_{j+1};H_j)$, $j \in [k]$, and $b_{k+1} \in \cS_2(H_1;H_{k+1})$.
As mentioned in the outline of the proof, Lemma \ref{lem.extendvecprotomeas} tells us we only need to check countable additivity of $\nu_{b_1,\ldots,b_{k+1}}$ on $\mathscr{E}$.
From now on, write $m \coloneqq k+1$.

First, suppose $b_j = \la \cdot , h_j \ra_{H_{j+1}} k_j$, where $k_j \in H_j$ and $h_j \in H_{j+1}$, for all $j \in [m-1]$, and $b_m = \la \cdot, h_0 \ra_{H_1}k_m$, where $h_0 \in H_1$ and $k_m \in H_m$.
If $G = G_1 \times \cdots \times G_m \in \mathscr{E}$, then
\[
\mu_b^{00}(G)\,b_m = \Bigg(\prod_{j=2}^m \la P_j(G_j)k_j,h_{j-1} \ra_{H_j} \Bigg) \la \cdot, h_0 \ra_{H_1} P_1(G_1)k_1,
\]
so that
\begin{align*}
    \nu_{b_1,\ldots,b_m}(G) & = \Tr\big(\mu_b^{00}\big(G\big)\,b_m\big) = \la P_1(G_1)k_1,h_0 \ra_{H_1}\prod_{j=2}^m \la P_j(G_j)k_j,h_{j-1} \ra_{H_j} \\
    & = ((P_1)_{k_1,h_0} \otimes (P_2)_{k_2,h_1} \otimes \cdots \otimes (P_m)_{k_m,h_{m-1}})(G) \\
    & = P_{k_1 \otimes \cdots \otimes k_m, h_0 \otimes \cdots \otimes h_{m-1}}(G). \numberthis\label{eq.nuexp}
\end{align*}
Since this formula is the restriction to $\mathscr{E} \subseteq \sF$ of a complex measure, we get that $\nu_{b_1,\ldots,b_m}$ is countably additive.
Since $\nu_{b_1,\ldots,b_m}$ is clearly $m$-linear in $(b_1,\ldots,b_m)$, we then conclude that $\nu_{b_1,\ldots,b_m}$ is countably additive for all finite rank operators $b_1 \in \cS_2(H_2;H_1),\ldots,b_{m-1} \in \cS_2(H_m;H_{m-1})$ and $b_m \in \cS_2(H_1;H_m)$.
To finish this step, we approximate arbitrary $b$'s by finite rank ones.

Let $b_1 \in \cS_2(H_2;H_1),\ldots,b_{m-1} \in \cS_2(H_m;H_{m-1})$ and $b_m \in \cS_2(H_1;H_m)$ be arbitrary.
If $(G_p)_{p \in \N} \in \mathscr{E}^{\N}$ is a pairwise disjoint sequence with $G \coloneqq \bigcup_{p \in \N}G_p \in \mathscr{E}$, then we must show
\[
\delta_N \coloneqq \Bigg|\nu_{b_1,\ldots,b_m}(G) - \sum_{p=1}^N\nu_{b_1,\ldots,b_m}(G_p)\Bigg| \to 0
\]
as $N \to \infty$.
To this end, let $(b_1^n)_{n \in \N}, \ldots, (b_m^n)_{n \in \N}$ be sequences of finite rank operators such that $b_j^n \to b_j$ in $\cS_2$ as $n \to \infty$ for all $j \in [m]$.
Then, by the previous paragraph,
\begin{align*}
    \delta_N & = \Bigg|\nu_{b_1,\ldots,b_m}(G) - \nu_{b_1^n,\ldots,b_m^n}(G) + \sum_{p=1}^{\infty}\nu_{b_1^n,\ldots,b_m^n}(G_p) - \sum_{p=1}^N\nu_{b_1,\ldots,b_m}(G_p)\Bigg|\\
    & \leq |\nu_{b_1,\ldots,b_m}(G) - \nu_{b_1^n,\ldots,b_m^n}(G)| + \sum_{p=1}^N|\nu_{b_1,\ldots,b_m}(G_p) - \nu_{b_1^n,\ldots,b_m^n}(G_p)| +\sum_{p > N}|\nu_{b_1^n,\ldots,b_m^n}(G_p)|
\end{align*}
for all $n,N \in \N$, where the last term --- for fixed $n \in \N$ --- goes to zero as $N \to \infty$.
But now, notice that $m$-linearity gives us that
\[
\nu_{b_1,\ldots,b_m} - \nu_{b_1^n,\ldots,b_m^n} = \sum_{j=1}^m\nu_{b_1^n,\ldots,b_{j-1}^n,b_j-b_j^n,b_{j+1},\ldots,b_m} \numberthis\label{eq.multilin}
\]
for all $n \in \N$.
This observation and \ref{step3} then imply
\begin{align*}
    \limsup_{N \to \infty}\delta_N & \leq \sum_{j=1}^m\Bigg(\big|\nu_{b_1^n,\ldots,b_{j-1}^n,b_j-b_j^n,b_{j+1},\ldots,b_m}(G)\big| +\sum_{p=1}^{\infty}\big|\nu_{b_1^n,\ldots,b_{j-1}^n,b_j-b_j^n,b_{j+1},\ldots,b_m}(G_p)\big| \Bigg)\\
    & \leq 2\sum_{j=1}^m\big|\nu_{b_1^n,\ldots,b_{j-1}^n,b_j-b_j^n,b_{j+1},\ldots,b_m}\big|(G) \leq 2\sum_{j=1}^m\|\nu_{b_1^n,\ldots,b_{j-1}^n,b_j-b_j^n,b_{j+1},\ldots,b_m}\|_{\mathrm{var}} \\
    & \leq 2\sum_{j=1}^m\|b_j^n\|_{\cS_2} \cdots \|b_{j-1}^n\|_{\cS_2} \|b_j-b_j^n\|_{\cS_2}\|b_{j+1}\|_{\cS_2} \cdots \|b_m\|_{\cS_2} \to 0
\end{align*}
as $n \to \infty$.
We conclude $\lim_{N \to \infty} \delta_N = 0$.
Thus $\mu_b^0$ is weakly countably additive.

Since $\cS_2(H_{k+1};H_1)$ is a Hilbert space and therefore reflexive, \ref{step3}, what we just proved, and Corollary \ref{cor.vecCaratheodory} imply $\mu_b^0$ extends uniquely to a $\cS_2(H_{k+1};H_1)$-valued vector measure $\mu_b = P \sh b$ on $\sigma(\mathscr{E}) = \sF$ with $\|P \sh b\|_{\mathrm{svar}} = \|\mu_b^0\|_{\mathrm{svar}} \leq \|b_1\|_{\cS_2} \cdots \|b_k\|_{\cS_2}$.
This completes \ref{step4} and the construction of $P \sh b$.

\ref{step5}.
We use the approximation argument from \ref{step4}.
Suppose $G \in \sF$ is such that $P(G) = 0$ (which implies $P(\tilde{G}) = 0$ when $\sF \ni \tilde{G} \subseteq G$ because $P$ is a projection valued measure).
If $\sF \ni \tilde{G} \subseteq G$ and $b_1,\ldots,b_k$ and $b_{k+1} \in \cS_2(H_1;H_{k+1})$ have rank one, then \eqref{eq.nuexp} implies that $\Tr((P\sh b)(\tilde{G}) \, b_{k+1}) = 0$.
By multilinearity, this implies $\Tr((P\sh b)(\tilde{G}) \, b_{k+1}) = 0$ for finite rank $b_1,\ldots,b_{k+1}$.
Now, approximating in $\cS_2$ arbitrary $b_1,\ldots,b_{k+1}$ by finite rank operators gives, using \eqref{eq.multilin} and the semivariation bound, that $\Tr((P\sh b)(\tilde{G}) \, b_{k+1}) = 0$.
We conclude $(P\sh b)(\tilde{G}) = 0$, as claimed.
This completes the proof.
\end{proof}

\begin{ack}
\phantomsection
\addcontentsline{toc}{section}{Acknowledgements}
I acknowledge support from NSF grant DGE 2038238 and partial support from NSF grants DMS 1253402 and DMS 1800733. Some of the arguments from this paper appear (less cleanly) in the first version of my paper \cite{nikitopoulosNCk} on noncommutative $C^k$ functions. I am grateful to Edward McDonald for informing me that these arguments are of independent interest and for encouraging me to publish them on their own. I am also indebted to Bruce Driver for helpful conversations and guidance about writing.
\end{ack}

\end{document}